\documentclass{amsart}

\usepackage{latexsym,amsfonts,amssymb,exscale,enumerate,comment}
\usepackage{amsmath,amsthm,amscd}
\usepackage{verbatim}
\usepackage{kbordermatrix}

\usepackage{fullpage}
\usepackage{url}
\usepackage[bookmarks=true,%
    colorlinks=true,%
    linkcolor=blue,%
    citecolor=blue,%
    filecolor=blue,%
    menucolor=blue,%
    urlcolor=blue,%
    breaklinks=true,
		bookmarksdepth=2]{hyperref}

\usepackage{bm}

\input xy
\usepackage[all]{xy}
\xyoption{line}
\xyoption{arrow}
\xyoption{color}
\SelectTips{cm}{}
\usepackage{tikz}
\usetikzlibrary{shapes}
\usetikzlibrary{decorations.markings}
\usetikzlibrary{decorations.pathreplacing}
\tikzstyle directed=[postaction={decorate,decoration={markings,
    mark=at position #1 with {\arrow{>}}}}]

\newcommand{\hackcenter}[1]{
 \xy (0,0)*{#1}; \endxy}

\tikzset{->-/.style={decoration={
  markings,
  mark=at position #1 with {\arrow{>}}},postaction={decorate}}}

\tikzset{middlearrow/.style={
        decoration={markings,
            mark= at position 0.5 with {\arrow{#1}} ,
        },
        postaction={decorate}
    }
}

\usepackage{graphicx}
\usepackage{color}
\theoremstyle{plain}
\newtheorem{theorem}{Theorem}

\newtheorem{corollary}[theorem]{Corollary}
\newtheorem{proposition}[theorem]{Proposition}
\newtheorem{lemma}[theorem]{Lemma}

\theoremstyle{definition}
\newtheorem{example}[theorem]{Example}
\newtheorem{definition}[theorem]{Definition}

\newtheorem{convention}[theorem]{Convention}
\theoremstyle{definition}
\newtheorem{remark}[theorem]{Remark}
\newtheorem{warning}[theorem]{Warning}

\numberwithin{equation}{section}
\numberwithin{theorem}{section}
\newcommand{\refequal}[1]{\xy {\ar@{=}^{#1}
(-1,0)*{};(1,0)*{}};
\endxy}
\newcommand{\Hom}{{\rm Hom}}

\newcommand{\maps}{\colon}
\renewcommand{\to}{\rightarrow}

\def\1{\mathbf{1}}%

\newcommand{\und}[1]{\underline{#1}}

\newcommand{\rkq}{{\rm rk}_q}

\def\mf{\mathfrak}

\numberwithin{equation}{section}

\let\hat=\widehat
\let\tilde=\widetilde

\let\epsilon=\varepsilon

\def\C{{\mathbb{C}}}
\def\cal#1{\mathcal{#1}}%
\def\nn{\notag}

\renewcommand{\l}{\lambda}

\newcommand{\W}{\cal{W}}
\newcommand{\Wa}{\cal{W}^{\alpha}}
\def\cal#1{\mathcal{#1}}

\newcommand\nc{\newcommand}
\nc\rnc{\renewcommand}
\nc\Kar{\operatorname{Kar}}
\nc\End{\operatorname{End}}

\newcommand{\infl}{{\rm infl}}
\newcommand{\pr}{{\rm pr}}

\newcommand{\scs}{\scriptstyle}
\nc\Sym{\operatorname{Sym}}
\allowdisplaybreaks

\DeclareMathOperator{\Ext}{Ext}
\DeclareMathOperator{\im}{im}

\newcommand{\A}{\mathcal{A}}
\newcommand{\B}{\mathcal{B}}
\newcommand{\F}{\mathbb{F}}
\newcommand{\Z}{\mathbb{Z}}
\newcommand{\x}{\mathbf{x}}
\newcommand{\y}{\mathbf{y}}
\newcommand{\z}{\mathbf{z}}

\newcommand{\Ib}{\mathbf{I}}
\newcommand{\Bq}{\overline{\mathcal{B}}_l(n,k)}
\newcommand{\Bqk}{\overline{\mathcal{B}}_l^{\Bbbk}(n,k)}
\newcommand{\co}{\colon}
\newcommand{\dimq}{\dim_q}

\newcommand{\ootimes}{
  \mathbin{
    \mathchoice
      {\buildcircleotimes{\displaystyle}}
      {\buildcircleotimes{\textstyle}}
      {\buildcircleotimes{\scriptstyle}}
      {\buildcircleotimes{\scriptscriptstyle}}
  }
}

\newcommand\buildcircleotimes[1]{%
  \begin{tikzpicture}[baseline=(X.base), inner sep=0, outer sep=0]
    \node[draw,circle] (X)  {$#1\otimes$};
  \end{tikzpicture}%
}

\title{Ozsv{\'a}th-Szab{\'o} bordered algebras and \\subquotients of category $\cal{O}$}

\begin{document}

\author{Aaron D. Lauda}
\email{lauda@usc.edu}
\address{Department of Mathematics\\ University of Southern California \\ Los Angeles, CA}

\author{Andrew Manion}
\email{amanion@usc.edu}
\address{Department of Mathematics\\ University of Southern California \\ Los Angeles, CA}

\date{\today}

\begin{abstract}
We show that Ozsv{\'a}th--Szab{\'o}'s bordered algebra used to efficiently compute knot Floer homology is a graded flat deformation of the regular block of a $\mf{q}$-presentable quotient of parabolic category $\cal{O}$.
We identify the endomorphism algebra of a minimal projective generator for this block with an explicit quotient of the Ozsv{\'a}th--Szab{\'o} algebra using Sartori's diagrammatic formulation of  the endomorphism algebra. Both of these algebras give rise to categorifications of tensor products of the vector representation $V^{\otimes n}$ for $U_q(gl(1|1))$. Our isomorphism allows us to transport a number of constructions between these two algebras, leading to a new (fully) diagrammatic reinterpretation of Sartori's algebra, new modules over Ozsv{\'a}th--Szab{\'o}'s algebra lifting various bases of $V^{\otimes n}$, and bimodules over Ozsv{\'a}th--Szab{\'o}'s algebra categorifying the action of the quantum group element $F$ and its dual on $V^{\otimes n}$.
\end{abstract}

\maketitle

\setcounter{tocdepth}{1}

\tableofcontents

% ====================================================
\section{Introduction}
% ====================================================

Categorification originated with the goal of lifting quantum 3-manifold invariants, specifically the Witten--Reshetikhin--Turaev invariants based on Chern--Simons gauge theory, to smooth 4-manifold invariants.  Within Crane and Frenkel's original proposal~\cite{CF}, quantum groups associated to semi-simple Lie algebras heavily influenced the investigation of categorified quantum 3-manifold invariants.  Positive and integral structures arising from geometric representation theory and the discovery of canonical bases for quantum groups suggested that quantum groups could themselves be categorified.  These original insights ultimately birthed the field of higher representation theory and the study of categorified quantum groups.

Quantum groups associated to symmetrizable Kac-Moody algebras have been categorified along with a significant amount of their representation theory~\cite{KL3,Rou2,RouQH,KK,Web}.
These categorical representations, or higher representations, govern link homology theories categorifying the Reshetikhin-Turaev invariants of knots and tangles.  Though these link homologies such as Khovanov-Rozansky homology can be formulated in many different languages like matrix factorizations \cite{KhR1}, Soergel bimodules~\cite{Rou-Soergel, KhTrip}, coherent sheaves on the affine Grassmannian~\cite{CaKa-slm}, BGG category $\cal{O}$~\cite{StroppelKh,Sus,MazS}, and tensor product 2-representations \cite{Web}, higher representation theory unifies these different formulations by realizing them all as 2-representations of categorified quantum groups~\cite{Cautis-clasp,Mackaay-Webster}.

Despite these successes,  thus far the higher representation theoretic approach has fallen short at categorifying quantum invariants for 3-manifolds, not just links in $S^3$. This issue is partly related to the challenges associated with categorification at a root of unity, though there has been some progress in this direction~\cite{Kh4,KQ,Q1,KQYpDG,QYpDG}.

% ------------------------------------------------------
\subsection{Knot Floer homology and categorification}
% ------------------------------------------------------

On the other hand, Heegaard Floer homology \cite{HFOrig,PropsApps} has proven tremendously successful as a 4-dimensional TQFT sensitive to the smooth structure of 4-manifolds.  This theory has a much different flavor than the quantum invariants discussed above; it is a symplectic approach to Seiberg--Witten theory, a more analytically tractable relative of the celebrated Donaldson--Floer invariants that initially sparked mathematical interest in TQFTs in the 1980s. Its definitions rely on moduli spaces of pseudoholomorphic curves as in Lagrangian Floer homology. Heegaard Floer theory also provides a categorification of the Alexander polynomial, similar to Khovanov's categorification of the Jones polynomial, called knot Floer homology ($HFK$) \cite{OSzHFK,RasmussenThesis}. This invariant determines important knot-theoretic information like genus and fiberedness that is only bounded or restricted by the Alexander polynomial \cite{OSzGenus, NiFibered, JuhaszDecomposition}.

Despite its analytic origins, knot Floer homology can be understood fruitfully from an algebraic perspective by making it into a local tangle invariant based on the ideas of bordered Floer homology as studied by Lipshitz--Ozsv{\'a}th--Thurston \cite{LOTBorderedOrig}. In this framework, one associates $A_{\infty}$-algebras (usually dg) to parametrized surfaces and $A_{\infty}$-bimodules to certain diagrams for 3d cobordisms. Applying these methods to tangle complements viewed as cobordisms between genus-zero surfaces with boundary, Ozsv{\'a}th--Szab{\'o} \cite{OSzNew} recently introduced a computational method for knot Floer homology.
They have used their theory to write a very fast $HFK$ calculator program \cite{HFKcalc}, capable of computing $HFK$ and some related concordance invariants for most knots with 40-50 crossings (and some with significantly more, e.g. the 80+ crossing examples\footnote{This claim can be verified by downloading ComputeHFKv1.zip from \cite{HFKcalc}, then compiling the enclosed C++ files and running them on the enclosed examples. The files K2b86.txt, K3c83.txt, and K3d91.txt give 86, 83, and 91-crossing presentations of the knots $K_2$ and $K_3$ from \cite{ManMachine}; they can be run in a few minutes on a laptop.} from \cite{ManMachine}).

Ozsv{\'a}th--Szab{\'o}'s theory is similar in its motivation and formal structure to another construction due to Petkova--V{\'e}rtesi \cite{PetkovaVertesi}, which computes $HFK$ using local versions of ``nice diagrams'' in the sense of \cite{SarkarWang}. Holomorphic curve counts arising from nice diagrams can always be understood combinatorially, but the resulting Heegaard Floer invariants are typically the homology of complexes with a large (e.g. factorial-sized) number of generators. Ozsv{\'a}th--Szab{\'o} \cite{OSzNew} gain efficiency by using a diagram giving a small and natural number of generators, but in which the curve counts are quite complicated to understand; nevertheless, they succeed in describing the counts and their associated $\A_{\infty}$ structures algebraically.

The local or bordered approach to knot Floer homology provides a bridge to representation theory by categorifying the Alexander polynomial as a {\em quantum invariant}. From the introduction of quantum link invariants in the 1980s, it became a natural question to ask if the Alexander polynomial has a definition as a physical observable in some 3-dimensional Chern--Simons theory. The relevant Chern--Simons theory turns out to be the one whose structure group is given by the Lie superalgebra $\mf{gl}(1|1)$ (or $\mf{gl}(n|n)$ for any $n > 0$).  Indeed, the Alexander polynomial can be understood as the quantum invariant associated to the quantum superalgebra $U_q(\mf{gl}(1|1))$, where endpoints of a tangle correspond to tensor powers of the vector representation $V$ and its dual $V^{\ast}$, and tangles give maps of $U_q(\mf{gl}(1|1))$-representations, see e.g. \cite{KaufSal, SartoriAlexander}.

As shown in \cite{ManionDecat}, Ozsv{\'a}th--Szab{\'o}'s theory gives a categorification of these tensor powers of $V$ and $V^*$, together with tangle maps between them.  Closely related results were obtained for Petkova--Vertesi's theory in \cite{EPV}; this theory categorifies tensor powers of $V$ and $V^*$ with one additional factor $L(\lambda_P)$. See \cite{TianUq,TianUT} for still another approach using bordered Floer ingredients with more of a contact topology flavor, although Tian does not categorify tangle maps. None of these constructions categorify the action on $V^{\otimes n}$ of both generators $E$ and $F$ of $U_q(\mf{gl}(1|1))$; Ellis--Petkova--Vertesi categorify both $E$ and $F$ acting on a related representation, Tian works with a different quantum group, and actions of quantum group generators are not considered at all in \cite{OSzNew,OSzNewer,ManionDecat} (we will rectify this last issue in the current paper).

% ------------------------------------------------------
\subsection{Algebraic categorifications associated with \texorpdfstring{$\mf{gl}(1|1)$}{gl(1|1)}}
% ------------------------------------------------------

Moving to the algebraic side, Sartori \cite{Sar-tensor} defines a categorification of tensor powers of $V$, with intertwining maps and (half of) the action of $U_q(\mf{gl}(1|1))$, in the usual spirit of algebraic categorification via the Bernstein--Gelfand--Gelfand category $\cal{O}$ \cite{BGG}. More specifically, categorification is achieved though certain subquotient categories of category $\cal{O}(\mf{gl}_n)$, or what are referred to as $\mf{q}$-presentable quotients $\cal{O}^{\mf{p}, \mf{q}\textrm{-pres}}_0$ of the regular block $\cal{O}^{\mf{p}}_0$ of the parabolic subcategory $\cal{O}^{\mf{p}} \subset \cal{O}$ (see Section~\ref{sec:connectO} for more details; such presentable quotients were first defined in \cite{FKM-presentable} and studied in relation to categorification in \cite{MazS-presentable}). Sartori uses projective functors on these quotients to categorify the Hecke algebra action on $V^{\otimes n}$ and Zuckerman's approximation functors to categorify the action of $U_q(\mf{gl}(1|1))$ (more precisely, of the generator $F$ of half of the quantum group and its dual $E'$ with respect to a bilinear form on $V^{\otimes n}$ arising from graded dimensions of morphism spaces in $\cal{O}^{\mf{p}, \mf{q}\textrm{-pres}}_0$).

In the $\mf{sl}(n)$ case, category $\cal{O}$ is related to geometric categorification via perverse sheaves \cite{BFK} by localization, and related to elementary diagrammatic definitions in the original Khovanov style by work of Stroppel \cite{StroppelKh}. Webster can describe general blocks of (parabolic) category $\cal{O}$, up to equivalence, as module categories over his diagram tensor product algebras from \cite{Web}. In this way, Sartori's categorifications fit naturally into traditional structures associated with higher representation theory.

While an explicit description of blocks in parabolic category $\cal{O}$ can become unwieldy in general, Sartori defines diagrammatic algebras whose module categories are equivalent to  the subquotients $\cal{O}^{\mf{p}, \mf{q}\textrm{-pres}}_0$ for $\mf{p}, \mf{q}$ of Levi type $(1,\ldots,1,n-k)$ and $(k,1,\ldots,1)$ respectively~\cite{Sar-diagrams}.  The structure of these subquotients is accessed through a combinatorial relationship with Soergel modules corresponding to smooth Schubert varieties.  Through a careful analysis of these Soergel modules and the maps between them, Sartori defines algebras $\cal{A}_{n,k}$ we call {\em Sartori algebras}.  We note the description of multiplication on these algebras is not entirely diagrammatic;  rewriting products in the basis of the algebra requires significant effort. Sartori goes on to show that his algebras are graded cellular and properly stratified, equipping  them with explicit classes of modules and filtrations lifting the standard and canonical bases for $V^{\otimes n}$ and their duals. Relationships between Sartori's constructions and categorifications of tensor product representations of $\mf{sl}(k)$ are studied in \cite{SS}.

% ------------------------------------------------------
\subsection{Relating categorifications}
% ------------------------------------------------------

Both the Oszv{\'a}th-Szab{\'o} algebras $\B_l(n,k)$ appearing in \cite{OSzNew,OSzNewer} and the Sartori algebras $\cal{A}_{n,k}$ can be used to categorify the same $U_q(\mf{gl}(1|1))$-representations. While higher representation theory is often useful for unifying categorifications like these that come from different worlds, it has not been developed enough in the case of superalgebras to make the path to such a unification clear. New ideas appear necessary for defining categorified $U_q(\mf{gl}(1|1))$ tangle and link invariants from the usual algebraic ingredients (such as geometric categorifications and skew Howe duality) and connecting them to $HFK$. Indeed, the elaborate structures invoked by Ozsv{\'a}th--Szab{\'o} to solve this problem (e.g. curved $\cal{A}_{\infty}$ bimodules) may suggest a way forward on the algebraic side, leading to a wide range of possible generalizations.

In fact, there are surprising general relationships between bordered Floer homology and higher representation theory. Work in preparation \cite{ManionRouquier} of Rapha{\"e}l Rouquier and the second named author will show that in considerable generality, bordered Floer homology has close ties to the $U_q(\mf{gl}(1|1))$ case of Rouquier's tensor product operation $\ootimes$ for higher representations  applied to Khovanov's categorification $\cal{U}^+$ of the positive half of $U_q(\mf{gl}(1|1))$ \cite{Kh-gl11}.  This work reinterprets and extends cornered Floer homology \cite{DM, DLM}, a further extension of bordered Floer homology. The connection with bordered Floer homology yields 2-representations of $\cal{U}^+$ on a very general family of examples, including bordered Floer algebras for surfaces of arbitrary genus, together with gluing formulas based on $\ootimes$.

The constructions of \cite{ManionRouquier} simplify considerably when applied via \cite{MMW2} to Ozsv{\'a}th--Szab{\'o}'s theory \cite{OSzNew}, and higher morphisms in $\cal{U}^+$ do not have room to act. On the other hand, this particular genus-zero example of a bordered Floer algebra is highly symmetric and has an explicit and powerful bimodule theory for tangles. The relationship to Sartori's theory studied in this paper is of particular interest as mentioned above; it also provides a window into the structure of Heegaard Floer homology and its relationship with other areas of mathematics, advancing the general aim of understanding 4-dimensional gauge theories via categorified quantum invariants.

The first hint that such a relationship might exist came
in \cite{Man-KS} where the second  named  author related Sartori's algebra $\cal{A}_{n,1}$ with Ozsv{\'a}th--Szab{\'o}'s algebra $\B_l(n,1)$; both algebras categorify a next-to-extremal weight space of $V^{\otimes n}$ for the $U_q(\mf{gl}(1|1))$ action.   This weight space for $\mf{gl}(1|1)$ is actually isomorphic to the corresponding weight space of the $n$-th tensor power of the vector representation of $\mf{sl}(2)$. Sartori's algebra $\cal{A}_{n,1}$ for this weight space describes $\cal{O}^{\mf{p}}_0$ for $\mf{p}$ of Levi type $(1,n-1)$, since the $\mf{q}$-presentable quotient does nothing here, and correspondingly $\cal{A}_{n,1}$ is isomorphic to the Khovanov--Seidel quiver algebra from \cite{KS} (a well-known algebra describing this particular $\cal{O}^{\mf{p}}_0$). It is shown in \cite{Man-KS} that the Khovanov--Seidel algebra is isomorphic to a quotient of $\B_l(n,1)$.

Our main result generalizes and reframes the quotient description of the Khovanov--Seidel quiver algebra from \cite{Man-KS}.

\begin{theorem}\label{thm:MainQuotient}[cf. Theorem~\ref{thm:newinjective}, Theorem~\ref{thm:flat}]
For $0 \leq k \leq n$, Ozsv{\'a}th--Szab{\'o}'s algebra $\B_l(n,k)$ is a graded flat deformation of Sartori's algebra $\cal{A}_{n,k}$.
\end{theorem}

Explicitly, $\cal{A}_{n,k}$ is isomorphic to the quotient of $\B_l(n,k)$ by the ideal generated by the elementary symmetric polynomials $e_i(U_1,\ldots,U_n)$ for $1 \leq i \leq k$, where $U_i$ is a central element of $\B_l(n,k)$ reviewed in Section~\ref{sec:BigStep}. We prove Theorem~\ref{thm:MainQuotient} for the $\Z$ version of Ozsv{\'a}th--Szab{\'o}'s algebra from \cite[Section 12]{OSzNewer} and a $\Z$ lift of Sartori's  $\C$-algebra defined here. Flatness of the deformation follows from Theorem~\ref{thm:flat}, which gives an explicit basis for $\B_l(n,k)$ as a free module over the polynomial ring $\Z[\varepsilon_1,\ldots,\varepsilon_k]$ with $\varepsilon_i$ acting as $e_i(U_1,\ldots,U_n)$.

The case $k = 1$ is the first main theorem of \cite{Man-KS}, although flatness was not considered.
For general $k$, Sartori's algebras are much more complicated than the Khovanov--Seidel algebra, so more intricate arguments are required.  By transporting the diagrammatics of $\B_l(n,k)$ from \cite{OSzNew,MMW1} through this isomorphism, we obtain as a corollary a new purely diagrammatic interpretation of Sartori's algebra with a more natural product operation.

The following remarks, written with Heegaard Floer readers in mind, may be helpful for those unfamiliar with category $\cal{O}$ but familiar with Khovanov's tangle theory involving the arc algebra $H^n$ \cite{KhovFunctor}. Given known and conjectured spectral sequences relating Khovanov and Khovanov--Rozansky homology to $HFK$ \cite{KPKH, DGR, DowlinSS}, one could try to find relationships between $\B_l(n,k)$ and $H^n$. Since $\B_l(n,k)$ has $\binom{n}{k}$ basic idempotents, it is natural to replace $H^n$ with the ``platform algebras'' of \cite{StroppelKh, ChenKhovanov} having $H^n$ as an idempotent truncation. These have $\binom{n}{k}$ basic idempotents, but they still seem unrelated to the idempotents of $\B_l(n,k)$.

Representation theory sheds significant light on this question. Idempotents in $\B_l(n,k)$ and the platform algebras both correspond to certain canonical basis elements for a $2^n$-dimensional vector space $V^{\otimes n}$, but the basis elements depend on the quantum group in question: the $U_q(\mf{sl}(2))$ action gives one canonical basis for $V^{\otimes n}$ while the $U_q(\mf{gl}(1|1))$ action gives a different one. When $k \in \{2,\ldots,n-1\}$, this difference in bases means we cannot hope to relate $\B_l(n,k)$ with the platform algebras except in a derived sense (when $k=1$ the bases agree and the platform algebra is Khovanov--Seidel's algebra).

To make progress, one could ask where the platform algebras come from. Their idempotents correspond to indecomposable projectives in parabolic versions $\cal{O}^{\mf{p}}_0$ of category $\cal{O}$ by \cite{StroppelKh}, where $\mathfrak{p}$ has Levi type $(k,n-k)$.  Similarly, Sartori's idempotents correspond to indecomposable projectives in related categories $\cal{O}^{\mf{p}, \mf{q}\textrm{-pres}}_0$, where the Levi types are as described above; in this case they categorify $U_q(\mf{gl}(1|1))$ basis elements, not $U_q(\mf{sl}(2))$ basis elements. A reasonable update of the question about $\B_l(n,k)$ and $H^n$ or the platform algebras is to ask whether $\B_l(n,k)$ is related to Sartori's algebras. Theorem~\ref{thm:MainQuotient} answers this question affirmatively by giving a close relationship with immediate applications for the structure of Ozsv{\'a}th--Szab{\'o}'s theory~\cite{OSzNew}.

% ------------------------------------------------------
\subsection{Applications }
% ------------------------------------------------------

Theorem~\ref{thm:MainQuotient} establishes a bridge between modern constructions in Heegaard Floer homology and the wider world of mathematics. For example, Theorem~\ref{thm:MainQuotient} is, to the authors' knowledge, the first result relating $HFK$  with category $\cal{O}$, outside the $k=1$ case proved in \cite{Man-KS}. We note that other bordered Floer algebras have been related to category $\cal{O}$ in \cite{AGW}, including to the Khovanov--Seidel quiver algebra; these bordered Floer algebras appear to be more related to $\widehat{HF}$ of branched double covers than to $HFK$, although interesting connections between the two may exist. In general, work of Auroux suggests a path from bordered Floer algebras to geometry via Fukaya categories of symmetric products; we discuss this connection further in Section~\ref{sec:Fuk} below. First we discuss some ramifications for $\B_l(n,k)$ of the conceptual framework surrounding $\cal{A}_{n,k}$.

% - - - - - - - - - - - - - - - - -
\subsubsection{Bilinear forms on $V^{\otimes n}$}
% - - - - - - - - - - - - - - - - -
In general, given a graded categorification of a $\C(q)$-vector space $V$, one gets a sesquilinear pairing on $V$ from graded dimensions of Ext spaces in the categorification, which can be made bilinear using algebra symmetries if they exist. In particular, the results of \cite{ManionDecat} imply that projective modules over Ozsv{\'a}th--Szab{\'o}'s algebra $\B_l(n,k)$ give a bilinear pairing on $V^{\otimes n}$ where $V$ is the vector representation of $U_q(\mf{gl}(1|1))$, but this pairing is not discussed in \cite{ManionDecat}.

On the other hand, Sartori \cite{Sar-tensor} studies the bilinear form arising from his categorification in some detail and shows that this form has a scalar matrix in the standard basis of each weight space of $V^{\otimes n}$. His results suggest analogous conjectures for the Ozsv{\'a}th--Szab{\'o} bilinear form, which we verify; this form also turns out to be scalar in the standard basis of each weight space of $V^{\otimes n}$, with a different scalar than Sartori's.

Using the simple relationship between these forms, we revisit the identification of indecomposable projectives over $\B_l(n,k)$ with basis elements of $V^{\otimes n}$ given in \cite{ManionDecat}. The bilinear forms suggest a change of conventions under which indecomposable projectives correspond exactly to Sartori's canonical basis elements of $V^{\otimes n}$, rather than the modified basis elements introduced in \cite{ManionDecat}.

\begin{theorem}[cf Theorem~\ref{thm:CategorificationsCompatible}]
Under these conventions, the projection functors
\[
\pr \maps \B_l^{\Bbbk}(n,k){\rm -proj} \to \cal{A}_{n,k}{\rm - proj}
\]
between categories of finitely generated projective graded modules induce isomorphisms on Grothendieck groups $K_0$ intertwining the identification of indecomposable projectives with canonical basis elements on each side.
\end{theorem}

% - - - - - - - - - - - - - - - - -
\subsubsection{Categorified action of $U_q(\mf{gl}(1|1))$}
% - - - - - - - - - - - - - - - - -

The quantum group generator $F$ acting on $V^{\otimes n}$ has a dual $E'$ with respect to Sartori's bilinear form.  This dual is related to the usual generator $E$ by a weight dependent scalar in $\C(q)$; as a quantum group element, $E'$ only makes sense in the idempotented, or modified, form $\dot{U}_q(\mf{gl}(1|1))$ of the quantum group.     Sartori categorifies the action of $E', F \in \dot{U}_q(\mf{gl}(1|1))$ via certain Zuckerman functors $\cal{E}'$ and $\cal{F}$ acting on $\cal{O}^{\mf{p},\mf{q}\textrm{-pres}}_0$. This action can be interpreted on the algebra level as tensoring with bimodules $\mathbf{E}'=(\mathbf{E}')^S$ and $\mathbf{F}=\mathbf{F}^S$ where $\mathbf{E}'$ is the left dual of $\mathbf{F}$ as a bimodule.

The close relationship between Ozsv{\'a}th--Szab{\'o}'s and Sartori's algebras suggests analogous definitions of bimodules $\mathbf{E}''=(\mathbf{E}'')^{OSz}$ and $\mathbf{F}=\mathbf{F}^{OSz}$ over Ozsv{\'a}th--Szab{\'o}'s algebras.

\begin{theorem}[cf. Theorem~\ref{thm:OSzFCat}, Theorem~\ref{thm:OSzECat}]
The bimodules $\mathbf{E}''$ and $\mathbf{F}$ over Ozsv{\'a}th--Szab{\'o}'s algebras square to zero and categorify the action of $E''$ and $F$ on $V^{\otimes n}$, where $E''$ is the dual of $F$ with respect to Ozsv{\'a}th--Szab{\'o}'s bilinear form.
\end{theorem}

This result fills a gap in the discussion of \cite{ManionDecat}. Unlike $E'$, $E'' = (q^{-1} - q)EK$ makes sense directly in $U_q(\mf{gl}(1|1))$ without passing to an idempotented form. We note that the relations satisfied by $E''$ and $F$ agree with the algebra $U_T$ studied by Tian~\cite{TianUq} upon setting $T = K^2$.

\begin{theorem}[cf. Theorem~\ref{thm:SarOSzFRel}, Theorem~\ref{thm:SarOSzERel}]
The inflation functors
\[
\infl \maps \cal{A}_{n,k}{\rm - fmod} \to \B_l^{\Bbbk}(n,k){\rm -fmod}
\]
between categories of finite dimensional graded modules intertwine $(\cal{E}'')^{OSz}$ and $\cal{F}^{OSz}$ with $(\cal{E}')^{S}$ and $\cal{F}^{S}$.
\end{theorem}

The constructions of \cite{ManionRouquier} also yield bimodules that square to zero, defined over a bordered strands algebra known by \cite{MMW2} to be quasi-isomorphic to $\B_l(n,k)$.  The bimodules defined here are compatible with the ones from \cite{ManionRouquier} under this quasi-isomorphism.

% - - - - - - - - - - - - - - - - -
\subsubsection{Modules over Ozsv{\'a}th--Szab{\'o}'s algebras}
% - - - - - - - - - - - - - - - - -

One important feature of Sartori's algebra $\cal{A}_{n,k}$ is that it is graded cellular and properly stratified (for more details, see \cite{Sar-diagrams}). The cellular structure gives us a family of modules over $\cal{A}_{n,k}$ (cell modules or standard modules) whose classes in an appropriate Grothendieck group correspond to standard tensor product basis elements of $V^{\otimes n}$. Thus, from $\cal{A}_{n,k}$, Sartori naturally sees both the standard tensor-product basis and the canonical basis for $V^{\otimes n}$.

We can use our quotient map to inflate Sartori's modules over $\cal{A}_{n,k}$ into modules over $\B_l(n,k)$ (in other words, an element of $\B_l(n,k)$ acts after applying the quotient map). Our understanding of the bilinear forms on $V^{\otimes n}$ and how they relate allows us to identify these inflated modules with certain basis elements of $V^{\otimes n}$.

\begin{theorem}[cf. Theorem~\ref{thm:BorderedCategorificationOfBases}]
The inflations of Sartori's modules categorify the bases of $V^{\otimes n}$ listed in Theorem~\ref{thm:BorderedCategorificationOfBases}, including a multiple of the standard basis as well as the Ozsv{\'a}th--Szab{\'o} dual standard basis with no multiple.
\end{theorem}

To see the standard basis of $V^{\otimes n}$ via modules over $\B_l(n,k)$, rather than a weight dependent multiple of this basis, it would be desirable to give $\B_l(n,k)$ the structure of a (graded) affine cellular algebra.

% - - - - - - - - - - - - - - - - -
\subsubsection{Bimodules for intertwining maps}
% - - - - - - - - - - - - - - - - -

In \cite{KS}, Khovanov--Seidel define dg bimodules categorifying maps for braids acting on the weight space of $V^{\otimes n}$ categorified by their quiver algebra $\cal{A}_{n,1}$. It is shown in \cite{Man-KS} that these bimodules are $\cal{A}_{\infty}$ homotopy equivalent to Ozsv{\'a}th--Szab{\'o}'s bimodules over $\B_l(n,1)$ after applying induction and restriction. Generalizing to $k > 1$, Sartori has a categorical Hecke algebra action, including functors categorifying $U_q(\mf{gl}(1|1))$-linear maps for singular crossings or ``thick edges.'' Ozsv{\'a}th--Szab{\'o} have bimodules for tangles with arbitrary orientations, but they do not define bimodules for thick edges. Alishahi--Dowlin's bimodules from \cite{ADLink} provide one candidate generalization (see also \cite{AlishahiDowlin}); more complicated $\A_{\infty}$ bimodules are defined in \cite{ManionSingular}, and both constructions may be relevant when trying to define both upward- and downward-pointing thick-edge bimodules. It would be desirable to relate any of these bimodules to Sartori's categorical Hecke action.

% - - - - - - - - - - - - - - - - -
\subsubsection{Fukaya categories}\label{sec:Fuk}
% - - - - - - - - - - - - - - - - -
Sartori's theory fits into a rich framework of algebraic and geometric constructions, and Theorem~\ref{thm:MainQuotient} suggests natural defomations of these structures. One can use Theorem~\ref{thm:MainQuotient} to investigate relationships between Heegaard Floer homology and deformed category $\cal{O}$, Schubert varieties, Soergel modules, and other entities. While these objects might seem far afield from the holomorphic curve counts motivating the definition of $\B_l(n,k)$, general conjectures suggest that geometric categorifications should have Fukaya interpretations. For example, $\cal{O}^{\mf{p}}_0$ is equivalent to a category of perverse sheaves on a partial flag variety $X_{\mf p}$ \cite{BeilinsonBernstein, HollandPolo}, and thereby to a subcategory of a Fukaya category of $T^*(X_{\mf p})$ \cite{NadlerZaslow, Nadler}.

Going beyond cotangent bundles, the symplectic Khovanov homology program \cite{SeidelSmith, ManolescuSymplectic, AS-formal, AS-khov} formulates standard algebraic categorifications like Khovanov homology in terms of Fukaya categories of certain symplectic manifolds. In fact, Khovanov--Seidel's work in \cite{KS} can be seen as a progenitor of this program; they interpret their quiver algebra as an Ext-algebra of Lagrangians in the Fukaya category of a Milnor fiber. Similar results have been obtained for Khovanov's arc algebra $H^n$ by Abouzaid--Smith \cite{AS-formal,AS-khov}, allowing them to prove that the construction of \cite{SeidelSmith} agrees with Khovanov homology. The symplectic interpretation of the Khovanov--Seidel algebra and $H^n$ has recently been extended to the above-mentioned platform algebras by Mak--Smith \cite{MakSmith}.

By the results of this paper, Ozsv{\'a}th--Szab{\'o}'s algebras $\B_l(n,k)$ for general $k$ are flat deformations of algebras describing $\cal{O}^{\mf{p}, \mf{q}\textrm{-pres}}_0$ for certain $\mf{p},\mf{q}$. It is reasonable to suspect that $\B_l(n,k)$ describes an $\cal{O}^{\mf{p}, \mf{q}\textrm{-pres}}_0$-analogue of Soergel's deformed category $\hat{\cal{O}}$ \cite{Soe2} (related to equivariant rather than ordinary cohomology). When $k=1$, so there is no $\mf{q}$-presentable quotient, one can further speculate that $\B_l(n,k)$ is an Ext-algebra of Lagrangians in a deformed or equivariant Fukaya category of Khovanov--Seidel's Milnor fiber. For $k>1$ one could hope for a similar story, although it is less clear what symplectic manifolds should be involved.

On the other hand, bordered Floer algebras are known by Auroux's work \cite{Auroux} to be related to Fukaya categories of symmetric products. Using the strands interpretation of $\B_l(n,k)$ given in \cite{MMW2} (and assuming that Auroux's results extend to this setting), $\B_l(n,k)$ should be an Ext-algebra of certain noncompact Lagrangians in a wrapped Fukaya category of the $k$-th symmetric power of an $n$-punctured disk. When $k=1$, it appears that we have two Fukaya interpretations of $\B_l(n,k)$; one is presumably equivariant and applied to a Milnor fiber, while the other is non-equivariant and applied to a punctured disk. It would be desirable to have a Fukaya-theoretical explanation of this apparent coincidence, and the quotient results of this paper; the question is especially immediate when $k=1$ but generalized explanations for arbitrary $k$ do not seem implausible. Viewing the Milnor fiber as the total space of a Lefschetz fibration following Khovanov--Seidel, the complement of the singular fibers has a free $\C^*$ action whose quotient is the $n$-punctured plane. Roughly, the above apparent coincidence for $k=1$ seems to suggest that a suitably $\C^*$-equivariant version of the wrapped Fukaya category of this complement should be related to an analogous category for the full Milnor fiber; if so, it would be informative to understand the relationship geometrically.

%If a suitably $\C^*$-equivariant version of the wrapped Fukaya category of this complement can be related to an analogous category for the full Milnor fiber, a Fukaya--categorical explanation in the case $k=1$ does not seem out of reach.

% ------------------------------------------------------
\subsection*{Acknowledgements}
% ------------------------------------------------------
The authors especially thank Antonio Sartori for valuable discussions relating to the relationship established in this paper, and for suggesting the form of the quotient map relating his algebra to Ozsv{\'a}th--Szab{\'o}'s in 2014.   The authors would also like to thank Jonathan Brundan, Ben Elias, Eric Friedlander, Sheel Ganatra, Anthony Licata, Ciprian Manolescu, Volodymyr Mazorchuk, Ina Petkova, Rapha{\"e}l Rouquier, Joshua Sussan, and Zolt{\'a}n Szab{\'o} for helpful conversations and suggestions. A.D.L.~ is partially supported by the NSF grants DMS-1902092 and DMS-1664240.

% ====================================================
\section{Ozsv{\'a}th--Szab{\'o}'s algebras} \label{sec:bordered}
% ====================================================

In \cite{OSzNew}, Ozsv{\'a}th--Szab{\'o} define an I-state to be a subset $\x \subset \{0,\ldots,n\}$ (such states correspond to primitive idempotents in the algebras they consider). Here we will work with a truncated version of the algebra disallowing $n \in \x$; we will call a subset $\x \subset \{0,\ldots,n-1\}$ a left I-state.    We write $V(n,k)$ for the set of I-states with $|\x|=k$  and $V_{l}(n,k)$ for the subset of left I-states with $|x|=k$. For $\x \in V(n,k)$, write $\x = \{x_1,\dots, x_k\}$ with $x_1 < \dots < x_k$.

\begin{convention}
If $\x \xrightarrow{a} \y \xrightarrow{b} \z$ are arrows in a quiver, we will write their product in the path algebra as $ab$. At times it is also useful to view $a$ and $b$ as morphisms in a category whose objects correspond to the quiver vertices (for example, categories of Soergel modules appearing below). When taking this perspective, we will view $a$ as a morphism from $\y$ to $\x$ and $b$ as a morphism from $\z$ to $\y$. The composition $ab$, without order reversal, also makes sense in the category and is interpreted as $\x \xleftarrow{a} \y \xleftarrow{b} \z$.
\end{convention}

% ------------------------------------
\subsection{Big-step quiver description}\label{sec:BigStep}
% ------------------------------------

We begin by giving a ``big-step'' quiver description of Ozsv{\'a}th--Szab{\'o}'s algebras, following \cite{OSzNew} (although we work over $\Z$ as in \cite[Section 12]{OSzNewer}).
\begin{definition}
For $n \geq 0$, we define the following elements of $\Z_{\geq 0}^n$ associated to I-states:
\begin{itemize}
\item If $\x$ is an I-state, define $v^{\x}$ by $v_i^{\x} = |\x \cap \{i,i+1,\ldots,n\}|$.
\item If $\x$ and $\y$ are two I-states, define $w^{\x,\y}$ by $w_i^{\x,\y} = \frac{1}{2}|v_i^{\x} - v_i^{\y}|$.
\item If $\x$, $\y$, and $\z$ are three I-states, define $g^{\x,\y,\z}$ by $g_i^{\x,\y,\z} = w_i^{\y,\z} - w_i^{\x,\z} + w_i^{\x,\y}$.
\end{itemize}
\end{definition}

\begin{definition}
For $n \geq 0$ and $0 \leq k \leq n$, the $\Z[U_1,\ldots,U_n]$-algebra $\B_0(n,k)$ is the path algebra over $\Z[U_1,\ldots,U_n]$ of the quiver whose vertices are I-states with $k$ elements, with a unique arrow $f_{\x,\y}$ from any I-state $\x$ to any I-state $\y$, modulo the relations
\[
f_{\x,\y} f_{\y,\z} = \prod_{i=1}^n U_i^{g_i^{\x,\y,\z}} f_{\x,\z}.
\]
\end{definition}

Next we take a quotient of $\B_0(n,k)$. For $\x \in V(n,k)$, let $\Ib_{\x} = f_{\x,\x}$; the elements $\Ib_{\x}$ for I-states $\x$ are primitive orthogonal idempotents that sum to $1$. For $1 \leq i \leq n$, define elements $R_i$ and $L_i$ of $\B_0(n,k)$ by
\[
R_i = \sum_{\x \, : \, \x \cap \{i-1,i\} = \{i-1\}} f_{\x,\x \setminus \{i-1\} \cup \{i\}}
\]
and
\[
L_i = \sum_{\x \, : \, \x \cap \{i-1,i\} = \{i\}} f_{\x, \x \setminus \{i\} \cup \{i-1\}}.
\]
\begin{definition}\label{def:BigStepBAlg}
The $\Z[U_1,\ldots,U_n]$-algebra $\B(n,k)$ is the quotient of $\B_0(n,k)$ by the ideal generated by the following elements for $1 \leq i \leq n$:
\begin{enumerate}
\item\label{rel:TwoLinePassOSz} $R_{i-1} R_i$ and $L_i L_{i-1}$,
\item\label{rel:UVanishingOSz} $U_i \Ib_{\x}$ if $\x$ is an I-state with $\x \cap \{i-1,i\} = \varnothing$.
\end{enumerate}
The $\Z[U_1,\ldots,U_n]$-algebra $\B_l(n,k)$ is defined to be
\[
\B_l(n,k) := \left( \sum_{\x \, : \, n \notin \x} \Ib_{\x} \right) \cdot \B(n,k) \cdot \left( \sum_{\x \, : \, n \notin \x} \Ib_{\x} \right);
\]
 equivalently, the sums are over $\x \in V_l(n,k)$.
\end{definition}

Analogous to the ``global'' elements $R_i$ and $L_i$ above, for $1 \leq i \leq n$ we define
\[
U_i = \sum_{\x \in V_l(n,k)} U_i \Ib_{\x},
\]
a central element of $\B_l(n,k)$ (one can define similar elements in $\B(n,k)$).  These elements are obtained by acting on $1 \in \B_l(n,k)$ with $U_i$ using the $\Z[U_1,\ldots,U_n]$-module structure on $\B_l(n,k)$.

% ------------------------------------
\subsection{Small-step quiver description}
% ------------------------------------

The following quiver description was shown to be equivalent to Definition~\ref{def:BigStepBAlg} in \cite[Section 4.4]{MMW1}. The results of \cite{MMW1} are formulated over $\F_2$, but they can be lifted to $\Z$ as in \cite[Section 12]{OSzNewer}. An analogous statement holds for $\B(n,k)$, but we will focus on  the algebra $\B_l(n,k)$ most closely related to Sartori's algebras.
\begin{proposition}[Proposition 4.19 of \cite{MMW1}]\label{prop:MMWThm1}
For $n \geq 0$ and $0 \leq k \leq n$, the algebra $\B_l(n,k)$ is isomorphic as a $\Z[U_1,\ldots,U_n]$-algebra to the path algebra over $\Z$ of the quiver whose vertices are left I-states $\x$ with $k$ elements, whose arrows are given as follows:
\begin{itemize}
\item for vertices $\x,\y$ differing in only one element, with $\x \cap \{i-1,i\} = \{i-1\}$ and $\y \cap \{i-1,i\} = \{i\}$, an arrow (called an $R_i$ arrow) from $\x$ to $\y$ and an arrow (called an $L_i$ arrow) from $\y$ to $\x$
\item for each vertex $\x$ and each $i$ between $1$ and $n$, an arrow from $\x$ to itself (called a $U_i$ arrow)
\end{itemize}
and whose relations we now describe. Any linear combination of paths with the same source and target in the above quiver has an associated (noncommutative) polynomial in the variables $R_i$, $L_i$, and $U_i$, and we include such a linear combination as a generator of the relation ideal if its polynomial is:
\begin{align}
&  R_i U_j - U_j R_i,  \quad L_i U_j - U_j L_i, \;\; \text{or} \;\; U_i U_j - U_j U_i , \label{Brel:Ucommute}
\\
&  R_i L_i - U_i \;\; \text{or} \;\;  L_i R_i - U_i , \label{Brel:loop}
\\
& R_i R_j - R_j R_i ,  \quad L_i L_j - L_j L_i , \;\; \text{or} \;\;   R_i L_j - L_j R_i  \;\; \text{for  $|i-j| > 1$} , \label{Brel:RLcommute}
\\
&  R_{i-1} R_{i} \;\; \text{or} \;\;  L_{i} L_{i-1} , \label{Brel:twoline}
\\
&  \text{$U_i$  if the source and target are a left I-state  $\x$  with  $\x \cap \{i-1,i\} = \varnothing$} . \label{Brel:Uzero}
\end{align}
Note that from this perspective, the first set of relations gives $\B_l(n,k)$ its algebra structure over $\Z[U_1,\ldots,U_n]$.
\end{proposition}

The isomorphism of Proposition~\ref{prop:MMWThm1} makes the following identifications:
\begin{itemize}
\item $R_i$ arrow starting at $\x$ $\leftrightarrow$ $f_{\x,\x \setminus \{i-1\} \cup \{i\}}$
\item $L_i$ arrow starting at $\x$ $\leftrightarrow$ $f_{\x,\x \setminus\{i\} \cup \{i-1\}}$
\item  $U_i$ loop at $\x$ $\leftrightarrow$ $U_i f_{\x,\x}$
\item Trivial path (no edges) at  $\x$ $\leftrightarrow$ the idempotent $\Ib_{\x} = f_{\x,\x}$.
\end{itemize}

% ------------------------------------
\subsection{Gradings}
% ------------------------------------

The algebra $\B_l(n,k)$ has a multi-grading by $\Z^{2n}$, but this grading is not preserved by the quotient map to Sartori's algebra we will define in Section~\ref{sec:Homomorphism}; Sartori's algebra has only a $\Z$ grading. Correspondingly, we will use Ozsv{\'a}th--Szab{\'o}'s single ``Alexander grading'' by $\frac{1}{2}\Z$, corresponding to the power of $t$ in the Alexander polynomial (for a closed knot these powers are always integers, but fractional powers may appear when considering tangles). When viewing the Alexander polynomial as a quantum invariant depending on a parameter $q$, the variables are related by $t = q^2$. Thus, in relating Ozsv{\'a}th--Szab{\'o}'s algebras to Sartori's, it will be useful to double the Alexander gradings.

\begin{definition}
Let $\deg^t$ be defined by $\deg^t(R_i) = \deg^t(L_i) = 1/2$ and $\deg^t(U_i) = 1$. Let $\deg^q$ be obtained by doubling $\deg^t$; explicitly, $\deg^q(R_i) = \deg^q(L_i) = 1$ and $\deg^q(U_i) = 2$.
\end{definition}

\begin{remark}
In the conventions of \cite{OSzNew}, this definition of the single Alexander grading from the refined grading (together with the absence of $C_i$ variables, a homological grading, and a differential) is meant for an algebra associated to $n$ endpoints of a tangle, all pointing downwards. When relating Ozsv{\'a}th--Szab{\'o}'s theory with constructions in representation theory, various changes of convention are often necessary; see e.g. Section~\ref{sec:ConventionComparison}.
\end{remark}

% ------------------------------------
\subsection{Basis for the algebra} \label{subsec:Bbasis}
% ------------------------------------

Ozsv{\'a}th--Szab{\'o}'s proof of \cite[Proposition 3.7]{OSzNew} works over $\Z$ and implies that for any $\x,\y\in V(n,k)$, $\Ib_{\x} \B(n,k) \Ib_{\y}$ is a graded free abelian group with a basis we review below. In particular, we get a basis for $\Ib_{\x} \B_l(n,k) \Ib_{\y}$ when $\x$ and $\y$ are left I-states in $V_{l}(n,k)$.

\begin{definition}\label{def:bord-toofar}
If  $\x,\y \in V(n,k)$ with $|x_i - y_i| > 1$ for some $i$, then $\x$ and $\y$ are said to be {\em too far}.
\end{definition}

If $\x$ and $\y$ are too far then $\Ib_{\x} \B(n,k) \Ib_{\y} = 0$.  Otherwise, we consider the following further definitions.

\begin{definition}\label{def:OSzTerminology}
Let $\x$ and $\y$ be I-states with $k$ elements.
\begin{itemize}
\item If $0 \leq i \leq n$ and $i \in \x \cap \y$, we say that $i$ represents a ``fully used region.'' Otherwise, $i$ represents a ``not fully used region.''
\item If $1 \leq j \leq n$ and $v_j^{\x} \neq v_j^{\y}$, we say that $j$ represents a ``crossed line.'' Otherwise, $j$ represents an ``uncrossed line.''
\end{itemize}
\end{definition}

\begin{remark}\label{rem:DotsAndRegions}
One can visualize the above notions as follows: depict an I-state $\x$ by drawing $n$ parallel lines, with $n+1$ regions outside the lines labeled $0,\ldots,n$, and placing a dot in region $i$ if and only if $i \in \x$. If $\x$ and $\y$ are two I-states with $k$ elements, draw the dot pattern for $\x$ next to that of $\y$, and connect the dots of $\x$ to the dots of $\y$ in the unique order-preserving way. A region is fully used if it has a dot in both $\x$ and $\y$; otherwise it is not fully used. One of the original parallel lines is crossed if a line connecting a dot of $\x$ to a dot of $\y$ crosses it; otherwise it is uncrossed.
\end{remark}

\begin{definition} \label{def:generatingint}
Given I-states $\x,\y \in V(n,k)$ that are not too far, an interval of coordinates $G = [j+1,j+l] = \{j+1,j+2,\ldots,j+l\} \subset [1,n]$ for some $l \geq 1$ is called a ``generating interval'' for $\x$ and $\y$ if:
\begin{itemize}
\item regions $j$ and $j+l$ are not fully used,
\item regions $j+1,\ldots,j+l-1$ are fully used, and
\item the lines $j+1,\ldots, j+l$ are all uncrossed.
\end{itemize}
If $G$ is a generating interval for $\x$ and $\y$, we define a monomial $p_G := U_{j+1} \cdots U_{j+l} \in \Z[U_1,\ldots,U_n]$.
\end{definition}

\begin{proposition}[Proposition 3.7 of \cite{OSzNew}, $\Z$ version]\label{prop:OSzBasis}
Let $\x,\y \in V(n,k)$ be not too far. We have an isomorphism
\[
\Ib_{\x} \B(n,k) \Ib_{\y} \cong \frac{\Z[U_1,\ldots,U_n]}{(p_G : G \textrm{ generating interval})}
\]
of $\Z[U_1,\ldots,U_n]$-modules. When $\x$ and $\y$ are left I-states, we get a description of $\Ib_{\x} \B_l(n,k) \Ib_{\y}$.
\end{proposition}

In the big-step quiver description, the element $1$ of the above quotient of $\Z[U_1,\ldots,U_n]$ corresponds to $f_{\x,\y}$. In the small-step quiver description, a recursive definition of a path in the above quiver giving rise to the element of $\Ib_{\x} \B_l(n,k) \Ib_{\y}$ corresponding to $1 \in \Z[U_1,\ldots,U_n]$ is given in \cite[Definition 2.28]{MMW1}.

\begin{corollary}\label{cor:OSzBasis}
Let $\x,\y \in V(n,k)$ that are not far. The graded abelian group $\Ib_{\x} \B(n,k) \Ib_{\y}$ is free with a basis given by monomials in $U_1,\ldots,U_n$ that are not divisible by the monomial $p_G$ of any generating interval $G$ for $\x$ and $\y$.
\end{corollary}

% ------------------------------------
\subsection{Characterization of generating intervals}
% ------------------------------------

\begin{definition} \label{def:hole-sequence}
Given an I-state $\x$, define its hole sequence $h^{\x}$ to be $\{ 0, 1, \dots , n\} \setminus \x$.  If $|\x| = k$, then $|h^{\x}| = n-k+1$ and we write $h^{\x} = \{ h_1^{\x}, h_2^{\x}, \dots , h_{n-k}^{\x}, h_{n-k+1}^{\x}\}$ with $h_1^{\x} < h_2^{\x} < \dots < h_{n-k+1}^{\x}$.  For all $\x \in V_{l}(n,k)$ we have $h_{n-k+1}^{\x}=n$.
\end{definition}

 In terms of Remark~\ref{rem:DotsAndRegions}, say $\x$ has a hole in a region if it does not have a dot there; then $h_i^{\x}$ is the region containing the $i$-th hole of $\x$ (compare with $x_i$ which is the region containing the $i$-th dot of $\x$).  Notice that $j$ is not fully used if and only if $j \in h^{\x} \cup h^{\y}$. If $h_i^{\x} > h_i^{\y}$, set $j=h_i^{\x}$; then $v_j^{\x} > v_j^{\y}$.

\begin{lemma} \label{lem:h-ipone}
The I-states $\x,\y \in V(n,k)$  are too far if and only if $h_i^{\x} \geq h_{i+1}^{\y}$ or $h_i^{\y} \geq h_{i+1}^{\x}$ for some $1 \leq i \leq n-k$.
\end{lemma}

\begin{proof}
Set $j = h_i^{\x}$ and suppose that  $h_i^{\x} \geq h_{i+1}^{\y}$.   Recall  that $j=h_i^{\x}$ means that $j$ is the position of the $i$th hole in $\x$, so that of the $j$ possible entries in $\{0,1, \dots, j-1\}$ there are $(i-1)$ missing entries in $\x$, leaving $j-(i-1)$ filled.  Hence $x_{j-i+1} < j$.  But $j = h_i^{\x} \geq h_{i+1}^{\y}$ means that $\y$ has at least $i$ holes in the set $\{0, \dots, j-1\}$, implying that $y_{j-i} \geq j$.  We then have
\[
x_{j-i} < x_{j-i+1} < j \leq  y_{j-i},
\quad \text{i.e.} \quad
x_{j-i} \leq  x_{j-i+1}-1 \leq  j -2 \leq y_{j-i} -2,
\]
so that $|x_{j-i} - y_{j-i}| > 1$ and $\x$ and $\y$ are too far.  A similar argument shows that if $h_i^{\y} \geq h_{i+1}^{\x}$ then $\x$ and $\y$ are too far.

Now suppose that $\x$ and $\y$ are too far, so that $|x_i - y_i| >1$ for some index $i$.  First, assume that $y_i \geq x_i +2$ and that $i$ is the largest index satisfying this condition. Since $x_{i+1} \geq x_i + 2$ (otherwise $y_{i+1} \geq y_i + 1 \geq x_i + 3 = x_{i+1} + 2$, contradicting the maximality of $i$), we have $x_i + 1 \notin \x$. It follows that $x_i + 1 = h_{j}^{\x}$ for some $1 \leq j \leq n-k$. Since $y_i > x_i + 1$, the set $\y \cap \{0,\ldots,x_i + 1\}$ has size at most $i-1$, whereas $\x \cap \{0,\ldots,x_i+1\}$ has size $i$. We see that $\y$ has at least one more hole in $\{0,\ldots,x_i+1\} = \{0,\ldots,h_j^{\x}\}$ than does $\x$, so $h_{j+1}(\y) \leq h_j(\x)$.

If $y_i \leq x_i + 1$ for all $i$ but $\x$ and $\y$ are too far, then $x_i \geq y_i + 2$ for some minimal $i$. In this case, we can show that $h_{j+1}(\x) \leq h_j(\y)$ for some $j$ with a similar argument.
\end{proof}

The following gives an alternative characterization of the generating intervals from Definition~\ref{def:generatingint}.

\begin{lemma}\label{lem:SartoriStyleGenInts}
Suppose $\x,\y \in V(n,k)$ are not too far. An interval $[j+1, j+l]$ is a generating interval for $\x$ and $\y$ if and only if $j= \max\left( h_i^{\x}, h_i^{\y} \right)$ and $j+l = \min\left( h_{i+1}^{\x}, h_{i+1}^{\y}\right)$ for some index $ 1 \leq i \leq n-k$.
\end{lemma}

\begin{proof}

First, assume that $[j+1, j+l]$ is a generating interval; we will show that it has the described form. We have $j , j+l \in h^{\x} \cup h^{\y}$ and $j+1, \dots, j+l-1 \in \x \cap \y$.  Lemma~\ref{lem:h-ipone} implies that $h_i^{\x} < h_{i+1}^{\y}$ and $h_i^{\y} < h_{i+1}^{\x}$ for any $1 \leq i \leq n-k$.

We claim that $j = \max\left(h_i^{\x},h_i^{\y}\right)$ for some $i$. Indeed, we have $j = h_i^{\x}$ or $j = h_i^{\y}$ for some $i$ because $j$ is not fully used; assume to derive a contradiction that $j = h_i^{\x}$ and $h_i^{\x} < h_i^{\y}$ (the case when $j = h_i^{\y}$ and $h_i^{\y} < h_i^{\x}$ is analogous). Since $h_i^{\y} < \min\left(h_{i+1}^{\x},h_{i+1}^{\y}\right)$, the first non-fully-used coordinate to the right of $j$ is $h_i^{\y}$, so $j+l = h_i^{\y}$. Since $\x$ and $\y$ have differing number of holes/dots to the left of line $j+l$ (and thus differing numbers of holes/dots to its right), we have $v_{j+l}^{\x} \neq v_{j+l}^{\y}$. Therefore, line $j + l$ is crossed, a contradiction.

Now, since $j = \max\left(h_i^{\x},h_i^{\y}\right)$, the first non-fully-used coordinate to the right of $j$ is $\min\left( h_{i+1}^{\x}, h_{i+1}^{\y}\right)$, so $j + l$ is also equal to this quantity, and we have shown any generating interval is of the form described in the statement.

Conversely, for $1 \leq i \leq n-k$, let $j = \max\left(h_i^{\x},h_i^{\y}\right)$ and $j+l = \min\left( h_{i+1}^{\x}, h_{i+1}^{\y}\right)$. We claim that $[j+1,j+l]$ is a generating interval. Indeed, the coordinates $j$ and $j+l$ are non-fully-used because at least one of $\{\x,\y\}$ has a hole in coordinates $j$ and $j+l$, while all coordinates between them are fully used. The I-states $\x$ and $\y$ have the same number of holes/dots to the left of line $j + 1$ and to the right of line $j+l$, so all lines between $j$ and $j+l$ are uncrossed.

\end{proof}

%-------------------------------------
\subsection{Anti-automorphism} \label{subsec:PsiOS}
%-------------------------------------

In \cite[Section 3.6]{OSzNew}, Ozsv{\'a}th--Szab{\'o} define an anti-automorphism of $\B(n,k)$ that restricts to an anti-automorphism of $\B_l(n,k)$.

\begin{definition} \label{def:PsiOS}
The anti-automorphism $\psi_{OSz}\colon \B_l(n,k) \to \B_l(n,k)^{{\rm opp}}$ sends $R_i \mapsto L_i$, $L_i \mapsto R_i$, and $U_i \mapsto U_i$ in the small-step quiver description of $\B_l(n,k)$.
\end{definition}

\begin{remark}\label{rem:OSzAlgSymms}
In \cite{OSzNew}, $\psi_{OSz}$ is called $o$. Ozsv{\'a}th--Szab{\'o} also describe another symmetry $\cal{R}$ of $\B(n,k)$; restricted to $\B_l(n,k)$, it gives an isomorphism from $\B_l(n,k)$ to $\B_r(n,k)$, where $\B_r(n,k)$ is defined by summing over $\x$ with $0 \notin \x$ rather than $n \notin \x$ in the definition of $\B_l(n,k)$. The symmetry $\cal{R}$ is not present in the Sartori algebras we review below.
\end{remark}

% ====================================================
\section{Sartori's algebras}
% ====================================================

% ------------------------------------
\subsection{Polynomial rings and bases for quotient rings}
% ------------------------------------

Let  $R = \Z[x_1, \dots , x_n]$ and set $\deg(x_i)=2$ so that $R$ is a graded ring. Write $R_{\C}$ for $R \otimes \C = \C[x_1,\ldots,x_n]$. Denote by  $S_n$  the symmetric group, $R^{S_n}$  the symmetric polynomials in $R$, and $R_+^{S_n}$ the symmetric polynomials of strictly positive degree. The coinvariant algebra $R/R_+^{S_n}$ is a graded free abelian group with both a monomial basis $\{ \und{x}^{\und{\ell}} = x_1^{\ell_1} \dots x_{n}^{\ell_n}  \mid 0 \leq \ell_i \leq n-i \}$ and a Schubert polynomial basis $\{\mf{S}_w \mid w \in S_n \}$ indexed by permutations $w\in S_n$.  It is possible to enumerate the monomial basis by permutations by defining $c_i = \#\{ j < w^{-1}(i) \mid w(j) >i \}$ for $w \in S_n$ and defining
\begin{equation} \label{eq:schubert-prime}
  \mf{S}_w'(x_1, \dots, x_n) = x_1^{c_1} x_2^{c_2} \dots x_{n-1}^{c_n}.
\end{equation}
 The monomial $\mf{S}_w'$ is the leading term of the Schubert polynomial $\mf{S}_{w}$  in the lexicographic order generated by $x_n > x_{n-1} > \dots > x_1$.

Define the {\em elementary} and {\em complete} symmetric polynomials by
\[
 e_j(x_1, \dots, x_n) =
 \sum_{1 \leq i_1 < \dots < i_j \leq n}
 x_{i_1} \dots x_{i_j},
 \qquad \quad
 h_j(x_1, \dots, x_n) =
 \sum_{1 \leq i_1 \leq \dots \leq i_j \leq n}
 x_{i_1} \dots x_{i_j}.
\]
for $j \geq 1$. Let $\mathbf{b} = (b_1, \dots, b_n) \in \Z_{\geq 1}^{n}$ be a decreasing sequence that decreases by at most 1, i.e. we have $b_i \geq b_{i+1} \geq b_i-1$ for all $i$. Given $\bf{b}$, define a homogeneous ideal $I_{\mathbf{b}} \subset R$ by
\begin{equation} \label{eq:Ib}
I_{\mathbf{b}}:= \left\langle h_{b_1}(x_1), \; h_{b_2}(x_1,x_2), \; \dots , h_{b_n}(x_1,\dots,x_n)\right\rangle.
\end{equation}
Set $R_{\mathbf{b}} = R/I_{\mathbf{b}}$ and $R_{\mathbf{b}}^{\C} = R_{\mathbf{b}} \otimes \C$.  By a $\Z$ analogue of \cite[Proposition 2.3]{Sar-diagrams}, the quotient ring $R_{\mathbf{b}}$ is a graded free abelian group of rank $b_1 b_2 \cdots b_n$, and a basis is given by
$\{
\mathbf{x}^{\mathbf{j}} = x_1^{j_1} \dots x_n^{j_n} \mid 0 \leq j_i < b_i
\}.$

% ------------------------------------
\subsection{Sequences and permutations} \label{subsection:sequences}
% ------------------------------------

Fix $0 \leq j \leq n$ and consider the set $D=D_{n,k}$ of sequences $\mu=\mu_1  \dots \mu_n$ where each $\mu_i \in \{ \wedge , \vee\}$ and there are $k$ $\wedge$'s and $(n-k)$ $\vee$'s.  By mapping the sequence $\wedge \dots \wedge \vee \dots \vee$ to the identity element $e \in S_n$, each element of $D_{n,k}$ corresponds to a minimal coset representative for $(S_k \times S_{n-k}) \backslash S_n$.  We will identify a $\wedge \vee$ sequence with its corresponding permutation.

Given a $\wedge \vee$ sequence $\mu \in D$, number the positions from left to right, the $\wedge$ terms from 1 to $k$, and the $\vee$ terms from 1 to $n-k$.  Let $\wedge_i^{\mu}$ be the position of the $i$th $\wedge$ and $\vee_j^{\mu}$ be the position of the $j$th $\vee$ in $\mu$.   Define the $\bf{b}$-sequence associated to $\mu \in D$ to be the sequence $\mathbf{b}^{\mu} = (b_1^\mu, b_2^\mu, \dots, b_n^\mu)$ with $b_i^{\mu}-1$ equal to the number of $\wedge$'s strictly to the right of position $i$.

\begin{remark}\label{rem:OSzSarStates}
There is a bijection between $D_{n,k}$ and $V_{l}(n,k)$ defined by sending a left I-state $\x=\{ x_1, \dots, x_k\}$ to the $\wedge \vee$ sequence with a $\wedge$ in position $x_i+1$.  Given $\mu \in D$, we define a left I-state $\mathbf{x}^{\mu}$ whose $i$th term $x_i$ is $\wedge_i^{\mu}-1$. Sartori's sequence $\bf{b}^{\mu}$ is obtained from Ozsv{\'a}th--Szab{\'o}'s weight $v^{\x^{\mu}}$ by adding $1$ to each coordinate.  The hole sequence associated to the I-state $\x^{\mu}$ is expressed in the language of $\wedge \vee$ sequences by $h_i^{\x}=\vee_i^{\mu}-1$.
\end{remark}

\begin{definition} \label{def:Sart-toofar}
We say that $\wedge \vee$ sequences $\mu,\l  \in D_{n,k}$ are \emph{too far} if for some index $1 \leq j \leq n-k-1$ we have $\vee_j^{\mu} \geq \vee_{j+1}^{\l}$ or $\vee_j^{\l} \geq \vee_{j+1}^{\mu}$.
\end{definition}

Lemma~\ref{lem:h-ipone} above shows that the sequences $\l$ and $\mu$ are too far if and only if the corresponding I-states $\x^{\l}$ and $\x^{\mu}$ are too far in the sense of Definition~\ref{def:bord-toofar}; note that for $\x,\y \in V_l(n,k)$ it is impossible to have $h_{n-k}^{\x} \geq h_{n-k+1}^{\y} = n$ or $h_{n-k}^{\y} \geq h_{n-k+1}^{\x} = n$.

% ------------------------------------
\subsection{Soergel modules and their hom spaces}
% ------------------------------------

Soergel modules for the symmetric group $S_n$ are modules ${\sf C}_{w}$ over the polynomial ring $R_{\C}=\C[x_1,\dots, x_n]$ indexed by permutations $w\in S_n$. Set $B=R_{\C}/(R_{\C,+})^{S_n}$ and for a simple transposition $s_i \in S_n$ let $B^{s_i}$ denote the invariants of $B$ under $s_i$. Given a reduced expression $w=s_{i_r} \dots s_{i_1}$, the module ${\sf C}_w$ is defined as the indecomposable direct summand of the module
\[
B \otimes_{B^{s_{i_r}}} B \otimes \dots \otimes B \otimes_{B^{s_{i_1}}} B \otimes_B \C
\]
containing $1 \otimes \dots \otimes 1$.    Soergel showed~\cite{Soe} that ${\sf C}_w$ is the unique indecomposable summand of the above tensor-product $B$-module that is not isomorphic to any ${\sf C}_{w'}$ for $w' \prec w$ in the Bruhat order on $S_n$.  Up to isomorphism, ${\sf C}_w$ does not depend on the choice of reduced expression for $w$.  This holds more generally for $R=\Z[x_1,\dots, x_n]$; see \cite[Theorem 1.1 (3)]{EW}.

Identifying summands giving rise to Soergel modules is in general a difficult task; even the dimensions of the modules are computed using Kazhdan-Lusztig polynomials.  However, things simplify dramatically when the Soergel module is cyclic.  Under the identification of $B$ with the cohomology of the full flag variety, the condition of a Soergel module to be cyclic is equivalent to the rational smoothness of the corresponding Schubert variety in the full flag variety \cite[Appendix]{KazLus}.

\begin{proposition}\label{prop:SartoriSoergel}
We collect the relevant results from \cite{Sar-diagrams} on Soergel modules and the hom spaces between them.
\begin{enumerate}[(i)]
  \item\label{it:SarProp4.5}
Let $w_k$ denote the longest element of $S_k$ considered as an element of $S_n$ via the inclusion $S_k \times S_1 \times \dots \times S_1 \to S_n$.
  For every $\mu \in D_{n,k}$, the module ${\sf C}_{w_k \mu}$ is cyclic  (\cite[Proposition 4.5]{Sar-diagrams}).

  \item\label{it:ExplicitSoergelModules} By \cite[Theorem 4.10]{Sar-diagrams} there is an isomorphism ${\sf C}_{w_k \mu} \cong R_{\mathbf{b}^{\mu}}^{\C}$ so that ${\sf C}_{w_k \mu}$ has a basis given by
  \[
   \{ x_1^{c_1} x_2^{c_2} \dots x_{n-1}^{c_{n-1}} \mid 0 \leq c_i < b_i^{\mu}\}.
  \]

  \item The dimension of ${\sf C}_{w_k \mu}  = R_{\mathbf{b}^{\mu}}^{\C}$ is $b_1^{\mu} \cdots b_n^{\mu}$.

  \item\label{it:SoergelMorphismBasis} Given $\l,\mu \in D_{n,k}$, set $c_i= \max(b_i^{\mu}-b_i^{\l},0)$.  By \cite[Corollary 4.11]{Sar-diagrams}, a basis for the vector space $\Hom_{R_{\C}}( {\sf C}_{w_k \l}, {\sf C}_{w_k \mu}) = \Hom_{R_{\C}}(R_{\mathbf{b}^{\l}}^{\C}, R_{\mathbf{b}^{\mu}}^{\C})$ is given by
      \[
       \{
       1\mapsto x_1^{j_1} \dots x_{n-1}^{j_{n-1}} \mid c_i \leq j_i < b_i^{\mu}
       \}.
      \]
\end{enumerate}
\end{proposition}

% ------------------------------------
\subsection{Illicit morphisms}
% ------------------------------------

Let $W_k \subset S_n$ be the subgroup generated by $s_1, \dots, s_{k-1}$, and $W_k^{\perp}$ the subgroup generated by $s_k,s_{k+1}, \dots, s_{n-1}$.  Let $D'$ be the set of shortest coset representatives for $W_k^{\perp} \setminus S_n$, so that for any $\mu \in D_{n,k}$, we have $\mu,w_k\mu \in D'$ where $w_k$ is defined in Proposition~\ref{prop:SartoriSoergel}, item \eqref{it:SarProp4.5}.

 \begin{definition}
For $\l,\mu \in D_{n,k}$ we say that a morphism ${\sf C}_{w_k\l} \to {\sf C}_{w_k\mu}$ is {\em illicit} if it factors through some ${\sf C}_{y}$, where $y$ is a longest coset representative for $W_k\setminus S_n$ with $y \notin D'$.  Let $\cal{W}_{\l,\mu}$ be the $R_{\C}$-submodule of $\Hom_{R_{\C}}({\sf C}_{w_k\l}, {\sf C}_{w_k \mu})$ consisting of illicit morphisms.
\end{definition}

Sartori's algebras $\cal{A}_{n,k}$ are built from maps between cyclic Soergel modules modulo illicit morphisms.   For our purposes, it suffices to know a few simple illicit morphisms; then the fact that composition with an illicit morphism is illicit will enable us to completely characterize illicit morphisms in Corollary~\ref{cor:W} below.  The following lemma collects Lemmas 4.14-4.16 from \cite{Sar-diagrams}.

\begin{lemma}[\cite{Sar-diagrams}] \label{lem:illicit} \hfill
\begin{enumerate}[(i)]
  \item\label{ill:TooFar} Suppose that $\mu, \l \in D_{n,k}$ are identical except in entries $(j,j+1,j+2)$ where $(\mu_j,\mu_{j+1},\mu_{j+2}) = \wedge \vee \vee$ and $(\l_j,\l_{j+1},\l_{j+2}) = \vee \vee \wedge$.  Then $\cal{W}_{\mu,\l} = \Hom_{R_{\C}}(R_{\mathbf{b}^{\mu}}^{\C}, R_{\mathbf{b}^{\l}}^{\C})$ and $\cal{W}_{\l,\mu} = \Hom_{R_{\C}}(R_{\mathbf{b}^{\l}}^{\C}, R_{\mathbf{b}^{\mu}}^{\C})$.

  \item\label{ill:Uzero} For $\l \in D_{n,k}$ with $\vee_{j+1}^{\l} = \vee_{j}^{\l}+1$, the endomorphism
$(1 \mapsto x_{\vee_j^{\l}})$ of $R_{\mathbf{b}^{\l}}^{\C}$ is illicit.

\item\label{ill:GenInts} For $\l \in D_{n,k}$, the morphisms
\[
 1 \mapsto x_{\vee_j^{\l}} x_{\vee_j^{\l} +1} \dots x_{\vee_{j+1}^{\l} -1}
\]
are illicit endomorphisms of $R_{\mathbf{b}^{\l}}^{\C}$ for each $1 \leq j \leq n-k-1$.
\end{enumerate}
\end{lemma}

% ------------------------------------
\subsection{Definition of Sartori's algebra}
% ------------------------------------

\begin{definition}
Define the Sartori algebra $\cal{A}$ to be the graded $\C$-algebra
\begin{equation}
  \cal{A} = \cal{A}_{n,k} =
\bigoplus_{\l,\mu \in D_{n,k}} \Hom_{R_{\C}}( R_{\mathbf{b}^{\l}}^{\C} , R_{\mathbf{b}^{\mu}}^{\C}) / \cal{W}_{\l,\mu}.
\end{equation}
\end{definition}

Let $\1_{\l}$ denote the identity map on $R_{\mathbf{b}^{\l}}^{\C}$.  The collection $\{ \1_{\l} \mid \l \in D_{n,k}\}$ form a system of mutually orthogonal idempotents on the algebra $\cal{A}$. We can view $\cal{A}$ as a $\C[x_1,\ldots,x_n]$-algebra by sending $x_i$ to $\sum_{\lambda \in D_{n,k}} x_i \1_{\lambda}$. In the subsequent sections we will review Sartori's diagrammatic description of the graded vector spaces $\1_{\mu} \cal{A} \1_{\l}$ and define a version of Sartori's algebra over $\Z$.

% - - - - - - - - - - - - - - - - -
\subsubsection{Connection to BGG category $\cal{O}$ } \label{sec:connectO}
% - - - - - - - - - - - - - - - - -

Let $\Pi=\{\alpha_1,\dots, \alpha_{n-1}\}$ denote the set of simple roots for $\mf{gl}_n$, so that $\alpha_i = \epsilon_i - \epsilon_{i+1}$, with $\epsilon_i = (0,\dots,0,1,0\dots,0)$ the standard basis vectors.

Define $^{\Z}\cal{O}^{\mf{p},\mf{q}{\rm -pres}}$ to be the $\mf{q}$-presentable quotient of the $\mf{p}$-parabolic subcategory of graded category $\cal{O}(\mf{gl}_n)$, where $\mf{q}$, $\mf{p}$ are the standard parabolic subalgebras of $\mf{gl}_n$ with sets of simple roots $\Pi_{\mf{q}}=\{\alpha_1,\dots, \alpha_{k-1}\}$ and $\Pi_{\mf{p}} = \{\alpha_k, \dots, \alpha_{m-1} \}$ (so $\mf{p}, \mf{q}$ have Levi types $(1,\ldots,1,n-k)$ and $(k,1,\ldots,1)$ as in the introduction). For the relevant definitions see \cite[Section 6]{Sar-tensor} and the references therein.

The main result of \cite[Theorem 6.7]{Sar-diagrams} shows that the graded algebra $\cal{A}_{n,k}$ is isomorphic to the endomorphism algebra of a minimal projective generator of the block $^{\Z}\cal{O}^{\mf{p},\mf{q}{\rm -pres}}_0$ of $^{\Z}\cal{O}^{\mf{p},\mf{q}{\rm -pres}}$ containing the trivial representation.  In particular, there is a graded equivalence
\begin{equation}
  \cal{A}_{n,k}{\rm -gmod} \cong \;  {}^{\Z}\cal{O}^{\mf{p},\mf{q}{\rm -pres}}_0,
\end{equation}
where $\cal{A}_{n,k}{\rm -gmod}$ is the category of finite-dimensional graded $\cal{A}_{n,k}$-modules, so that the Sartori algebra provides a combinatorial model for $^{\Z}\cal{O}_0^{\mf{p},\mf{q}{\rm -pres}}$.

%-------------------------------------
\subsection{Anti-automorphism} \label{sec:Sart-anti}
%-------------------------------------

There is a symmetry between maps of Soergel modules defined for any $\l,\mu \in D_{n,k}$ by
\begin{align}
 \Hom_{R_{\C}}( R_{\mathbf{b}^{\l}}^{\C}, R_{\mathbf{b}^{\mu}}^{\C})
 & \longrightarrow \Hom_{R_{\C}}( R_{\mathbf{b}^{\mu}}^{\C}, R_{\mathbf{b}^{\l}}^{\C} )
 \\ \nn
 (1 \mapsto p) &\longmapsto (1 \mapsto x^{\mathbf{b}^{\l} - \mathbf{b}^{\mu}} p)
\end{align}
with $x^{\mathbf{b}^{\l} - \mathbf{b}^{\mu}}:= x_1^{b_1^{\l} - b_1^{\mu}} \dots x_n^{b_n^{\l} - b_n^{\mu}}$ (the product of this expression with $p$ will have nonnegative exponents). By \cite[Lemma 5.14]{Sar-diagrams} this map sends $\W_{\l,\mu}$ to $\W_{\mu,\l}$ and thus extends to an anti-automorphism $\psi_S \maps \cal{A}_{n,k} \longrightarrow \cal{A}_{n,k}$ (Sartori refers to $\psi_S$ as $\star$).

% ====================================================
\section{A surjective homomorphism from Ozsv{\'a}th--Szab{\'o} to Sartori}\label{sec:Homomorphism}
% ====================================================

 If $\Bbbk$ is a field, let $\B^{\Bbbk}_l(n,k)$ denote $\B_l(n,k) \otimes \Bbbk$. View $\cal{A}_{n,k}$ as a $\C[U_1,\ldots,U_n]$-algebra by relabeling $x_i$ as $U_i$ in the $\C[x_1,\ldots,x_n]$-algebra structure.
\begin{proposition}\label{prop:BigStepXi}
Let $\mu=\mu^{\x}$, $\l=\mu^{\y}$ denote $\vee\wedge$-sequences in $D_{n,k}$ associated to left I-states $\x, \y \in V_{l}(n,k)$ as in Remark~\ref{rem:OSzSarStates}.
The map $\Xi\co \B_{l}^{\C}(n,k)\to \mathcal{A}_{n,k}$ sending $\Ib_{\x}$ to $\1_{\mu}$, sending $f_{\x,\y}$ to zero if $\x,\y$ are too far, and otherwise sending $f_{\x,\y}$ to (the equivalence class of) the morphism of Soergel modules
\begin{align*}
 \; \Xi_{f_{\x,\y}}:= \Xi(f_{\x,\y}) \maps R_{\bf{b}^{\l}}^{\C} &\to R_{\bf{b}^{\mu}}^{\C} \\
1 \quad &\mapsto x_1^{c_1} \cdots x_n^{c_n}
\end{align*}
where $c_i = \max\left(v_i^{\x} - v_i^{\y},0\right) = \max\left(b_i^{\mu} - b_i^{\l},0\right)$, extended linearly over $\C[U_1,\ldots,U_n]$, is a well-defined surjective homomorphism of $\Z$-graded $\C[U_1,\ldots,U_n]$-algebras.
\end{proposition}

Note that if $\x,\y,\z$ are pairwise not too far, then  $\Xi_{f_{\x,\y}} \Xi_{f_{\y,\z}} = \prod_{i=1}^n U_i^{g_i^{\x,\y,\z}} \Xi_{f_{\x,\z}}$; this identity follows from the relationship $\max(a,b) = \frac{a+b+|a-b|}{2}$ between maxima and absolute values. However, to prove Proposition~\ref{prop:BigStepXi} along these lines, one would also need to consider what happens when some pairs among $\x,\y,\z$ are too far (as well as present $\B_l(n,k)$ as a quotient of a truncation $\B_{l,0}(n,k)$ of $\B_0(n,k)$ and show that the quotient relations are satisfied).

Instead, to show that $\Xi$ is well-defined and prove Proposition~\ref{prop:BigStepXi}, we will make use of the small-step quiver description of $\B_l(n,k)$. The small-step quiver generators $R_i$, $L_i$, and $U_i$ can be viewed as instances of $f_{\x,\y}$ or $U_i f_{\x,\x}$ for certain $\x$ and $\y$; below we describe where $\Xi$ sends these generators.

\begin{lemma}\label{lem:TildeXiLemma}
For a small-step quiver generator with label $R_i$, left idempotent $\x$, and right idempotent $\y$, we have
\begin{align*}
\Xi_{R_i}\co R_{\bf{b}^{\y}}^{\C} &\to R_{\bf{b}^{\x}}^{\C} \\
1 \quad &\mapsto \quad 1.
\end{align*}
For a small-step quiver generator with label $L_i$, left idempotent $\x$, and right idempotent $\y$, we have
\begin{align*}
\Xi_{L_i}\co R_{\bf{b}^{\y}}^{\C} &\to R_{\bf{b}^{\x}}^{\C} \\
1 \quad &\mapsto \quad x_i.
\end{align*}
For a small-step quiver generator with label $U_i$ and left and right idempotent $\x$, we have
\begin{align*}
\Xi_{U_i}\co R_{\bf{b}^{\x}}^{\C} &\to R_{\bf{b}^{\x}}^{\C} \\
1 \quad &\mapsto \quad x_i.
\end{align*}
\end{lemma}

\begin{proof}
If a small-step generator with label $R_i$ exists with left idempotent $\Ib_{\x}$ and right idempotent $\Ib_{\y}$, then we have $\x \cap \{i-1,i\} = \{i-1\}$ and $\y = (\x \setminus \{i-1\}) \cup \{i\}$. It follows that $v^{\x} - v^{\y}$ is the element of $\Z^n$ with $-1$ in entry $i$ and zero in all other entries, so we have $c_j = 0$ for all $j$ in Proposition~\ref{prop:BigStepXi}. For small-step generators labeled $L_i$, the argument is similar, except that $v^{\x} - v^{\y}$ has $1$ in entry $i$. It follows that $c_j = 0$ for $j \neq i$ and $c_i = 1$ in Proposition~\ref{prop:BigStepXi}. Finally, for small-step generators labeled $U_i$, the claim follows because these generators can be written as $U_i f_{\x,\x}$, $\Xi$ was defined to be $\C[U_1,\ldots,U_n]$-linear, and the action of $\C[U_1,\ldots,U_n]$ on the Soergel module morphism space has $U_i$ acting as multiplication by $x_i$.
\end{proof}

\begin{lemma}\label{lem:TildeXiWellDef}
Extending the values of $\Xi$ from Lemma~\ref{lem:TildeXiLemma} multiplicatively, we get a well-defined $\C$-algebra homomorphism
\[
\widetilde{\Xi}\co \B^{\C}_l(n,k) \to \cal{A}_{n,k}.
\]
\end{lemma}

\begin{proof}
We need to check that the relations of Proposition~\ref{prop:MMWThm1} are satisfied. The relations \eqref{Brel:Ucommute} follow because Soergel module morphisms are assumed to be linear over $R_{\C} = \C[x_1,\ldots,x_n]$. The relations \eqref{Brel:loop} and \eqref{Brel:RLcommute} follow from the explicit formulas of Lemma~\ref{lem:TildeXiLemma}. The relations \eqref{Brel:twoline} and \eqref{Brel:Uzero} follow from Lemma~\ref{lem:illicit}, items \eqref{ill:TooFar} and \eqref{ill:Uzero} respectively.
\end{proof}

\begin{lemma}\label{lem:TildeXiOnFxy}
The ring homomorphism $\widetilde{\Xi}$ from Lemma~\ref{lem:TildeXiWellDef} is linear over $\C[U_1,\ldots,U_n]$ and satisfies $\widetilde{\Xi}(f_{\x,\y}) = \Xi(f_{\x,\y})$ for all big-step generators $f_{\x,\y}$ of $\B^{\C}_l(n,k)$.
\end{lemma}

\begin{proof}
Since $\C[U_1,\ldots,U_n]$ acts on the small-step description of $\B^{\C}_l(n,k)$ via the small-step generators with label $U_i$, $\widetilde{\Xi}$ sends these to Soergel module endomorphisms that multiply by $x_i$, and $\C[U_1,\ldots,U_n]$ acts on $\cal{A}_{n,k}$ by having $U_i$ multiply by $x_i$, the map $\widetilde{\Xi}$ is $\C[U_1,\ldots,U_n]$-linear.

To show that $\widetilde{\Xi}$ maps $f_{\x,\y}$ as claimed, first note that if $\x$ and $\y$ are too far then $f_{\x,\y}$ is zero in $\B^{\C}_l(n,k)$, so $\widetilde{\Xi}$ maps it to $0 = \Xi(f_{\x,\y})$ in $\cal{A}_{n,k}$.

When $\x$ and $\y$ are not too far, we will proceed by induction on $k - |\x \cap \y|$ using a small-step path $\gamma_{\x,\y}$ representing $f_{\x,\y}$ in $\B_l(n,k)$ under the isomorphism of Proposition~\ref{prop:MMWThm1}.  We make use of a recursive definition of $\gamma_{\x,\y}$ from \cite[Definition 2.28]{MMW1}.

If $k - |\x \cap \y|$ is zero, then $\x = \y$ and the claim follows. Assume that $\x \neq \y$ and that the claim holds for all $(\x,\y)$ with $k - |\x' \cap \y'| < k - |\x \cap \y|$. We consider two cases:
\begin{itemize}
\item If $x_a < y_a$ for some $a$, let $a$ be the maximal such index. Since $\x$ and $\y$ are not too far, we must have $y_a = x_a + 1$, and since $a$ is maximal, we must have $x_a + 1 \notin \x$. Let $\x' = (\x \setminus \{x_a\}) \cup \{x_a + 1\}$; by construction, we have $\gamma_{\x,\y} = R_{x_a + 1} \gamma_{\x',\y}$. Thus, by the induction hypothesis,
\begin{align*}
\widetilde{\Xi}(\gamma_{\x,\y}) &= \widetilde{\Xi}(R_{x_a + 1}) \widetilde{\Xi}(\gamma_{\x',\y}) \\
&= (1 \mapsto 1) \circ (1 \mapsto x_1^{c'_1} \cdots x_n^{c'_n})
\end{align*}
where $c'_i = \max(v_i^{\x'} - v_i^{\y},0)$. For all $i$ except $i = x_a + 1$, we have $v_i^{\x'} = v_i^{\x}$, and we have $v_{x_a+1}^{\x'} = v_{x_a+1}^{\x} + 1$. Since $v_{x_a+1}^{\x} - v_{x_a+1}^{\y} = -1$, we have $v_{x_a+1}^{\x'} - v_{x_a+1}^{\y} = 0$. We see that for all $i$, $c'_i = c_i = \max(v_i^{\x} - v_i^{\y},0)$, so that
\[
\widetilde{\Xi}(\gamma_{\x,\y}) = (1 \mapsto x_1^{c_1} \cdots x_n^{c_n})
\]
as claimed.

\item If $x_a \geq y_a$ for all $a$, then $x_a > y_a$ for some minimal $a$ because $\x \neq \y$. As above, $y_a = x_a - 1$ and $x_a - 1 \notin \x$. Letting $\x' = (\x \setminus \{x_a\}) \cup \{x_a - 1\}$, we have $\gamma_{\x,\y} = L_{x_a} \gamma_{\x',\y}$. We get
\begin{align*}
\widetilde{\Xi}(\gamma_{\x,\y}) &= \widetilde{\Xi}(L_{x_a}) \widetilde{\Xi}(\gamma_{\x',\y}) \\
&= (1 \mapsto x_{x_a}) \circ (1 \mapsto x_1^{c'_1} \cdots x_n^{c'_n})
\end{align*}
where $c'_i = \max(v_i^{\x'} - v_i^{\y},0)$. For all $i$ except $i = x_a$, we have $v_i^{\x'} = v_i^{\x}$, and we have $v_{x_a}^{\x'} = v_{x_a}^{\x} - 1$. Since $v_{x_a}^{\x} - v_{x_a}^{\y} = 1$, we have $v_{x_a}^{\x'} - v_{x_a}^{\y} = 0$. We see that for $i \neq x_a$, we have $c'_i = c_i = \max(v_i^{\x} - v_i^{\y},0)$, while $c'_{x_a} = c_{x_a} - 1  (= 0)$. Again, it follows that for all $i$, we have
\[
\widetilde{\Xi}(\gamma_{\x,\y}) = (1 \mapsto x_1^{c_1} \cdots x_n^{c_n})
\]
as claimed.
\end{itemize}
\end{proof}

\begin{proof}[Proof of Proposition~\ref{prop:BigStepXi}]
Since $\B^{\C}_l(n,k)$ is generated over $\C[U_1,\ldots,U_n]$ by the elements $f_{\x,\y}$, and both $\Xi$ and $\widetilde{\Xi}$ are $\C[U_1,\ldots,U_n]$-linear, Lemma~\ref{lem:TildeXiOnFxy} implies that $\Xi = \widetilde{\Xi}$. Thus, $\Xi$ is a well-defined algebra homomorphism; surjectivity of $\Xi$ follows from Proposition~\ref{prop:SartoriSoergel}, item \eqref{it:SoergelMorphismBasis}.
\end{proof}

\begin{proposition}
The surjective homomorphism $\Xi\co \B^{\C}_{l}(n,k) \to \mathcal{A}_{n,k}$ sends the anti-autormorphism $\psi_{OSz}$ from Section~\ref{subsec:PsiOS} to the anti-automorphism $\psi_{S}$ from Section~\ref{sec:Sart-anti}.
\end{proposition}

\begin{proof}
This is easiest to see in the small-step description of the homomorphism using the mappings from Lemma~\ref{lem:TildeXiLemma}.  Then it is clear that $\psi_{OSz}$, which swaps the roles of $R_i$ and $L_i$, agrees with the definition of $\psi_{S}$.
\end{proof}

% ====================================================
\section{Characterizing illicit morphisms via generating intervals}
% ====================================================

% ------------------------------------
\subsection{Generating interval and the \texorpdfstring{$\Wa$}{W-alpha} submodule}
% ------------------------------------

We now begin to work with Sartori-style constructions over $\Z$ rather than $\C$.

\begin{definition}  \label{def:newW}
For $\mu,\l \in D_{n,k}$, let $\Wa_{\l,\mu} \subset \Hom_R(R_{\mathbf{b}^{\l}}, R_{\mathbf{b}^{\mu}} )$  be the $R$-submodule  defined by:
\begin{enumerate}[(i)]
  \item if $\l$ and $\mu$ are too far in the sense of Definition~\ref{def:Sart-toofar}, then $$\Wa_{\l,\mu} = \Hom_R(R_{\mathbf{b}^{\l}}, R_{\mathbf{b}^{\mu}} ) ;$$

  \item otherwise, define $\Wa_{\l,\mu}$ to be the submodule generated by the morphisms
\begin{equation} \label{eq:Wa-illicit}
   1 \mapsto \left(x_{\alpha(j)}, x_{\alpha(j) +1 } \dots x_{\beta(j)} \right)
   \left(
   x_1^{c_1} \dots x_n^{c_n}
   \right)  \qquad \text{for $1 \leq j \leq n-k$},
\end{equation}
where $c_i = \max\left( b_i^{\mu} - b_i^{\l} , 0\right)$ and
\begin{equation} \label{eq-Wa-alpha}
\alpha(j) = \max\left( \vee_j^{\l}, \vee_j^{\mu} \right), \qquad
  \beta(j) =
  \left\{
    \begin{array}{ll}
      \min\left( \vee_{j+1}^{\mu} \vee_{j+1}^{\l} \right) -1, & \hbox{if $j < n-k$,} \\
      n, & \hbox{if $j=n-k$.}
    \end{array}
  \right.
\end{equation}
\end{enumerate}

\end{definition}

\begin{proposition}[cf. Theorem 4.17 \cite{Sar-diagrams} ] \label{prop:WaW}
For all $\mu,\l \in D_{n,k}$, we have $\Wa_{\l,\mu} \otimes \C \subset \cal{W}_{\l,\mu}$, so that the submodule $\Wa_{\l,\mu} \otimes \C$ contains only illicit morphisms.
\end{proposition}

\begin{proof}
Identify $\x$ with $\mu$ and $\y$ with $\l$. Observe that
$\vee_{j}^{\mu} = h_j^{\x}+1$.  By Lemma~\ref{lem:SartoriStyleGenInts}, the element of $\cal{A}_{n,k}$ represented by each generator of $\Wa_{\l,\mu} \otimes \C$ is the image under the homomorphism $\Xi$ of $p_G(U_1,\ldots,U_n) f_{\x,\y}$ for some generating interval $G = [j+1,\ldots,j+l]$ between $\x$ and $\y$ as defined in Definition~\ref{def:generatingint}. These elements $p_G(U_1,\ldots,U_n) f_{\x,\y}$ are zero in $\B_l(n,k)$, so $\Xi$ sends them to zero in $\cal{A}_{n,k}$. Thus, each generator of $\Wa_{\l,\mu} \otimes \C$ represents zero in $\cal{A}_{n,k}$ and is hence illicit.
\end{proof}

The statement above differs from \cite[Theorem 4.17]{Sar-diagrams} as we explain in the next section.

% ------------------------------------------------------
\subsection{Comparison with Sartori's Theorem 4.17 }
% ------------------------------------------------------

In \cite[Theorem 4.17]{Sar-diagrams}, Sartori defines a collection of illicit morphisms $\tilde{\cal{W}}_{\l,\mu}$ expressed in our terminology as follows.
For $\mu,\l \in D_{n,k}$, let $\tilde{\cal{W}}_{\l,\mu} \subset \Hom_R(R_{\mathbf{b}^{\l}}, R_{\mathbf{b}^{\mu}} )$ be defined as in Definition~\ref{def:newW} with $\alpha(j):=\vee^{\l}_j$ in \eqref{eq:Wa-illicit}.

Whenever $\l$ and $\mu$ are not too far and $\vee_j^{\mu} > \vee_j^{\l}$ for some $1 \leq j \leq n-k$, this definition differs from Definition~\ref{def:newW}.  Below we give an example showing that $\Wa_{\l,\mu} \nsubseteq\tilde{W}_{\l,\mu}$ (the same is true after complexification), so that $\tilde{W}_{\l,\mu}$ cannot be equivalent to the submodule of illicit morphisms.

\begin{example}[Comparing $\tilde{W}$ and $\Wa$]\label{ex:TildeWCounterex}
Let $\l= \vee \wedge \wedge \vee$ and $\mu = \wedge \vee \vee \wedge$. Then $b^{\l} = (3,2,1,1)$ and $b^{\mu} = (2,2,2,1)$. The submodule $\tilde{ W }_{\l,\mu}$ is generated by $(x_1x_2) (x_3)$ (from $j=1$) and $(x_4)(x_3)$ (from $j=2$). However, $\Wa_{\l,\mu}$ is generated by $(x_2) (x_3)$ and $(x_4)(x_3)$.

A $\Z$ analogue of Proposition~\ref{prop:SartoriSoergel}, item~\eqref{it:SoergelMorphismBasis} implies that a $\Z$-basis for the space of maps from $R_{\bf{b}^{\l}}$ to $R_{\bf{b}^{\mu}}$ is given by the set
\[
\{ x_1^{j_1} x_{2}^{j_2} x_3^{j_3}  \mid  0 \leq j_1 \leq 1, \, 0 \leq j_2 \leq 1, \, j_3 = 1 \} = \{x_3, \, x_1 x_3, \, x_2 x_3, \, x_1 x_2 x_3\}
\]
(since $c_3=1 \leq j_3 < b_3^{\mu} = 2$).

Notice that the $j=1$ generators of both $\tilde{W}_{\l,\mu}$ and $\Wa_{\l,\mu}$ are both elements of the above basis. The $j=2$ generator $x_3 x_4$ of these ideals is redundant in both cases since
\[
I_{\mathbf{b}^{\mu}} = \langle h_2(x_1) , h_2(x_1,x_2), h_2(x_1,x_2,x_3), h_1(x_1, x_2, x_3, x_4) \rangle
\]
so
\[
x_3 x_4 = -x_1 x_3 - x_2 x_3 - x_3^2 = x_1^2 + x_2^2 + x_1 x_2 =  h_2(x_1,x_2) \in I_{\mathbf{b}^{\mu}}.
\]
 It follows that $\{x_3, x_1 x_3, x_2 x_3\}$ is a basis for $\Hom_{R_{\C}}(R_{\bf{b}^{\l}}, R_{\bf{b}^{\mu}}) / \tilde{ W }_{\l,\mu}$; in particular, $x_2 x_3 \in \Wa_{\l,\mu} \setminus \tilde{W}_{\l,\mu}$.
\end{example}

% ------------------------------------
\subsection{Fork diagrams and the dimension of homs mod illicits}\label{subsec:fork-illicit}
% ------------------------------------
% - - - - - - - - - - - - - - - - -
\subsubsection{Oriented fork diagrams}
% - - - - - - - - - - - - - - - - -

Here we recall the notion of (enhanced) fork diagram from \cite[Section 5.1]{Sar-diagrams}. An {\em $m$-fork} is a  tree with a single root and valency $m$, with 1-forks called {\em rays}.  Let ${\bf H}_-$ (resp. ${\bf H}_+$) denote the lower (resp. upper) half plane.  A {\em lower fork diagram} is a collection of forks in ${\bf H}_-$ such the leaves of each $m$-fork are $m$ distinct points on the boundary $\partial {\bf H}_-$ of ${\bf H}_-$.  Upper fork diagrams are defined analogously.
Below is an example of an $m$-fork for $m=5$ and a lower and upper fork diagram.
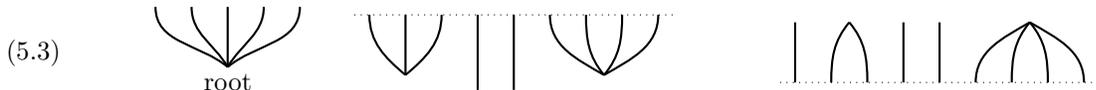
\begin{equation} \label{ex:ul-fork}
\hackcenter{\begin{tikzpicture}[scale=0.8]
% UPPER FORK
    \draw[thick, -] (-1.2,-.25) to[out=-90, in=140] (0,-1.25);
    \draw[thick, -] (-.6,-.25) to[out=-90, in=120] (0,-1.25);
    \draw[thick, -] (0,-.25) to  (0,-1.25);
    \draw[thick, -] (.6,-.25) to[out=-90, in=60] (0,-1.25);
    \draw[thick, -] (1.2,-.25) to[out=-90, in=40] (0,-1.25);
    \node at (0,-1.5) {root};
\end{tikzpicture}}
\qquad
\hackcenter{\begin{tikzpicture}[scale=0.8]
% LOWER FORK
    \draw[dotted ] (-.25,-.25) to (5.05,-.25);
    \draw[thick, -] (0,-.25) to[out=-90, in=140]  (.6,-1.25);
    \draw[thick, -] (.6,-.25) to (.6,-1.25);
    \draw[thick, -] (1.2,-.25) to[out=-90, in=40] (.6,-1.25);
    \draw[thick, -] (1.8,-.25) -- (1.8,-1.5);
    \draw[thick, -] (2.4,-.25) -- (2.4,-1.5);
    \draw[thick, -] (3,-.25) to[out=-90, in=150] (3.9,-1.25);
    \draw[thick, -] (3.6,-.25) to[out=-90, in=130] (3.9,-1.25);
    \draw[thick, -] (4.2,-.25) to[out=-90, in=50] (3.9,-1.25);
    \draw[thick, -] (4.8,-.25) to[out=-90, in=30] (3.9,-1.25);
\end{tikzpicture}}
\qquad \qquad
\hackcenter{\begin{tikzpicture}[scale=0.8]
% UPPER FORK
	\begin{scope}[yscale=-1,xscale=1]
% LOWER FORK
 \draw[dotted ] (-.25,-.25) to (5.05,-.25);
    \draw[thick, -] (0,-.25) to  (0,-1.25);
    \draw[thick, -] (.6,-.25) to[out=-90, in=130] (.9,-1.25);
    \draw[thick, -] (1.2,-.25) to[out=-90, in=50] (.9,-1.25);
    \draw[thick, -] (1.8,-.25) -- (1.8,-1.25);
    \draw[thick, -] (2.4,-.25) -- (2.4,-1.25);
    \draw[thick, -] (3,-.25) to[out=-90, in=150] (3.9,-1.25);
    \draw[thick, -] (3.6,-.25) to[out=-90, in=130] (3.9,-1.25);
    \draw[thick, -] (4.2,-.25) to[out=-90, in=50] (3.9,-1.25);
    \draw[thick, -] (4.8,-.25) to[out=-90, in=30] (3.9,-1.25);
	\end{scope}
\end{tikzpicture}}
\end{equation}

If $c$ is a lower fork diagram with $|c \cap \partial {\bf H}_{-}| =n$ and $\l \in D_{n,k}$ then $c\l$ is a {\em unenhanced oriented lower fork diagram} if:
\begin{itemize}
 \item each $m$-fork for $m\geq 2$ is labeled\footnote{In \cite{Sar-diagrams} there is a typo indicating $m\geq 1$.} with exactly one $\vee$ and $m-1$ $\wedge$'s;
 \item the diagram begins at the left with a (possibly empty) sequence of rays labelled $\wedge$, and there are no other rays labelled $\wedge$ in c.
\end{itemize}
For example, in the labeled fork diagrams below
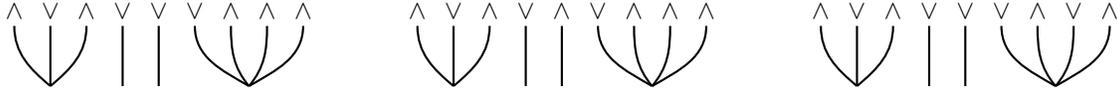
\[
\hackcenter{\begin{tikzpicture}[scale=0.8]
% Lower FORK
	\begin{scope}[yscale=-1,xscale=1]
    \draw[thick, -] (0,.25) to[out=90, in=220] (.6,1.25);
    \draw[thick, -] (.6,.25) -- (.6,1.25);
    \draw[thick, -] (1.2,.25) to[out=90, in=-40] (.6,1.25);
    \draw[thick, -] (1.8,.25) -- (1.8,1.25);
    \draw[thick, -] (2.4,.25) -- (2.4,1.25);
    \draw[thick, -] (3,.25) to[out=90, in=210] (3.9,1.25);
    \draw[thick, -] (3.6,.25) to[out=90, in=230] (3.9,1.25);
    \draw[thick, -] (4.2,.25) to[out=90, in=-50] (3.9,1.25);
    \draw[thick, -] (4.8,.25) to[out=90, in=-30] (3.9,1.25);
% Weight in middle
    \node at (0,0) {  $\wedge$};
    \node at (.6,0) {  $\vee$};
    \node at (1.2,0) {  $\wedge$};
    \node at (1.8,0) {  $\vee$};
    \node at (2.4,0) {  $\vee$};
    \node at (3,0) {  $\vee$};
    \node at (3.6,0) {  $\wedge$};
    \node at (4.2,0) {  $\wedge$};
    \node at (4.8,0) {  $\wedge$};
	\end{scope}
\end{tikzpicture}}
\qquad\quad
\hackcenter{\begin{tikzpicture}[scale=0.8]
% Lower FORK
	\begin{scope}[yscale=-1,xscale=1]
    \draw[thick, -] (0,.25) to[out=90, in=220] (.6,1.25);
    \draw[thick, -] (.6,.25) -- (.6,1.25);
    \draw[thick, -] (1.2,.25) to[out=90, in=-40] (.6,1.25);
    \draw[thick, -] (1.8,.25) -- (1.8,1.25);
    \draw[thick, -] (2.4,.25) -- (2.4,1.25);
    \draw[thick, -] (3,.25) to[out=90, in=210] (3.9,1.25);
    \draw[thick, -] (3.6,.25) to[out=90, in=230] (3.9,1.25);
    \draw[thick, -] (4.2,.25) to[out=90, in=-50] (3.9,1.25);
    \draw[thick, -] (4.8,.25) to[out=90, in=-30] (3.9,1.25);
% Weight in middle
    \node at (0,0) {  $\wedge$};
    \node at (.6,0) {  $\vee$};
    \node at (1.2,0) {  $\wedge$};
    \node at (1.8,0) {  $\vee$};
    \node at (2.4,0) {  $\wedge$};
    \node at (3,0) {  $\vee$};
    \node at (3.6,0) {  $\wedge$};
    \node at (4.2,0) {  $\wedge$};
    \node at (4.8,0) {  $\wedge$};
	\end{scope}
\end{tikzpicture}}
\quad\qquad
\hackcenter{\begin{tikzpicture}[scale=0.8]
% Lower FORK
	\begin{scope}[yscale=-1,xscale=1]
    \draw[thick, -] (0,.25) to[out=90, in=220] (.6,1.25);
    \draw[thick, -] (.6,.25) -- (.6,1.25);
    \draw[thick, -] (1.2,.25) to[out=90, in=-40] (.6,1.25);
    \draw[thick, -] (1.8,.25) -- (1.8,1.25);
    \draw[thick, -] (2.4,.25) -- (2.4,1.25);
    \draw[thick, -] (3,.25) to[out=90, in=210] (3.9,1.25);
    \draw[thick, -] (3.6,.25) to[out=90, in=230] (3.9,1.25);
    \draw[thick, -] (4.2,.25) to[out=90, in=-50] (3.9,1.25);
    \draw[thick, -] (4.8,.25) to[out=90, in=-30] (3.9,1.25);
% Weight in middle
    \node at (0,0) {  $\wedge$};
    \node at (.6,0) {  $\vee$};
    \node at (1.2,0) {  $\wedge$};
    \node at (1.8,0) {  $\vee$};
    \node at (2.4,0) {  $\vee$};
    \node at (3,0) {  $\vee$};
    \node at (3.6,0) {  $\wedge$};
    \node at (4.2,0) {  $\vee$};
    \node at (4.8,0) {  $\wedge$};
	\end{scope}
\end{tikzpicture}}
\]
only the first is an oriented fork diagram. The second has a $\wedge$ labelled ray not appearing at the beginning of the diagram, while the third is disallowed because the 4-fork has two strands labelled $\vee$.

 A {\em enhanced oriented lower fork diagram} $c\l^{\sigma}$ is a unenhanced lower fork diagram equipped with a bijection $\sigma$ between the vertices labelled $\wedge$ in $\l$ and the set $\{1,\dots, k\}$.  Unenhanced and enhanced upper fork diagrams are defined analogously.

By a {\em fork diagram} we mean a diagram of the form $ab$ obtained by gluing a lower fork diagram $a$ underneath an upper fork diagram $b$ with compatible endpoints on the boundaries.  An {\em unenhanced oriented fork diagram} is a fork diagram $a\l b$ in which both $a\l$ is an oriented lower fork diagram and $\l b$ is an oriented upper fork diagram.  Some examples of oriented fork diagrams are given below.
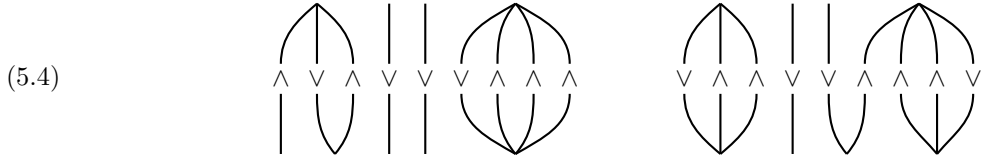
\begin{equation}
\hackcenter{\begin{tikzpicture}[scale=0.8, yscale =-1.0]
% UPPER FORK
	\begin{scope}[yscale=-1,xscale=1]
    \draw[thick, -] (0,.25) to[out=90, in=220] (.6,1.25);
    \draw[thick, -] (.6,.25) -- (.6,1.25);
    \draw[thick, -] (1.2,.25) to[out=90, in=-40] (.6,1.25);
    \draw[thick, -] (1.8,.25) -- (1.8,1.25);
    \draw[thick, -] (2.4,.25) -- (2.4,1.25);
    \draw[thick, -] (3,.25) to[out=90, in=210] (3.9,1.25);
    \draw[thick, -] (3.6,.25) to[out=90, in=230] (3.9,1.25);
    \draw[thick, -] (4.2,.25) to[out=90, in=-50] (3.9,1.25);
    \draw[thick, -] (4.8,.25) to[out=90, in=-30] (3.9,1.25);
% LOWER FORK
    \draw[thick, -] (0,-.25) to  (0,-1.25);
    \draw[thick, -] (.6,-.25) to[out=-90, in=130] (.9,-1.25);
    \draw[thick, -] (1.2,-.25) to[out=-90, in=50] (.9,-1.25);
    \draw[thick, -] (1.8,-.25) -- (1.8,-1.25);
    \draw[thick, -] (2.4,-.25) -- (2.4,-1.25);
    \draw[thick, -] (3,-.25) to[out=-90, in=150] (3.9,-1.25);
    \draw[thick, -] (3.6,-.25) to[out=-90, in=130] (3.9,-1.25);
    \draw[thick, -] (4.2,-.25) to[out=-90, in=50] (3.9,-1.25);
    \draw[thick, -] (4.8,-.25) to[out=-90, in=30] (3.9,-1.25);
% Weight in middle
    \node at (0,0) {  $\wedge$};
    \node at (.6,0) {  $\vee$};
    \node at (1.2,0) {  $\wedge$};
    \node at (1.8,0) {  $\vee$};
    \node at (2.4,0) {  $\vee$};
    \node at (3,0) {  $\vee$};
    \node at (3.6,0) {  $\wedge$};
    \node at (4.2,0) {  $\wedge$};
    \node at (4.8,0) {  $\wedge$};
	\end{scope}
\end{tikzpicture}}
\qquad \quad
\hackcenter{\begin{tikzpicture}[scale=0.8, yscale=-1.0]
% UPPER FORK
	\begin{scope}[yscale=-1,xscale=1]
    \draw[thick, -] (0,.25) to[out=90, in=220] (.6,1.25);
    \draw[thick, -] (.6,.25) -- (.6,1.25);
    \draw[thick, -] (1.2,.25) to[out=90, in=-40] (.6,1.25);
    \draw[thick, -] (1.8,.25) -- (1.8,1.25);
    \draw[thick, -] (2.4,.25) -- (2.4,1.25);
    \draw[thick, -] (3,.25) to[out=90, in=210] (3.9,1.25);
    \draw[thick, -] (3.6,.25) to[out=90, in=230] (3.9,1.25);
    \draw[thick, -] (4.2,.25) to[out=90, in=-50] (3.9,1.25);
    \draw[thick, -] (4.8,.25) to[out=90, in=-30] (3.9,1.25);
% LOWER FORK
    \draw[thick, -] (0,-.25) to[out=-90, in=140] (.6,-1.25);
    \draw[thick, -] (.6,-.25) -- (.6,-1.25);
    \draw[thick, -] (1.2,-.25) to[out=-90, in=40] (.6,-1.25);
    \draw[thick, -] (1.8,-.25) -- (1.8,-1.25);
    \draw[thick, -] (2.4,-.25) to[out=-90, in=130](2.7,-1.25);
    \draw[thick, -] (3,-.25) to[out=-90, in=50] (2.7,-1.25);
    \draw[thick, -] (3.6,-.25) to[out=-90, in=130] (4.2,-1.25);
    \draw[thick, -] (4.2,-.25) to  (4.2,-1.25);
    \draw[thick, -] (4.8,-.25) to[out=-90, in=50] (4.2,-1.25);
% Weight in middle
    \node at (0,0) {  $\vee$};
    \node at (.6,0) {  $\wedge$};
    \node at (1.2,0) {  $\wedge$};
    \node at (1.8,0) {  $\vee$};
    \node at (2.4,0) {  $\vee$};
    \node at (3,0) {  $\wedge$};
    \node at (3.6,0) {  $\wedge$};
    \node at (4.2,0) {  $\wedge$};
    \node at (4.8,0) {  $\vee$};
	\end{scope}
\end{tikzpicture}}
\label{eq:example-fork}
\end{equation}

The degree of an oriented $m$-fork for $m \geq 2$ is defined to be $(i-i_0)$ where $i$ is the index of the unique $\vee$ labelling the fork  and $i_0$ is the leftmost index in the fork; $1$-forks have degree zero. The degree of an  upper or lower oriented fork diagram is the sum of the degrees of its forks; the degree of an oriented fork diagram is the degree of its upper part plus the degree of its lower part.   Define the degree of a permutation $\sigma \in S_n$ as $\deg(\sigma) = 2\ell(\sigma)$. Then the degree of an enhanced oriented fork diagram is given by
\begin{equation}
\deg(a \l^{\sigma} b) = \deg(a\l) + \deg(\l b) + 2\ell(\sigma).
\end{equation}
For example, the first fork diagram in \eqref{eq:example-fork} has a degree 1 upper fork diagram and a degree zero lower fork diagram for a total degree of 1.  In the second example, the upper fork diagram has degree $0+3$ while the lower fork diagram has degree 0 + 0 + 2 for a total degree of 5.

For each sequence $\l$, we denote by $\und{\l}$ the unique lower fork diagram such that $\und{\l} \l^e$ is an oriented lower fork diagram of degree zero.  In other words, $\und{\l}$ is the fork diagram where each $\vee$ in $\l$ is the first vertex of an $m$-fork for some $m$.  We write $\bar{\lambda}$ for the unique degree zero upper fork diagram given by reflecting $\und{\l}$ across the horizontal axis.  For example,  the lower fork diagram in \eqref{ex:ul-fork} is $\und{\mu}$ for $\mu= \vee \wedge \wedge \vee \vee \vee \wedge \wedge \wedge$ and the upper fork diagram in \eqref{ex:ul-fork} is $\bar{\l}$ for $\l = \wedge \vee \wedge \vee \vee \vee \wedge \wedge \wedge$. Note that the same lower (resp. upper) fork diagram can correspond to $\wedge \vee$ sequences in $D_{n,k}$ for different $k$.  For example, the upper fork diagram in \eqref{ex:ul-fork} can also be identified with $\bar{\l}$ for $\l = \vee \vee \wedge \vee \vee \vee \wedge \wedge \wedge$.

\begin{definition}
Given two sequences $\mu,\l \in D_{n,k}$, define $\cal{Z}_{\mu,\l}$ to be the graded free abelian group with homogeneous basis given by
\[
 \left\{
 \und{\mu} \eta^{\sigma} \bar{\l} \mid \und{\mu} \eta^{\sigma} \bar{\l} \text{ is an oriented enhanced fork diagram}
\right\}.
\]
\end{definition}

% - - - - - - - - - - - - - - - - -
\subsubsection{Counting fork diagrams}
% - - - - - - - - - - - - - - - - -

For $\mu,\l \in D_{n,k}$ that are not too far, we would like an explicit formula for the  graded  rank of $\cal{Z}_{\mu,\l}$ (if $\mu$ and $\l$ are too far then the rank is zero). First, we recall a relevant lemma from \cite{Sar-diagrams}.

\begin{lemma}[cf. Lemma 5.5(i) of \cite{Sar-diagrams}]\label{lem:OrientedForkHelper}
Given $\mu,\eta \in D_{n,k}$, the lower fork diagram $\und{\mu} \eta$ is oriented if and only if
\[
\vee_i^{\mu} \leq \vee_i^{\eta} < \vee_{i+1}^{\mu}
\]
for $1 \leq i \leq n-k$, where we set $\vee_{n-k+1}^{\mu} = n+1$ by convention\footnote{In \cite{Sar-diagrams} there is a typo incorrectly indicating $1 \leq i \leq n-k-1$, rather than $1 \leq i \leq n-k$.  For example,
 take $\mu = \wedge \vee$ and $\eta = \vee \wedge$.} .
\end{lemma}

\begin{proof}
Note that $\und{\mu}\mu$ has some $\wedge$-labeled rays at the left (say $r_0$ of them), followed by $n-k$ ``$\vee$-labeled forks'' (each with one $\vee$ label and $m-1$ $\wedge$ labels in $\mu$ for some $m \geq 1$). Say $\und{\mu}$ has $r \geq r_0$ rays before its first $m$-fork for $m \geq 2$.

Assume that $\und{\mu}\eta$ is oriented. For $1 \leq i \leq n-k$, if the $i$-th $\vee$-labeled fork of $\und{\mu}\mu$ is an $m$-fork for $m \geq 2$, then $\eta$ must have exactly one $\vee$ that is $\geq \vee_i^{\mu}$ but $< \vee_{i+1}^{\mu}$. There must also be a $\vee$ in $\eta$ for every $\vee$-labeled ray of $\und{\mu}\mu$ occurring to the right of an $m$-fork for $m \geq 2$. We have accounted for $n-k-r+r_0$ $\vee$'s in $\eta$, and $r - r_0$ of them remain. Since $\und{\mu}\eta$ is oriented, $\eta$ must have these $\vee$'s on the rightmost $r-r_0$ out of the initial $r$ rays of $\und{\mu}$. These are precisely the initial rays of $\und{\mu}$ on which $\mu$ has a $\vee$, so the condition in the statement of the lemma holds.

Conversely, if the condition holds, then each $m$-fork of $\und{\mu}$ with $m \geq 2$ has exactly one $\vee$ in $\eta$. Furthermore, $\und{\mu}\mu$ starts with a sequence of $\wedge$-labeled rays and no $\wedge$-labeled rays appear in $\und{\mu}\mu$ outside this sequence, so the condition implies that the same is true for $\und{\mu}\eta$. Thus, $\und{\mu} \eta$ is oriented.
\end{proof}

\begin{proposition} \label{prop:how-many-oriented}
Let $\mu,\l \in D_{n,k}$ correspond to left I-states $\x,\y$ that are not too far as in Remark~\ref{rem:OSzSarStates}. Suppose the generating intervals between $\x$ and $\y$ are $[j_i + 1, j_i + l_i]$ for $1 \leq i \leq n-k$. There are exactly $l_1 \cdots l_{n-k}$ choices of $\eta$ such that $\und{\mu} \eta \bar{\l}$ is oriented.
\end{proposition}

\begin{proof}
By Lemma~\ref{lem:OrientedForkHelper}, a choice of $\eta$ such that $\und{\mu} \eta \bar{\l}$ is oriented amounts to a choice of integers $1 \leq \vee_1^{\eta} < \cdots < \vee_{n-k}^{\eta} \leq n$ such that
\[
\max(\vee_i^{\mu},\vee_i^{\l}) \leq \vee_i^{\eta} < \min(\vee_{i+1}^{\mu},\vee_{i+1}^{\l})
\]
for $1 \leq i \leq n-k$, where we set $\vee_{n-k+1}^{\mu} = \vee_{n-k+1}^{\l} = n+1$. Translating $\vee$-sequences for $\mu$ and $\eta$ to hole sequences $h_i^{\x} = \vee_i^{\mu} - 1$ and $h_i^{\y} = \vee_i^{\l} - 1$ (see Remark~\ref{rem:OSzSarStates}), the above condition is equivalent to
\[
\max(h_i^{\x},h_i^{\y}) + 1 \leq \vee_i^{\eta} \leq \min(h_{i+1}^{\x}, h_{i+1}^{\y}).
\]
By Lemma~\ref{lem:SartoriStyleGenInts}, these inequalities are equivalent to $\vee_i^{\eta} \in \{j_i+1,\ldots,j_i+l_i\}$. It follows that the number of $\eta$ such that $\und{\mu} \eta \bar{\l}$ is oriented is equal to the number of ways to choose one element from each set $\{j_i + 1,\ldots,j_i + l_i\}$ for $1 \leq i \leq n-k$, proving the lemma.
\end{proof}

 We use the nonsymmetric quantum integers and factorials
\begin{equation} \label{eq:nonsym-qnum}
(k)_{q^2} := 1+q^2 + \dots + q^{2(k-1)}= q^{k-1}[k], \qquad (k)_{q^2}^{!} := (k)_{q^2}(k-1)_{q^2} \dots (1)_{q^2} = q^{k(k-1)/2}[k]!.
\end{equation}

\begin{corollary}\label{cor:ForkCount}
For $\mu,\l \in D_{n,k}$ that are not too far, corresponding to left I-states $\x,\y$ as in Remark~\ref{rem:OSzSarStates}, the graded rank of the graded free abelian group $\cal{Z}_{\mu,\l}$ is
\[
\rkq(\cal{Z}_{\mu,\l}) = q^d  (k)^!_{q^2} \cdot \prod_{i=1}^{n-k} (l_i)_{q^2},
\]
where $d=\sum_{i=1}^{k}|\wedge_i^{\l} -\wedge_i^{\mu}|$ and $l_1,\ldots,l_{n-k}$ are the lengths of the generating intervals between $\x$ and $\y$. Thus, the ungraded rank of $\cal{Z}_{\mu,\l}$ is $k! \prod_{i=1}^{n-k} l_i$. If $\mu,\l$ are too far, then $\cal{Z}_{\mu,\l} = 0$.
\end{corollary}

\begin{proof}
By the remarks at the end of Proposition~\ref{prop:how-many-oriented}, each choice of $\eta \in D_{n,k}$ such that $\und{\mu}\eta\bar{\l}$ is oriented corresponds to a choice of one element from each of the generating intervals $[j_i+1, j_i + l_i]$, $1 \leq i \leq n-k$, for $\x$ and $\y$. Choosing the $(j_i+1)$-st term of $\eta$ to be $\vee$ for all $i$ produces the lowest degree element $\und{\mu} \eta \bar{\l}$ which has degree $d=\sum_{i=1}^{k}|\wedge_i^{\l} -\wedge_i^{\mu}|$.  More generally, if $\eta$ has a $\vee$ in index $(j_i+1+\gamma_i)$ for $0 \leq \gamma_i < l_i$, then $\und{\mu} \eta \bar{\l}$ has degree $d + \sum_{i=1}^k 2\gamma_i$, so the choices in the $i$-th generating interval contribute a factor of $(l_i)_{q^2}$ to the graded rank.  The result then follows since
\[
 \sum_{\sigma \in S_k}  \deg_q\left( \und{\mu} \eta^{\sigma} \bar{\l} \right) = \sum_{\sigma \in S_k} q^{2l(\sigma)} \deg_q\left( \und{\mu} \eta\bar{\l}\right)
=(k)_{q^2}^!\deg_q\left( \und{\mu} \eta\bar{\l}\right)
\]
\end{proof}

\begin{example}[Number of oriented fork diagrams]
Let $n=12$ and $k=8$. Set
\begin{align*}
\l = (\vee, \wedge,\vee, \wedge, \vee, \wedge, \wedge, \vee, \wedge, \wedge, \wedge, \wedge), \qquad
\mu = (\vee, \wedge, \wedge, \vee, \wedge, \vee, \wedge, \wedge, \wedge, \vee, \wedge, \wedge).
\end{align*}
We have
\begin{align} \nn
\bar{\l}  \;\; := \;\;  \hackcenter{\begin{tikzpicture}[scale=0.8, yscale=-1.0]
% LOWER FORK
    \draw[thick, -] (0,-.25) to[out=-90, in=130]  (.3,-1.25);
    \draw[thick, -] (.6,-.25) to [out=-90, in=50] (.3,-1.25);
    \draw[thick, -] (1.2,-.25) to[out=-90, in=130] (1.5,-1.25);
    \draw[thick, -] (1.8,-.25) to[out=-90, in=50] (1.5,-1.25);
    \draw[thick, -] (2.4,-.25) to[out=-90, in=140] (3,-1.25);
    \draw[thick, -] (3,-.25) -- (3,-1.25);
    \draw[thick, -] (3.6,-.25) to[out=-90, in=40] (3,-1.25);
    \draw[thick, -] (4.2,-.25) to[out=-90, in=160] (5.4,-1.25);
    \draw[thick, -] (4.8,-.25) to[out=-90, in=140] (5.4,-1.25);
    \draw[thick, -] (5.4,-.25) to (5.4,-1.25);
    \draw[thick, -] (6,-.25) to[out=-90, in=40] (5.4,-1.25);
    \draw[thick, -] (6.6,-.25) to[out=-90, in=20] (5.4,-1.25);
\end{tikzpicture}}
\qquad \quad
\und{\mu}  \;\; :=\;\; \hackcenter{\begin{tikzpicture}[scale=0.8, yscale=-1.0]
% UPPER FORK
    \draw[thick, -] (0,.25) to[out=90, in=220] (.6,1.25);
    \draw[thick, -] (.6,.25) -- (.6,1.25);
    \draw[thick, -] (1.2,.25) to[out=90, in=-40] (.6,1.25);
    \draw[thick, -] (1.8,.25) to[out=90, in=230]  (2.1,1.25);
    \draw[thick, -] (2.4,.25) to[out=90, in=-50] (2.1,1.25);
    \draw[thick, -] (3,.25) to[out=90, in=210] (3.9,1.25);
    \draw[thick, -] (3.6,.25) to[out=90, in=230] (3.9,1.25);
    \draw[thick, -] (4.2,.25) to[out=90, in=-50] (3.9,1.25);
    \draw[thick, -] (4.8,.25) to[out=90, in=-30] (3.9,1.25);
    \draw[thick, -] (5.4,.25) to[out=90, in=220]  (6,1.25);
    \draw[thick, -] (6,.25) --  (6,1.25);
    \draw[thick, -] (6.6,.25)to[out=90, in=-40] (6.0,1.25);
\end{tikzpicture}}
\end{align}
so that oriented fork diagrams $\und{\mu} \eta \bar{\l}$ look like
\begin{align}
 \hackcenter{\begin{tikzpicture}[scale=0.8]
% UPPER FORK
    \draw[thick, -] (0,.15) to[out=90, in=220] (.6,1.25);
    \draw[thick, -] (.6,.15) -- (.6,1.25);
    \draw[thick, -] (1.2,.15) to[out=90, in=-40] (.6,1.25);
    \draw[thick, -] (1.8,.15) to[out=90, in=230]  (2.1,1.25);
    \draw[thick, -] (2.4,.15) to[out=90, in=-50] (2.1,1.25);
    \draw[thick, -] (3,.15) to[out=90, in=210] (3.9,1.25);
    \draw[thick, -] (3.6,.15) to[out=90, in=230] (3.9,1.25);
    \draw[thick, -] (4.2,.15) to[out=90, in=-50] (3.9,1.25);
    \draw[thick, -] (4.8,.15) to[out=90, in=-30] (3.9,1.25);
    \draw[thick, -] (5.4,.15) to[out=90, in=220]  (6,1.25);
    \draw[thick, -] (6,.15) --  (6,1.25);
    \draw[thick, -] (6.6,.15)to[out=90, in=-40] (6.0,1.25);
% LOWER FORK
    \draw[thick, -] (0,-.15) to[out=-90, in=130]  (.3,-1.25);
    \draw[thick, -] (.6,-.15) to [out=-90, in=50] (.3,-1.25);
    \draw[thick, -] (1.2,-.15) to[out=-90, in=130] (1.5,-1.25);
    \draw[thick, -] (1.8,-.15) to[out=-90, in=50] (1.5,-1.25);
    \draw[thick, -] (2.4,-.15) to[out=-90, in=140] (3,-1.25);
    \draw[thick, -] (3,-.15) -- (3,-1.25);
    \draw[thick, -] (3.6,-.15) to[out=-90, in=40] (3,-1.25);
    \draw[thick, -] (4.2,-.15) to[out=-90, in=160] (5.4,-1.25);
    \draw[thick, -] (4.8,-.15) to[out=-90, in=140] (5.4,-1.25);
    \draw[thick, -] (5.4,-.15) to (5.4,-1.25);
    \draw[thick, -] (6,-.15) to[out=-90, in=40] (5.4,-1.25);
    \draw[thick, -] (6.6,-.15) to[out=-90, in=20] (5.4,-1.25);
% Weight in middle
    \node at (1.2,0) {  $\wedge$};
    \node at (2.4,0) {  $\wedge$};
    \node at (4.2,0) {  $\wedge$};
    \node at (4.8,0) {  $\wedge$};
    \node[draw, thick,blue, rounded corners=4pt,inner sep=3pt] at (.3,0) {$\qquad$};
    \node[draw, thick,blue, rounded corners=4pt,inner sep=3pt] at (1.8,0) {$\;\;$};
    \node[draw, thick,blue, rounded corners=4pt,inner sep=3pt] at (3.3,0) {$\qquad$};
    \node[draw, thick,blue, rounded corners=4pt,inner sep=3pt] at (6,0) {$\qquad \quad$};
    \node[blue] at (.3,-2) {$\scs x_1 x_2$};
    \node[blue] at (1.8,-2) {$\scs x_4$};
    \node[blue] at (3.3,-2) {$\scs x_6 x_7$};
    \node[blue] at (6,-2) {$\scs x_{10} x_{11} x_{12}$};
    \node at (.6,1.5) {$\scs 1$};
    \node at (2.1,1.5) {$\scs 2$};
    \node at (3.9,1.5) {$\scs 3$};
    \node at (6,1.5) {$\scs 4$};
    \node at (.3,-1.5) {$\scs 1$};
    \node at (1.5,-1.5) {$\scs 2$};
    \node at (3,-1.5) {$\scs 3$};
    \node at (5.4,-1.5) {$\scs 4$};
\end{tikzpicture}}
\end{align}
where the blue circled regions correspond to the generating intervals and represent the possible locations of the $\vee$ in $\und{\mu}\eta \bar{\l}$. Note that the $i$-th generating interval is formed by all indices (necessarily consecutive) connecting the $i$-th $\vee$-labeled fork (not a $\wedge$-labeled ray) of $\mu$ to the $i$-th $\vee$-labeled fork of $\l$ (both $\mu$ and $\l$ have $n-k$ $\vee$-labeled forks).
\end{example}

% - - - - - - - - - - - - - - - - -
\subsubsection{Fork diagrams and maps between Soergel modules}
% - - - - - - - - - - - - - - - - -

Proposition~\ref{prop:WaW} gives us a natural surjective morphism of $R_{\C}$-modules from $\Hom_{R_{\C}}(R_{\bf{b}^{\l}}^{\C} , R_{\bf{b}^{\mu}}^{\C})/(\cal{W}^{\alpha}_{\l,\mu} \otimes \C)$ to $\Hom_{R_{\C}}(C_{w_k \l},C_{w_k \mu})/\cal{W}_{\l,\mu}$. We want to show that this map is an isomorphism, so that the explicit submodule $\cal{W}^{\alpha}_{\l,\mu} \otimes \C$ describes the space of illicit morphisms from $C_{w_k \l}$ to $C_{w_k \mu}$.

Indeed, we know the dimension of $\Hom_{R_{\C}}(C_{w_k \l},C_{w_k \mu})/\cal{W}_{\l,\mu}$ by a result of Sartori. Using the connection between $\cal{A}_{n,k}$ and subquotients of category $\cal{O}$, Sartori establishes the following lemma.

\begin{lemma}[Lemma 6.6 \cite{Sar-diagrams}]\label{lem:Sartori6.6}
We have
 \[
 \dim\Hom_{R_{\C}}(C_{w_k \l}, C_{w_k\mu})/ \cal{W}_{\l,\mu} = \dim (\cal{Z}_{\mu,\l} \otimes \C).
 \]
\end{lemma}

To show that $\cal{W}_{\l,\mu} = \Wa_{\l,\mu} \otimes \C$, it therefore suffices to show that
\[
\dim \Hom_{R_{\C}}(R_{\bf{b}^{\l}}^{\C}, R_{\bf{b}^{\mu}}^{\C})/(\cal{W}^{\alpha}_{\l,\mu} \otimes \C) \leq \dim (\cal{Z}_{\mu,\l} \otimes \C).
\]

\begin{lemma}[cf. Proposition 5.8 of \cite{Sar-diagrams}, Remark~\ref{rem:SartoriProp58Counterex} below]\label{lem:NewerSartori58}
The image of the map $\Psi$ from oriented enhanced fork diagrams for $\mu$ and $\l$ (i.e. basis elements for $\cal{Z}_{\mu,\l}$) to $\Hom_{R}(R/I_{\mathbf{b}^{\l}},R/I_{\mathbf{b}^{\mu}})/ \Wa_{\l,\mu} $ sending
\[
\und{\mu} \eta^{\sigma} \bar{\l} \mapsto  \left(1 \mapsto p_{\und{\mu} \eta^{\sigma}}\right) +  \Wa_{\l,\mu} \nn,
\]
where
\begin{equation} \label{eq:p-poly}
  p_{\und{\mu} \eta^{\sigma}} = \mf{S}'_{\sigma}\left(x_{\wedge_1^{\eta}}, \dots , x_{\wedge_k^{\eta}} \right)
  \prod_{j=1}^{n-k} x_{\vee_j^{\mu}} x_{(\vee_j^{\mu}) +1 } \dots x_{\vee_j^{\eta} -1 } \in R,
\end{equation}
is degree zero and its image generates $\Hom_{R}(R_{\mathbf{b}^{\l}}, R_{\mathbf{b}^{\mu}})/ \Wa_{\l,\mu}$ over $\Z$. Thus, the image gives a homogenous spanning set for $\Hom_{R_{\C}}(R_{\mathbf{b}^{\l}}^{\C}, R_{\mathbf{b}^{\mu}}^{\C})/ (\Wa_{\l,\mu} \otimes \C)$ over $\C$.
\end{lemma}

\begin{proof}
One can check that the elements given in Proposition~\ref{prop:SartoriSoergel}, item~\eqref{it:SoergelMorphismBasis} generate $\Hom_R(R_{\mathbf{b}^{\l}},R_{\mathbf{b}^{\mu}})$ over $\Z$. The set of such elements that are not in $\Wa_{\l,\mu}$ is a generating set for $\Hom_{R}(R_{\mathbf{b}^{\l}},R_{\mathbf{b}^{\mu}}) / \Wa_{\l,\mu}$; it suffices to show that this set is contained in the image of $\Psi$.

Indeed, to see that (the class of) any such basis element that is not in $\Wa_{\l,\mu}$ is in the image of $\Psi$, one can use the argument in the final paragraph of \cite[Proposition 5.8]{Sar-diagrams}. This argument works (replacing $\vee_j^{\mu}$ with $\alpha(j) = \max(v_j^{\lambda},v_j^{\mu})$ in the definition of $\ell_j$) assuming we are given a monomial $\mathfrak{m}$ that is not in $\Wa_{\l,\mu}$; see Remark~\ref{rem:SartoriProp58Counterex} below for a counterexample when we are given $\mathfrak{m} \in \Wa_{\l,\mu} \setminus \widetilde{\cal{W}}_{\l,\mu}$.
\end{proof}

\begin{corollary}\label{cor:W}
Under the identification
\[
\Hom_{R_{\C}}(C_{w_k \l},C_{w_k \mu}) \leftrightarrow \Hom_{R_{\C}}(R_{\mathbf{b}^{\l}}^{\C}, R_{\mathbf{b}^{\mu}}^{\C})
\]
given by Proposition~\ref{prop:SartoriSoergel}, item \eqref{it:ExplicitSoergelModules}, we have $\cal{W}_{\l,\mu} \leftrightarrow \Wa_{\l,\mu} \otimes \C$.
\end{corollary}

\begin{proof}
As mentioned above, Proposition~\ref{prop:WaW} shows that the identification of Proposition~\ref{prop:SartoriSoergel}, item \eqref{it:ExplicitSoergelModules} gives a natural surjective linear map of complex vector spaces
\[
\Hom_{R_{\C}}(R_{\bf{b}^{\l}}^{\C}, R_{\bf{b}^{\mu}}^{\C})/(\cal{W}^{\alpha}_{\l,\mu} \otimes \C) \to \Hom_{R_{\C}}(C_{w_k \l},C_{w_k \mu})/\cal{W}_{\l,\mu}.
\]
By Lemma~\ref{lem:Sartori6.6} and Lemma~\ref{lem:NewerSartori58}, the dimension of the domain is no greater than the dimension of the codomain, so the map is an isomorphism.
\end{proof}

\begin{remark}\label{rem:SartoriProp58Counterex}
The final paragraph of Sartori's proof of \cite[Proposition 5.8]{Sar-diagrams} does not work if we are only given $\mathfrak{m}$ that does not lie in $\widetilde{\cal{W}}_{\l,\mu}$. For instance, take $\mathfrak{m} = x_2 x_3$ in Example~\ref{ex:TildeWCounterex}. We have $x_2 x_3 \in \Wa_{\l,\mu} \setminus \widetilde{\cal{W}}_{\l,\mu}$. Since $\vee_1^{\mu} = 2$ and $\vee_2^{\mu} = 3$, we have $\ell_1 = 4$ and $\ell_2 = 4$. This is a problem because we need $\ell_1 < \cdots < \ell_{n-k}$ for the proof to work.
\end{remark}

% ------------------------------------------------------
\subsection{A \texorpdfstring{$\Z$}{Z} lift of Sartori's algebra}
% ------------------------------------------------------

By Corollary~\ref{cor:W}, it is reasonable to define the following $\Z$ lift of Sartori's algebra.

\begin{definition}
Let $\A_{n,k}^{\Z}$ be the graded ring
\[
\A_{n,k}^{\Z} = \bigoplus_{\l,\mu \in D_{n,k}} \Hom_R(R_{\bf{b}^{\l}}, R_{\bf{b}^{\mu}}) / \Wa_{\l,\mu}.
\]
\end{definition}

Corollary~\ref{cor:W} implies that $\A_{n,k}^{\Z} \otimes \C \cong \A_{n,k}$. It is natural to ask whether passing from $\C$ to $\Z$ in this manner introduces any torsion in $\A_{n,k}^{\Z}$; the answer is ``no,'' as shown in the next proposition.

\begin{proposition} \label{prop:AZfree}
For $\l,\mu \in D_{n,k}$ that are not too far, the graded abelian group $\Hom_R(R_{\bf{b}^{\l}}, R_{\bf{b}^{\mu}}) / \Wa_{\l,\mu}$ is free.  Consequently, $\1_{\mu} \cal{A}_{n,k}^{\Z} \1_{\l}$ is a free graded $\Z$-module with homogeneous basis given by the set $\{1 \mapsto p_{\und{\mu} \eta^{\sigma}}\}$ where $\eta \in D_{n,k}, \sigma \in S_k$ run over all choices such that $\und{\l} \eta^{\sigma} \bar{\mu}$ is an enhanced oriented fork diagram. Furthermore, the graded rank $\rkq (\1_{\mu}\cal{A}_{n,k}^{\Z}\1_{\l})$ is equal to $\rkq(\cal{Z}_{\mu,\l})$ from Corollary~\ref{cor:ForkCount}.
\end{proposition}

\begin{proof}
Lemma~\ref{lem:NewerSartori58} gives us a generating set for $\Hom_R(R_{\bf{b}^{\l}}, R_{\bf{b}^{\mu}}) / \Wa_{\l,\mu}$ whose size is the number of oriented enhanced fork diagrams for $\l,\mu$, i.e. the dimension of $\cal{Z}_{\l,\mu} \otimes \C$. This number is also the rank of $\Hom_R(R_{\bf{b}^{\l}}, R_{\bf{b}^{\mu}}) / \Wa_{\l,\mu}$ by Lemma~\ref{lem:Sartori6.6} and Corollary~\ref{cor:W}, so the generating set is a basis. It follows that $\Hom_R(R_{\bf{b}^{\l}}, R_{\bf{b}^{\mu}}) / \Wa_{\l,\mu}$ is free; the graded rank follows from Corollary~\ref{cor:ForkCount} and Lemma~\ref{lem:NewerSartori58}.
\end{proof}
Thus, the natural map from $\A_{n,k}^{\Z}$ to $\A_{n,k}$ is injective.

\begin{proposition} \label{prop:Bord2Sart}
The composition
\[
\B_l(n,k) \to \B^{\C}_l(n,k) \xrightarrow{\Xi} \A_{n,k}
\]
has image in $\A_{n,k}^{\Z} \subset \A_{n,k}$, where $\Xi$ is the homomorphism from Proposition~\ref{prop:BigStepXi}.
\end{proposition}

\begin{proof}
One can check that each generator of $\B_l(n,k) \subset \B^{\C}_l(n,k)$, in either the big-step or the small-step description, gets sent by $\Xi$ to an element of $\A_{n,k}^{\Z} \subset \A_{n,k}$.
\end{proof}

\begin{corollary}\label{cor:IntegerXi}
We have a commutative square of $R$-algebra homomorphisms
\[
\xymatrix{
\B_l(n,k) \ar[r]^{\Xi} \ar[d] & \A_{n,k}^{\Z} \ar[d] \\
\B^{\C}_l(n,k) \ar[r]_{\Xi} & \A_{n,k},
}
\]
where the bottom edge is the map from Proposition~\ref{prop:BigStepXi}; by slight abuse of notation, we call both the top and the bottom edges $\Xi$. Since each fork monomial $p_{\und{\mu} \eta^{\sigma}}$ is an $R$-multiple of $\Xi(f_{\x,\y})$ where $\x,\y \in V_l(n,k)$ correspond to $\mu, \l \in D_{n,k}$, Lemma~\ref{lem:NewerSartori58} implies that the top edge $\Xi$ of the square is surjective.
\end{corollary}

The anti-automorphism $\psi_S \maps \cal{A}_{n,k} \to \cal{A}_{n,k}$ from \eqref{sec:Sart-anti} extends to an anti-automorphism of the integral form $\psi_S \maps \cal{A}_{n,k}^{\Z} \to \cal{A}_{n,k}^{\Z}$ given on the basis from Proposition~\ref{prop:AZfree} by
\begin{align}
  \psi_S \maps \cal{A}_{n,k}^{\Z} &\longrightarrow \cal{A}_{n,k}^{\Z}
   \\ \nn
   \und{\mu} \eta^{\sigma} \bar{\l} & \longmapsto \und{\l} \eta^{\sigma} \bar{\mu}.
\end{align}

% ====================================================
\section{A vanishing ideal in the Sartori algebra}
% ====================================================

In this section we identify an ideal that is present in $R_{\mathbf{b}^{\mu}}$ for all $\mu \in D_{n,k}$ and hence is zero in $\cal{A}_{n,k}^{\Z}$.
Define an ideal in $R$ by
\begin{equation} \label{eq:J}
  \cal{J} = \cal{J}_{n,k} :=
  \left\langle
  e_1(x_1,\dots, x_{n}), e_2(x_1, \dots, x_{n}), \dots e_{k}(x_1, \dots, x_{n})
  \right\rangle.
\end{equation}
Equivalently, $\cal{J}_{n,k}$ is the ideal generated by $h_i(x_1,\ldots,x_n)$ for $1 \leq i \leq k$; we will work primarily with the polynomials $h_i$ as generators of $\cal{J}_{n,k}$ below.

The ideal $\cal{J}$ admits a further alternate description that we give below after a preliminary lemma.

\begin{lemma}\label{lem:Defo1}
For $1 \leq k \leq n$ and $1 \leq p \leq k$, we have $h_k(x_1,\ldots,x_{n-p+1}) = h_k(x_1,\ldots,x_n)$ in the ring $\frac{R}{(h_1(x_1,\ldots,x_n),\ldots,h_{k-1}(x_1,\ldots,x_n))}$.
\end{lemma}

\begin{proof} Induct on $k$; the case $k=1$ follows from $1 \leq p \leq k$. For $k > 1$, we will induct on $p$; the case $p=1$ is trivial. Assume $p > 1$; then
\[
h_k(x_1,\ldots,x_{n-p+1}) = h_k(x_1,\ldots,x_{n-p+2}) - x_{n-p+2}h_{k-1}(x_1,\ldots,x_{n-p+2}).
\]
By induction on $k$, we have $h_{k-1}(x_1,\ldots,x_{n-p+2}) = h_{k-1}(x_1,\ldots,x_n)$ in $\frac{R}{(h_1(x_1,\ldots,x_n),\ldots,h_{k-2}(x_1,\ldots,x_n))}$. Thus,
\[
h_k(U_1,\ldots,U_{n-p+1}) = h_k(U_1,\ldots,U_{n-p+2}) = h_k(U_1,\ldots,U_n)
\]
modulo $(h_1(x_1,\ldots,x_n),\ldots,h_{k-1}(x_1,\ldots,x_n))$ (the second equality follows from induction on $p$).
\end{proof}

\begin{corollary}\label{cor:Defo1}
The elements $\theta_i := h_i(x_1,\dots, x_{n+1-i})$ of $R$, for $1 \leq i \leq k$, generate the ideal $\cal{J}$.
\end{corollary}

\begin{proof}
For $1 \leq i \leq k$, write $\eta_i = h_i(x_1,\ldots,x_n)$. Induct on $k$; the case $k=1$ is clear. If $k > 1$, then
\begin{align*}
(\theta_1,\ldots,\theta_k) &= (\theta_k) + (\theta_1,\ldots,\theta_{k-1}) \\
&= (\theta_k) + (\eta_1,\ldots,\eta_{k-1}) \\
&= (\eta_k) + (\eta_1,\ldots,\eta_{k-1}) \\
&= (\eta_1,\ldots,\eta_k),
\end{align*}
where we use the inductive hypothesis in the second equality and Lemma~\ref{lem:Defo1} in the third equality.
\end{proof}

\begin{proposition}
The ideal $\cal{J}_{n,k}$ acts by zero on $\cal{A}_{n,k}^{\Z}$.
\end{proposition}

\begin{proof}
The minimal sequence $\tilde{\mu}$ in the Bruhat order generated by $\wedge \vee \succ \vee \wedge$ is the sequence
\begin{equation} \label{eq:mu-min}
\tilde{\mu} = \vee^{n-k} \wedge^k, \qquad \mathbf{b}^{\tilde{\mu}} = (k+1, k+1, \dots, k+1, k,k-1,\dots,2,1).
\end{equation}
All other sequences $\l\in D_{n,k}$ will have $\mathbf{b}$-sequences with $b_i^{\l} \leq b_i^{\tilde{\mu}}$.  The result follows by a $\Z$ analogue of \cite[Lemma 2.7]{Sar-diagrams} showing that $h_a(x_1, \dots, x_i) \in I_{\mathbf{b}}$ for every $a \geq b_i$, so that $\cal{J} \subset I_{\mathbf{b}^{\l}}$ for all $\l \in D_{n,k}$ by Corollary~\ref{cor:Defo1}.
\end{proof}

In other words, the $\Z[x_1,\ldots,x_n]$-module structure on $\cal{A}_{n,k}^{\Z}$ descends to an action of $\frac{\Z[x_1,\ldots,x_n]}{\cal{J}}$.

\begin{definition}
Define the quotient Ozsv{\'a}th--Szab{\'o} algebra $\Bq$ to be the quotient of the $\Z[U_1,\ldots,U_n]$-algebra $\B_l(n,k)$ by the action of the ideal $\cal{J}$ defined in \eqref{eq:J} (with $x_i$ variables relabeled as $U_i$.)
\end{definition}

Note that Lemma~\ref{lem:Defo1} and Corollary~\ref{cor:Defo1} hold in $\B_l(n,k)$ since they hold in $\Z[x_1,\ldots,x_n]$.

\begin{corollary}\label{cor:AlgIsoWellDef}
The homomorphism $\Xi: \B_l(n,k) \to \cal{A}_{n,k}^{\Z}$ from Corollary~\ref{cor:IntegerXi} descends to a well-defined homomorphism
\[
\Xi: \Bq \to \cal{A}_{n,k}^{\Z}.
\]
\end{corollary}

% ====================================================
\section{Fork elements and injectivity}
% ====================================================

% ----------------------------------------------------
\subsection{Deformations}
% ----------------------------------------------------

Let $A$ be a graded ring, and let $U$ be a finitely generated free $\Z$-module. Following the notation of \cite[Section 4]{BLPPW}, we say a graded ring $\tilde{A}$ is a graded deformation of $A$ over $U^*$ if $\tilde{A}$ is equipped with graded homomorphisms
\[
{\rm Sym}(U) \xrightarrow{j} \tilde{A} \xrightarrow{\pi} A
\]
such that $\im(j) \subset Z(\tilde{A})$ and $\pi$ induces an isomorphism from $\frac{\tilde{A}}{\im(j)}$ to $A$. The deformation is called flat if $j$ makes $\tilde{A}$ a flat ${\rm Sym}(U)$-module.

Let $U = \Z^k$ with standard basis $\{\varepsilon_1,\ldots,\varepsilon_k\}$, so that ${\rm Sym}(U) = \Z[\varepsilon_1,\ldots,\varepsilon_k]$, and define $j\colon {\rm Sym}(U)\to \B_l(n,k)$ by sending $\varepsilon_i$ to the central element $e_i(U_1,\ldots,U_{n})$ (one could equivalently use complete homogeneous symmetric polynomials $h_i(U_1,\ldots,U_n)$ instead of elementary symmetric polynomials). If $\pi\colon B_l(n,k) \to \cal{A}^{\Z}_{n,k}$ is the homomorphism $\Xi$ from Corollary~\ref{cor:IntegerXi}, then Corollary~\ref{cor:AlgIsoWellDef} tells us that the image of $j$ is contained in the kernel of $\pi$.

We want to show that $\pi = \Xi$ induces an isomorphism from $\B_l(n,k)/\im(j)$ to $\cal{A}_{n,k}^{\Z}$, so that we may view $\B_l(n,k)$ as a graded deformation of $\cal{A}_{n,k}^{\Z}$, and we want to know that this deformation is flat. Indeed, we first show that $\B_l(n,k)$ is a free ${\rm Sym}(U)$-module on a basis defined in the next section; we deal with injectivity in Section~\ref{sec:InjectivityAfterFreeness}.

% ----------------------------------------------------
\subsection{Fork elements as an \texorpdfstring{$S$}{S}-basis}
% ----------------------------------------------------

The Bruhat order generated by $\wedge \vee \succ \vee \wedge$ induces a partial order on I-states with size $k$ given below.
\begin{definition}
Define a partial order on the set $V(n,k)$ of I-states with $|\x|=k$ by $\x \succ \y$ if $x_i \leq y_i$ for all $1 \leq i \leq k$.
\end{definition}

\begin{definition} \label{def:b-fork-element}
Let $\x,\y,\z \in V_l(n,k)$  and suppose that none of the three are too far from one another.  Let
$[j_i + 1,j_i + l_i]$ for $1 \leq i \leq n-k$  be the generating intervals for $\x$ and $\y$ and assume  $\x\preceq \z \succeq\y$ (if $\x,\y,\z$ correspond to $\mu,\l,\eta \in D_{n,k}$, Lemma~\ref{lem:OrientedForkHelper} implies that this condition is equivalent to $\und{\mu} \eta \bar{\l}$ being an oriented fork diagram).
For $\sigma \in S_k$, define fork polynomials
\begin{align} \label{eq:fork-element}
 {\sf p}_{(\x, \z^{\sigma}, \y) }
:=
\mf{S}'_{\sigma}\left(U_{z_1+1}, \dots , U_{z_k+1} \right)
 \left( \prod_{i=1}^{n-k} U_{j_i+1} U_{j_i+2} \dots U_{h_i^{\z} }\right) \in \Z[U_1\ldots,U_n],
\end{align}
where $h_i^{\z}$ is the hole sequence of $\z$ defined in Definition~\ref{def:hole-sequence}. We have corresponding fork elements ${\sf p}_{(\x, \z^{\sigma}, \y) } f_{\x,\y} \in \B_l(n,k)$.
\end{definition}

Note that ${\sf p}_{(\x, \z^{e}, \y) } = U_1^{c_1} \dots U_n^{c_n}$ with $c_j=\min(v^{\x}_j,v^{\y}_j)-v^{\z}_j \in \{0,1\}$, so that  ${\sf p}_{(\x, \z^{e}, \y) } = 1$ if $z_i = \min(x_i,y_i)$. Likewise, recall from \eqref{eq:schubert-prime} that $\mf{S}'_{\sigma}\left(U_{z_1+1}, \dots , U_{z_k+1} \right) \in \{ U_{z_1+1}^{\ell_1} \dots U_{z_k+1}^{\ell_k} \mid 0 \leq \ell_i \leq k-i \}$.

\begin{proposition}\label{prop:XiOnForks}
Let $\mu=\mu^{\x}$, $\l=\mu^{\y}$, and $\eta = \mu^{\z}$ denote $\vee\wedge$-sequences in $D_{n,k}$ associated to left I-states $\x, \y, \z \in V_l(n,k)$ satisfying the assumptions in Definition~\ref{def:b-fork-element} (equivalently, such that $\und{\mu} \eta \bar{\l}$ is oriented). Under the surjective homomorphism $\Xi \maps \B_l(n,k) \to \cal{A}_{n,k}$, we have
\begin{equation} \label{eq:fork-basis}
  \Xi_{{\sf p}_{(\x, \z^{\sigma}, \y) } f_{\x,\y}} = (1 \mapsto p_{\und{\mu} \eta^{\sigma}  })
\end{equation}
with $p_{\und{\mu} \eta^{\sigma} }$  defined in \eqref{eq:p-poly}.
\end{proposition}

\begin{proof}
Let $\varrho \in D_{n,k}$ be such that $\und{\mu}\varrho^{e} \bar{\l}$ is the minimal degree oriented enhanced fork diagram with lower fork $\und{\mu}$ and upper fork $\bar{\l}$, so that $\varrho$ has all $\vee$'s maximally to the left subject to the constraint that $\und{\mu}\varrho^{e} \bar{\l}$ is an oriented fork diagram. Explicitly, we have $\vee^{\varrho}_j = \max(\vee^{\mu}_j,\vee^{\l}_j)$. Then since $\varrho \preceq \eta$ for all oriented fork diagrams $\und{\mu}\eta\bar{\l}$ by assumption, $p_{\und{\mu}\varrho}$ divides $p_{\und{\mu}\eta^{\sigma}}$ for every oriented enhanced fork diagram $\und{\mu}\eta^{\sigma} \bar{\l}$.  The morphism $1 \mapsto p_{\und{\mu}\varrho}$ is the image of the generator $f_{\x,\y}$ of $\Ib_{\x} \B(n,k) \Ib_{\y}$ under $\Xi$. Then
\begin{align}
\Xi_{{\sf p}_{(\x, \z^{\sigma}, \y) } f_{\x,\y}}(1)
&:=
\mf{S}'_{\sigma}\left(x_{z_1+1}, \dots , x_{z_k+1} \right)
 \left( \prod_{i=1}^{n-k} x_{j_i+1} x_{j_i+2} \dots x_{h_i^{\z} }\right) \cdot \Xi_{f_{\x,\y}}(1)
 \nn \\
 &= \mf{S}'_{\sigma}\left(x_{\wedge_1^{\eta}}, \dots , x_{\wedge_k^{\eta}} \right)
 \left( \prod_{j=1}^{n-k} x_{\vee_j^{\varrho}} x_{(\vee_j^{\varrho}) +1 } \dots x_{\vee_j^{\eta} -1 }\right) \cdot p_{\und{\mu}\varrho}
\nn \\
&= \mf{S}'_{\sigma}\left(x_{\wedge_1^{\eta}}, \dots , x_{\wedge_k^{\eta}} \right)
 \left( \prod_{j=1}^{n-k} x_{\vee_j^{\varrho}} x_{(\vee_j^{\varrho}) +1 } \dots x_{\vee_j^{\eta} -1 }\right)
\left(\prod_{j=1}^{n-k} x_{\vee_j^{\mu}} x_{(\vee_j^{\mu}) +1 } \dots x_{\vee_j^{\varrho} -1 } \right)
  \nn\\
&= \mf{S}'_{\sigma}\left(x_{\wedge_1^{\eta}}, \dots , x_{\wedge_k^{\eta}} \right)
 \left( \prod_{j=1}^{n-k} x_{\vee_j^{\mu}} x_{(\vee_j^{\mu}) +1 } \dots x_{\vee_j^{\eta} -1 }\right)
  \nn\\
&=  p_{\und{\mu} \eta^{\sigma}}\nn
\end{align}
and the result follows.
\end{proof}

\begin{theorem}\label{thm:flat}
The Ozsv{\'a}th--Szab{\'o} algebra $\B_l(n,k)$ is free over $S$ with a basis given by elements ${\sf p}_{(\x, \z^{\sigma}, \y) } f_{\x,\y}$ where $\x,\y \in V_l(n,k)$ are not too far and $\{{\sf p}_{(\x, \z^{\sigma}, \y) }\}$ are the fork polynomials from \eqref{eq:fork-element}.
\end{theorem}

\begin{proof}
As an $S$-module, $\B_l(n,k)$ is the direct sum of $\Ib_{\x} \B_l(n,k) \Ib_{\y}$ for $\x,\y \in V_l(n,k)$. Let $\x,\y \in V_l(n,k)$ be not too far; it suffices to show that $\Ib_{\x} \B_l(n,k) \Ib_{\y}$ is a free $S$-module with a basis given by the above fork elements.

Let $i_0 \geq 0$ be the maximal index such that $x_i = y_i = i-1$ for all $i \leq i_0$. We will induct on $|\{i > i_0 : x_i =  y_i\}|$; first suppose this number is zero. We have a crossed line between $\x$ and $\y$ as in Definition~\ref{def:OSzTerminology} for each index $i_0 + 1 \leq j \leq k$, so $k-i_0$ lines are crossed. The $n-k$ generating intervals between $\x$ and $\y$ must all be contained in the $(n-i_0)$-element set $\{i_0+1,\ldots,n\}$, $k-i_0$ of whose elements correspond to crossed lines, so each generating interval must have length $1$. There is a unique $\z \in V_l(n,k)$ that is pairwise not too far from $\x,\y$ and satisfies $\x \preceq \z \succeq \y$, namely $z_i = \mathrm{min}(x_i,y_i)$, and $\Ib_{\x} \B_l(n,k) \Ib_{\y}$ is the quotient of $\Z[U_1,\ldots,U_n]$ by all variables except $U_{z_1 + 1},\ldots,U_{z_k+1}$. The fork elements of $\Ib_{\x} \B_l(n,k) \Ib_{\y}$ are given in this quotient by ``staircase'' monomials $U_{z_1 + 1}^{c_1} \cdots U_{z_k+1}^{c_k}$ with $0 \leq c_i \leq k-i$, and the elements $e_i(U_1,\ldots,U_n)$ of $S$ act as $e_i(U_{z_1 + 1},\ldots,U_{z_k + 1})$. The result now follows from \cite[Proposition 2.5.5]{Manivel}, which shows that the staircase monomials provide a basis for the polynomial ring over the ring of symmetric polynomials.

Now suppose that $x_i = y_i$ for some minimal $i > i_0$. It follows that $x_i = j+1$ for some generating interval $[j+1,\ldots,j+l]$ between $\x$ and $\y$. We consider four cases.
\begin{itemize}
\item If $j\in \x$ (so $j \notin \y$) and $j+l \in \x$ (so $j+l \notin \y$), let $\x' = \x$ and $\y' = (\y \setminus \{j+1\}) \cup \{j+l\}$. Let $\x'' = \x$ and $\y'' = (\y \setminus \{j+1\}) \cup \{j\}$.
\item If $j \in \x$ and $j+l \notin \x$, let $\x' = (\x \setminus \{j+1\}) \cup \{j+l\}$ and $\y' = \y$. Let $\x'' = \x$ and $\y'' = (\y \setminus \{j+1\}) \cup \{j\}$.
\item If $j \notin \x$ and $j+l \in \x$, let $\x' = \x$ and $\y' = (\y \setminus \{j+1\}) \cup \{j+l\}$. Let $\x'' = (\x \setminus \{j+1\}) \cup \{j\}$ and $\y'' = \y$.
\item If $j \notin \x$ and $j+l \notin \x$, let $\x' = (\x \setminus \{j+1\}) \cup \{j+l\}$ and $\y' = \y$. Let $\x'' = (\x \setminus \{j+1\}) \cup \{j\}$ and $\y'' = \y$.
\end{itemize}
In all cases, an element $\z \in V_l(n,k)$ is pairwise not too far from $\x'$ and $\y'$ and satisfies $\x' \preceq \z \succeq \y'$ if and only if $\z$ is pairwise not too far from $\x$ and $\y$, satisfies $\x \preceq \z \succeq \y$, and also satisfies $z_i = x_i (=y_i)$. Similarly, $\z$ is pairwise not too far from $\x'$ and $\y'$ and satisfies $\x' \preceq \z \succeq \y'$ if and only if $\z$ is pairwise not too far from $\x$ and $\y$, satisfies $\x \preceq \z \succeq \y$, and also satisfies $z_i = x_i -1 (=y_i -1)$.

For $\sigma \in S_k$ and $\z$ as above, we have ${\sf p}_{(\x,\z^{\sigma},\y)} = {\sf p}_{(\x',\z^{\sigma},\y')}$ or ${\sf p}_{(\x,\z^{\sigma},\y)} = U_{j+1} {\sf p}_{(\x'',\z^{\sigma},\y'')}$ as appropriate. We see that the fork polynomials for $\Ib_{\x} \B_l(n,k) \Ib_{\y}$ can be viewed as the fork polynomials for $\Ib_{\x'} \B_l(n,k) \Ib_{\y'}$ together with $U_{j+1}$ times the fork polynomials for $\Ib_{\x''} \B_l(n,k) \Ib_{\y''}$.

Now consider the exact sequence
\[
0 \to \Ib_{\x''} \B_l(n,k) \Ib_{\y''} \xrightarrow{\cdot U_{j+1}} \Ib_{\x} \B_l(n,k) \Ib_{\y} \xrightarrow{\cdot 1} \Ib_{\x'} \B_l(n,k) \Ib_{\y'} \to 0,
\]
of $S$-modules, interpreting each term as a quotient of $\Z[U_1,\ldots,U_n]$. We have $|\{i > i_0 : x'_i =  y'_i\}| < |\{i > i_0 : x_i =  y_i\}|$ and similarly for $\x'',\y''$. By induction, the first and third terms of the exact sequence are free $S$-modules with bases given by fork elements. Since the third term is free, the sequence splits; a basis for the middle term is given by $f_{\x,\y}$ times fork polynomials for the third term together with $U_{j+1} f_{\x,\y}$ times fork polynomials for the first term. As discussed above, the resulting basis coincides with the fork elements of $\Ib_{\x} \B_l(n,k) \Ib_{\y}$.
\end{proof}

\begin{remark}
The proof of Theorem~\ref{thm:flat}, together with \cite[Proposition 2.5.3]{Manivel}, shows that we can replace the leading term $\mf{S}'_{\sigma}\left(U_{z_1+1}, \dots , U_{z_k+1} \right)$ of the Schubert polynomial with the actual Schubert polynomial $\mf{S}_{\sigma}\left(U_{z_1+1}, \dots , U_{z_k+1} \right)$ in the definition of fork polynomials from Definition~\ref{def:b-fork-element} and Theorem~\ref{thm:flat} continues to hold.
\end{remark}

% ----------------------------------------------------
\subsection{Injectivity}\label{sec:InjectivityAfterFreeness}
% ----------------------------------------------------

We can now prove injectivity for the algebra homomorphism $\Xi$ from Corollary~\ref{cor:AlgIsoWellDef}.
\begin{theorem}\label{thm:newinjective}
The map $\Xi\colon \Bqk \to \cal{A}^{\Z}_{n,k}$ from Corollary~\ref{cor:AlgIsoWellDef} is an isomorphism.
\end{theorem}
\begin{proof}
By Theorem~\ref{thm:flat}, the fork elements from \eqref{eq:fork-element} for all $\x,\y \in V_l(n,k)$ that are not too far give a $\Z$-basis for $\Bqk$; Proposition~\ref{prop:XiOnForks} and Proposition~\ref{prop:AZfree} show that $\Xi$ sends these elements to a $\Z$-basis for $\cal{A}^{\Z}_{n,k}$.
\end{proof}

% ====================================================
\section{Categorification of bases and bilinear forms}
% ====================================================

% ----------------------------------------------------
\subsection{Quantum \texorpdfstring{$\mf{gl}(1|1)$}{gl(1|1)}}
% ----------------------------------------------------

Let $\epsilon_1 = (1,0)$ and $\epsilon_2=(0,1)$ denote the standard basis for the weight lattice $\Z^2$ of $\mf{gl}(1|1)$; let $h_1$ and $h_2$ denote the basis for dual weight lattice with associated pairing $\langle h_i, \epsilon_j\rangle = \delta_{i,j}$.  We denote the simple root of $\mf{gl}(1|1)$ by $\alpha=\epsilon_1-\epsilon_2$.

\begin{definition}
The Hopf superalgebra $U_q(\mf{gl}(1|1))$ is generated as a superalgebra over $\C(q)$ by two even generators $K_1^{\pm}, K_2^{\pm}$ and two odd generators $E,F$ with relations
\begin{alignat}{3}
  &K_i K_j = K_jK_i \quad   &&K_iK_{i}^{-1} = 1 = K_i^{-1}K_i  \quad  &&\text{for $i,j \in \{1,2\}$}
  \\
 & K_iE = q^{\langle h_i, \alpha \rangle}EK_i \quad    && K_iF  = q^{-\langle h_i, \alpha \rangle }FK_i &&
  \\
 & E^2 = F^2 = 0 \quad &&EF + FE = \frac{K - K^{-1}}{q - q^{-1}} &&
\end{alignat}
where $K = K_1 K_2$.   The comultiplication is given by $\Delta(E) = E \otimes K^{-1} + 1 \otimes E$ and $\Delta(F) = F \otimes 1 + K \otimes F$, and $\Delta(K_i) = K_i \otimes K_i$. We will not need explicit formulas for the counit or antipode.
\end{definition}

Let $V=\C(q)\langle v_0, v_1\rangle$ denote the two dimensional simple $U_q(\mf{gl}(1|1))$-module with highest weight $\epsilon_1$.   The super grading is fixed by setting the highest weight space spanned by $v_0$ to be even, so that $v_1$ is odd.  We write $V^{\otimes n}$ for the $n$-fold tensor power of $V$ and $(V^{\otimes n})_k$ for $\{ v \in V^{\otimes n} \mid Kv =q^{\langle h, k\epsilon_1, (n-k)\epsilon_2\rangle} v \}$.

% ------------------------------------------------------
\subsection{The canonical basis of \texorpdfstring{$V^{\otimes n}$}{Vn}}
% ------------------------------------------------------

We first describe the canonical basis of $V^{\otimes n}$ used by Sartori \cite[Section 4.3]{Sar-tensor}, following Zhang \cite{ZhangCanonical} (see also \cite{Zho,BKK}).

\begin{remark}
Here we are using ``canonical'' in the combinatorial or crystal sense of Kashiwara~\cite{Kas}. For superalgebras the authors are not aware of a geometric construction  in the sense of Lusztig~\cite{Lus3} giving rise to canonical bases for $U_q(\mf{gl}(1|1))$-modules.
\end{remark}

For $1 \leq i \leq n$, let $e_i$ be the standard basis element of $V^{\otimes n}$ with $v_0$ in position $i$ and $v_1$ in all other positions. We write the standard basis vector with $v_0$ in positions $i_k > \cdots > i_1$ and $v_1$ in all other positions as $e_{i_k} \wedge \cdots \wedge e_{i_1}$. Sartori writes this basis vector as $v_{\eta}$ where $\eta \in D_{n,k}$ has a $\wedge$ in positions $i_1,\ldots,i_k$ and $\vee$ elsewhere; we will also use this notation.

For $2 \leq i \leq n$, let $\ell_i = e_i + qe_{i-1}$; let $\ell_1 = e_1$. We expand wedge products of the $\ell_i$ as usual, without any ``super'' sign rules.

\begin{proposition}
Let $\eta \in D_{n,k}$ be a $\wedge \vee$ sequence with $\wedge$ in positions $1 \leq i_i < \cdots < i_k \leq n$. The canonical basis element $v_{\eta}^{\diamondsuit}$ for $V^{\otimes n}$, defined in \cite[Theorem 4.2]{Sar-tensor}, is
\[
\{ \ell_{i_k} \wedge \cdots \wedge \ell_{i_1} : 1 \leq i_1 < \cdots < i_k \leq n\}.
\]
\end{proposition}

\begin{proof}
First, we notice that if $i \geq j > 1$, then
\[
\ell_i \wedge \ell_{i-1} \wedge \cdots \wedge \ell_j = \sum_{m=j-1}^i q^{m-j+1} e_i \wedge \cdots \wedge \hat{e_m} \wedge \cdots \wedge e_{j-1},
\]
while if $i \geq 1$ then
\[
\ell_i \wedge \cdots \wedge \ell_1 = e_i \wedge \cdots \wedge e_1.
\]
In general, a wedge product element as above is a product of such expressions over its consecutive $\ell_i$ intervals.

By \cite[Proposition 5.5]{Sar-tensor}, the canonical basis element $v_{\eta}^{\diamondsuit}$ arises from an evaluation of the lower fork diagram of $\eta$. The result follows from a comparison of the above formulas with \cite[Figure 1]{Sar-tensor}.
\end{proof}

If $\x \in V_l(n,k)$ corresponds to $\eta \in D_{n,k}$, we will write $v_{\x}^{\diamondsuit} = v_{\eta}^{\diamondsuit}$. For example, if $\x = \varnothing$ then $v_{\x}^{\diamondsuit} = v_1 \otimes \cdots \otimes v_1$, while if $\x = \{1,\ldots,n\}$ then $v_{\x}^{\diamondsuit} = v_0 \otimes \cdots \otimes v_0$.

\begin{example}\label{ex:CanonicalBasisV3}
Write $\wedge \vee$ sequences $\eta \in D_{n,k}$ as sequences of zeroes and ones, with $\wedge$ corresponding to zero and $\vee$ corresponding to one. The canonical basis elements for $V^{\otimes 3}$ are:
\begin{alignat*}{4}
 &v_{000}^{\diamondsuit} = v_{000} \qquad
 &&v_{100}^{\diamondsuit} = v_{100} + qv_{010} + q^2v_{001} \qquad
&&v_{010}^{\diamondsuit} = v_{010} + qv_{001} \qquad
&&v_{001}^{\diamondsuit} = v_{001}
 \\
&v_{110}^{\diamondsuit} = v_{110} + qv_{101} \qquad
&&v_{101}^{\diamondsuit} = v_{101} + qv_{011} \qquad
&&v_{011}^{\diamondsuit} = v_{011} \qquad
&&v_{111}^{\diamondsuit} = v_{111}
\end{alignat*}
One can check that these elements are invariant under the bar involution of \cite[Section 4.3]{Sar-tensor}.
\end{example}

% ----------------------------------------------------
\subsection{Categorification of \texorpdfstring{$V^{\otimes n}$}{Vn} via Sartori's algebras}
% ----------------------------------------------------

Let $\Bbbk$ be an arbitrary field.  Write $\cal{A}_{n,k}^{\Bbbk}:= \cal{A}_{n,k}^{\Z} \otimes_{\Z} \Bbbk$. For $\l \in D_{n,k}$, let $P(\l) = \cal{A}^{\Bbbk}_{n,k} \1_{\l}$, and let $L(\l)$ be the one-dimensional irreducible $\cal{A}_{n,k}^{\Bbbk}$-module such that $\1_{\l} L(\l) \neq 0$. For any graded module $M$ and integer $i$, let $q^i M$ denote $M$ with degrees shifted upwards by $i$. The following result is standard.
\begin{proposition}
The Grothendieck group
\[
K_0(\cal{A}_{n,k}^{\Bbbk}) := K(\cal{A}_{n,k}^{\Bbbk}{\rm -proj})
\]
of the abelian category of finitely generated projective graded left $\cal{A}_{n,k}^{\Bbbk}$ modules is a free $\Z[q,q^{-1}]$-module with basis given by the classes of indecomposable projective modules $[P(\lambda)]$ for $\lambda \in D_{n,k}$. The action of $q^{\pm 1}$ is given by $q^{\pm 1}[P] := [q^{\pm 1}P]$.

Similarly, the Grothendieck group
\[
G_0(\cal{A}_{n,k}^{\Bbbk}) := K(\cal{A}_{n,k}^{\Bbbk}{\rm -fmod})
\]
of the abelian category of finite dimensional graded left $\cal{A}_{n,k}^{\Bbbk}$ modules is a free $\Z[q,q^{-1}]$-module with basis given by the classes of simple modules $[L(\lambda)]$ for $\lambda \in D_{n,k}$.
\end{proposition}

\begin{remark}
Since $\cal{A}_{n,k}^{\Bbbk}$ is finite-dimensional over $\Bbbk$, we have $\cal{A}_{n,k}^{\Bbbk}{\rm -fmod} = \cal{A}_{n,k}^{\Bbbk}{\rm -gmod}$, the category of finitely generated graded left $\cal{A}_{n,k}^{\Bbbk}$ modules. Similarly, all objects of $\cal{A}_{n,k}^{\Bbbk}{\rm -proj}$ are finite-dimensional over $\Bbbk$. Note that all objects of $\cal{A}_{n,k}^{\Bbbk}{\rm -proj}$ can be thought of as objects of $\cal{A}_{n,k}^{\Bbbk}{\rm -fmod}$, but not conversely. We will see below that $K_0(\cal{A}_{n,k}^{\Bbbk})$ and $G_0(\cal{A}_{n,k}^{\Bbbk})$ can be identified over $\C(q)$.
\end{remark}

Write $K_0^{\C(q)}(\cal{A}_{n,k}^{\Bbbk}) := K_0(\cal{A}_{n,k}^{\Bbbk}) \otimes_{\Z[q,q^{-1}]} \C(q)$, and similarly for $G_0$.
\begin{definition}[Theorem 7.13 of \cite{Sar-tensor}]\label{def:SarVnIdent}
We identify $K_0^{\C(q)}(\cal{A}_{n,k}^{\Bbbk})$ with $\left(V^{\otimes n}\right)_k$ by identifying the basis element $[P(\l)]$ of $K_0^{\C(q)}(\cal{A}_{n,k}^{\Bbbk})$ with the canonical basis element $v_{\l}^{\diamondsuit}$.
\end{definition}

% ----------------------------------------------------
\subsection{Categorification of \texorpdfstring{$V^{\otimes n}$}{Vn} via Ozsv{\'a}th--Szab{\'o}'s algebras}
% ----------------------------------------------------

As in Section~\ref{sec:Homomorphism} above, write $\B_l^{\Bbbk}(n,k) := \B_l(n,k) \otimes_{\Z} \Bbbk$. For $\x \in V(n,k)$, let $P(\x) = \B^{\Bbbk}_l(n,k) \Ib_{\x}$, and let $L(\x)$ be the one-dimensional irreducible $\B_l^{\Bbbk}(n,k)$-module such that $\Ib_{\x} L(\x) \neq 0$.

Since $\B^{\Bbbk}_l(n,k)$ is positively graded and semisimple in degree zero, grading shifts of the modules $P(\x)$ form a complete set of isomorphism classes of indecomposable projective graded $\B^{\Bbbk}_l(n,k)$-modules, and grading shifts of the modules $L(\x)$ form a complete set of isomorphism classes of simple graded $\B^{\Bbbk}_l(n,k)$-modules.  Let $\B_l^{\Bbbk}(n,k){\rm -fmod}$ denote the category of finite dimensional graded left $\B_l^{\Bbbk}(n,k)$-modules and let $\B_l^{\Bbbk}(n,k){\rm -proj}$ denote the category of finitely generated projective graded left $\B_l^{\Bbbk}(n,k)$-modules.

\begin{corollary}
The Grothendieck group
\[
K_0(\B_l^{\Bbbk}(n,k)) := K(\B_l^{\Bbbk}(n,k){\rm -proj})
\]
is a free $\Z[q,q^{-1}]$-module with basis given by the classes of indecomposable projective modules $[P(\x)]$ for $\x \in V_l(n,k)$.

Similarly, the Grothendieck group
\[
G_0(\B_l^{\Bbbk}(n,k)) := K(\B_l^{\Bbbk}(n,k){\rm -fmod})
\]
is a free $\Z[q,q^{-1}]$-module with basis given by the classes of simple modules $[L(\x)]$ for $\x \in V_l(n,k)$.
\end{corollary}

As with $\cal{A}_{n,k}^{\Bbbk}$, we will see below that $K_0(\B_l^{\Bbbk}(n,k))$ and $G_0(\B_l^{\Bbbk}(n,k))$ can be identified over $\C(q)$. Write $K_0^{\C(q)}(\B_l^{\Bbbk}(n,k)) := K_0(\B_l^{\Bbbk}(n,k)) \otimes_{\Z[q,q^{-1}]} \C(q)$, and similarly for $G_0$.
\begin{definition}\label{def:OSzVnIdent}
We identify $K_0^{\C(q)}(\B_l^{\Bbbk}(n,k))$ with $\left(V^{\otimes n}\right)_k$ by identifying $[P(\x)] \in K_0^{\C(q)}(\B_l^{\Bbbk}(n,k))$ with the canonical basis element $v_{\x}^{\diamondsuit}$.
\end{definition}

\begin{remark}
When $\x = \varnothing$, Definition~\ref{def:OSzVnIdent} sends $[P_{\x}]$ to $v_1 \otimes \cdots \otimes v_1$, in contrast to the conventions in \cite{ManionDecat} where this class $[P_{\x}]$ is sent to $v_0 \otimes \cdots \otimes v_0$. We will compare the identification of Definition~\ref{def:OSzVnIdent} with the one given in \cite{ManionDecat} in Section~\ref{sec:ConventionComparison}.
\end{remark}

% ----------------------------------------------------
\subsection{Relating the categorifications by projection and inflation}
% ----------------------------------------------------

If $P$ is a finitely generated projective $\B_l^{\Bbbk}(n,k)$-module, define the projection $\pr(P)$ of $P$ to be the $\cal{A}^{\Bbbk}_{n,k}$-module $\cal{A}^{\Bbbk}_{n,k} \otimes_{\B_l^{\Bbbk}(n,k)} P$, where $\cal{A}^{\Bbbk}_{n,k}$ is a right module over $\B_l^{\Bbbk}(n,k)$ via the quotient map $\Xi$ of Section~\ref{sec:Homomorphism}.

If $M$ is a finite-dimensional $\cal{A}^{\Bbbk}_{n,k}$-module, define the inflation $\infl(M)$ of $M$ to be the $\B_l^{\Bbbk}(n,k)$-module with the same underlying set as $M$, with an action of $\B_l^{\Bbbk}(n,k)$ given by applying $\Xi$ to get an element of $\cal{A}^{\Bbbk}_{n,k}$ and then acting on $M$. We get functors
\begin{equation}
\pr \maps \B_l^{\Bbbk}(n,k){\rm -proj} \to \Bqk{\rm - proj},
\qquad \infl \maps \Bqk{\rm - fmod} \to \B_l^{\Bbbk}(n,k){\rm -fmod};
\end{equation}
note that $\infl$ preserves exact sequences, since it acts as the identity on underlying sets and functions between them.

\begin{remark}
Heegaard Floer homologists may be most familiar with $\pr$ and $\infl$ as special cases of the induction and restriction functors discussed in \cite[Section 2.4.2]{LOTBimodules}, which make sense in a general $\A_{\infty}$ setting.
\end{remark}

\begin{theorem}\label{thm:CategorificationsCompatible}
The projection functor induces an isomorphism from $K_0(\B_l^{\Bbbk}(n,k))$ to $K_0(\cal{A}_{n,k}^{\Bbbk})$, compatible with the identifications of both Grothendieck groups over $\C(q)$ with $\left(V^{\otimes n}\right)_k$ in Definitions~\ref{def:OSzVnIdent} and \ref{def:SarVnIdent}.
\end{theorem}

\begin{proof}
By the above discussion, the result follows since projection sends basis elements $[P(\x)]$ to basis elements $[P(\l)]$ where $\l \in D_{n,k}$ corresponds to $\x \in V_l(n,k)$.
\end{proof}

On the other hand, Sartori defines interesting families of modules over $\cal{A}_{n,k}$, and one can obtain similar families of modules over $\B_l(n,k)$ by inflation. Since inflation sends simples $L(\l)$ to simples $L(\x)$, we have the following result.
\begin{corollary}\label{cor:InflIsIso}
The inflation functor gives us an isomorphism from $G_0(\cal{A}_{n,k}^{\Bbbk})$ to $G_0(\B_l^{\Bbbk}(n,k))$.
\end{corollary}
As mentioned above, we will be able to identify $K_0$ and $G_0$ over $\C(q)$ on both sides. From the identification $G^{\C(q)}_0(\B_l^{\Bbbk}(n,k)) \cong K_0^{\C(q)}(\B_l^{\Bbbk}(n,k))$ we define below, the inflated modules will give us classes in $K_0^{\C(q)}(\B_l^{\Bbbk}(n,k))$ and thus elements of $\left(V^{\otimes n}\right)_k$ by Definition~\ref{def:OSzVnIdent}.
\begin{warning}
Under the identifications of $K_0^{\C(q)}$ and $G_0^{\C(q)}$ on both sides, the inflation isomorphism is not the inverse of the projection isomorphism. Rather, they are related by a scalar multiple; see Proposition~\ref{prop:PhiInfl} below.
\end{warning}

To understand which elements of $V^{\otimes n}$ we get from the inflated modules, we need to study the identifications of $K_0^{\C(q)}$ and $G_0^{\C(q)}$ for Ozsv{\'a}th--Szab{\'o}'s and Sartori's algebras; this task will occupy the next few sections.

% ------------------------------------------------------
\subsection{The Sartori bilinear form for \texorpdfstring{$V^{\otimes n}$}{Vn}}
% ------------------------------------------------------

For a finitely generated projective (graded) left module $P$ over $\cal{A}^{\Bbbk}_{n,k}$, define ${^{\vee}}P^{\psi_S}$ to be the dual ${^{\vee}}P = \Hom_{\cal{A}^{\Bbbk}_{n,k}}(P,\cal{A}^{\Bbbk}_{n,k})$ of $P$ with its action of $\cal{A}^{\Bbbk}_{n,k}$ twisted by $\psi_S$.  Since ${^{\vee}}P$ is a right $\cal{A}^{\Bbbk}_{n,k}$-module, ${^{\vee}}P^{\psi_S}$ is a left $\cal{A}^{\Bbbk}_{n,k}$-module, like $P$ itself. When $P = q^i P(\l)$, we have ${^{\vee}}P^{\psi} \cong q^{-i} P(\l)$.

\begin{definition}\label{def:SartoriCatPairing}
Let $[,]_S$ be the $\Z[q,q^{-1}]$-bilinear pairing
\begin{align*}
[,]_S\colon K_0(\cal{A}^{\Bbbk}_{n,k}) &\times K_0(\cal{A}^{\Bbbk}_{n,k}) \to \Z[q,q^{-1}] \\
[P], [P'] &\mapsto \dimq(\Hom_{\cal{A}^{\Bbbk}_{n,k}}({^{\vee}}P^{\psi_S},P')).
\end{align*}
\end{definition}
Note that we have $[[P(\mu)],[P(\l)]]_S = \dimq 1_{\mu} \cal{A}^{\Bbbk}_{n,k} 1_{\l}$.

\begin{remark}\label{rem:SartoriStylePairingDef}
The form defined in \cite[Proposition 7.12]{Sar-tensor}, restricted to finitely-generated projective modules, is given by
\[
[[P],[P']]_S = \overline{\dimq(\Hom_{\cal{A}^{\Bbbk}_{n,k}}(P, (P')^*))}
\]
where $(P')^* = \Hom_{\Bbbk}(P',\Bbbk)$ with
\[
(a \phi)(x) = \phi(\psi(a) x)
\]
for $\phi \in (P')^*$, $a \in \cal{A}^{\Bbbk}_{n,k}$, and $x \in P'$. The operation $\overline{(\cdot)}$ is the involution of $\Z[q,q^{-1}]$ given by $q \mapsto q^{-1}$. One can check that this definition of the form is equivalent to Definition~\ref{def:SartoriCatPairing} by comparing the values on indecomposable projectives $[P(\l)]$.
\end{remark}

By Definition~\ref{def:SarVnIdent}, we get a $\C(q)$-bilinear pairing on $V^{\otimes n}$. We can describe this pairing as follows.
\begin{definition}
Sartori's bilinear form $(,)_S$ on $V^{\otimes n}$ has matrix $(k)_{q^2}^!$ times the identity in the standard basis of the weight space $\left(V^{\otimes n}\right)_k$.
\end{definition}

\begin{proposition}\label{prop:SartoriFormCat}
The identification of $K_0^{\C(q)}(\cal{A}^{\Bbbk}_{n,k})$ and $\left(V^{\otimes n}\right)_k$ from Definition~\ref{def:SarVnIdent} identifies $[,]_S$ with $(,)_S$.
\end{proposition}

\begin{proof}
This proposition is a consequence of \cite[Proposition 7.12]{Sar-tensor} and Remark~\ref{rem:SartoriStylePairingDef}.
\end{proof}

It follows from Proposition~\ref{prop:SartoriFormCat} that $[,]_S$ is perfect after tensoring with $\C(q)$, so it gives us an identification of $K_0^{\C(q)}(\cal{A}^{\Bbbk}_{n,k})$ with its dual $(K_0^{\C(q)}(\cal{A}^{\Bbbk}_{n,k}))^*$.

We can also consider the $\Z[q,q^{-1}]$-bilinear pairing
\begin{align*}
K_0(\cal{A}^{\Bbbk}_{n,k}) &\times G_0(\cal{A}^{\Bbbk}_{n,k}) \to \Z[q,q^{-1}] \\
[P], [M] &\mapsto \dimq(\Hom_{\cal{A}^{\Bbbk}_{n,k}}({^{\vee}}P^{\psi_S},M)).
\end{align*}
The matrix for this pairing in the bases of projectives for $K_0$ and simples for $G_0$ is the identity matrix, so the pairing allows us to identify $K_0$ and $G_0^*$ (or vice-versa) over $\Z[q,q^{-1}]$, and thus over $\C(q)$. Combining this identification with the isomorphism $K_0^{\C(q)}(\cal{A}^{\Bbbk}_{n,k}) \cong (K_0^{\C(q)}(\cal{A}^{\Bbbk}_{n,k}))^*$ from $[,]_S$, we get an identification of $K_0^{\C(q)}(\cal{A}^{\Bbbk}_{n,k})$ with $G_0^{\C(q)}(\cal{A}^{\Bbbk}_{n,k})$. The pairing on $G_0^{\C(q)}(\cal{A}^{\Bbbk}_{n,k})$ induced by $[,]_S$ can be described by
\[
[M,N]_S = \overline{\chi_q \left(\Ext^*_{\cal{A}^{\Bbbk}_{n,k}}(M,N^*) \right)}
\]
where $N^*$ is defined as in Remark~\ref{rem:SartoriStylePairingDef} and $\chi_q$ is the $q$-graded Euler characteristic.

The basis of simples $\{[L(\l)]\}$ for $G_0^{\C(q)}(\cal{A}^{\Bbbk}_{n,k})$ gives us a basis for $K_0^{\C(q)}(\cal{A}^{\Bbbk}_{n,k})$ under the above identification. Under the identification of this latter space with $V^{\otimes n}$, the basis of simples corresponds to Sartori's dual canonical basis, as we review in Section~\ref{sec:DualCanonical} below.

The change-of-basis matrix from projectives to simples on $K_0^{\C(q)}(\cal{A}^{\Bbbk}_{n,k})$ is the matrix for $[,]_S$ on $K_0^{\C(q)}(\cal{A}^{\Bbbk}_{n,k})$ in the basis of projectives. We compute this matrix below; equivalently, we compute the matrix for $(,)_S$ in the canonical basis of $V^{\otimes n}$.

Recall the nonsymmetric quantum integers defined in \eqref{eq:nonsym-qnum}.

\begin{proposition}\label{prop:SartoriFormInCanonicalBasis}
For $\mu,\l \in D_{n,k}$, we have $(v_{\mu}^{\diamondsuit},v_{\l}^{\diamondsuit})_S = 0$ if $\mu$ and $\l$ are too far. If $\mu,\l$ are not too far and correspond to $\x,\y \in V_l(n,k)$, we have
\[
(v_{\mu}^{\diamondsuit},v_{\l}^{\diamondsuit})_S = q^{d}(k)_{q^2}^!\prod_{i=1}^{n-k} (l_i)_{q^2}
\]
where $l_1,\ldots, l_{n-k}$ are the lengths of the generating intervals between $\x$ and $\y$ and $d = \sum_{i=1}^k |\wedge^{\l}_i - \wedge^{\mu}_i|$.
\end{proposition}

\begin{proof}
This proposition follows from Corollary~\ref{cor:ForkCount} and Proposition~\ref{prop:SartoriFormCat}.
\end{proof}

\begin{example}
The matrix for $(,)_S$ on $(V^{\otimes 3})_k$ for $0 \leq k \leq 3$ is given below.
\begin{center}
\setlength{\kbcolsep}{-4pt}
\begin{small}
\begin{tabular}{|c|c|c|c|}
  \hline
  % after \\: \hline or \cline{col1-col2} \cline{col3-col4} ...
  $ \xy (0,0)*+{(  V^{\otimes 3})_3} \endxy$ & $(V^{\otimes 3})_2$
  & $\xy (0,0)*+{ (V^{\otimes 3})_1} \endxy$ & $(V^{\otimes 3})_0$
  \\
  \hline  \hline
  $
  (3)^!_{q^2} \left( \kbordermatrix{
    & \mkern-6mu v_{000}^{\diamondsuit} \mkern-6mu\\
    v_{000}^{\diamondsuit} & \mkern-6mu 1 \mkern-6mu
    }   \right)$
&
  $(2)^!_{q^2} \left( \kbordermatrix{
    & v_{100}^{\diamondsuit} & v_{010}^{\diamondsuit} & v_{001}^{\diamondsuit} \\
    v_{100}^{\diamondsuit} & 1 + q^2 + q^4 & q + q^3 & q^2 \\
    v_{010}^{\diamondsuit} & q + q^3 & 1 + q^2 & q \\
    v_{001}^{\diamondsuit} & q^2 & q & 1
    } \right)$  &
    $\kbordermatrix{
& v_{110}^{\diamondsuit} & v_{101}^{\diamondsuit} & v_{011}^{\diamondsuit} \\
v_{110}^{\diamondsuit} & 1+q^2 & q & 0 \\
v_{101}^{\diamondsuit} & q & 1+q^2 & q \\
v_{011}^{\diamondsuit} & 0 & q & 1
}$ &
 $\kbordermatrix{
& \mkern-6mu v_{111}^{\diamondsuit} \mkern-6mu\\
v_{111}^{\diamondsuit} & \mkern-6mu 1 \mkern-6mu
}$\\
  \hline
\end{tabular}
 \end{small}
 \end{center}

\end{example}

% ------------------------------------------------------
\subsection{The Ozsv{\'a}th--Szab{\'o} bilinear form for \texorpdfstring{$V^{\otimes n}$}{Vn}}
% ------------------------------------------------------

We can define a similar bilinear pairing $[,]_{OSz}$ on $K_0(\B_l^{\Bbbk}(n,k))$.
\begin{definition}
Let $[,]_{OSz}$ be the $\Z[q,q^{-1}]$-bilinear pairing
\begin{align*}
[,]_S\colon K_0(\B_l^{\Bbbk}(n,k)) &\times K_0(\B_l^{\Bbbk}(n,k)) \to \Z[q,q^{-1}] \\
[P], [P'] &\mapsto \dimq(\Hom_{\B_l^{\Bbbk}(n,k)}({^{\vee}}P^{\psi_{OSz}},P')).
\end{align*}
\end{definition}
Note that we have $[[P(\x)],[P(\y)]]_{OSz} = \dimq \Ib_{\x} \B_l^{\Bbbk}(n,k) \Ib_{\y}$. By Definition~\ref{def:OSzVnIdent}, we get another $\C(q)$-bilinear pairing on $V^{\otimes n}$. This pairing has a simple description in the standard basis, analogous to Sartori's. To see this, we compute its matrix in the canonical basis, or equivalently the matrix for $[,]_{OSz}$ in the basis of projectives. This amounts to computing $\dimq \Ib_{\x} \B_l^{\Bbbk}(n,k) \Ib_{\y}$ for $\x,\y \in V_l(n,k)$.

\begin{proposition}\label{prop:OSzFormInCanonicalBasis}
For $\x,\y \in V(n,k)$ that are not too far, the graded dimension of $\Ib_{\x} \B^{\Bbbk}_l(n,k) \Ib_{\y}$ is $q^d\frac{\prod_{i=1}^{n-k} (l_i)_{q^2}}{(1-q^2)^k}$, where $l_1,\ldots, l_{n-k}$ are the lengths of the generating intervals between $\x$ and $\y$ and $d = \sum_{i=1}^k |x_i - y_i|$. If $\x$ and $\y$ are too far, this graded dimension is zero.
\end{proposition}

\begin{proof}
Since $\Ib_{\x} \B^{\Bbbk}_l(n,k) \Ib_{\y}$ can be viewed as $\Bbbk[U_1,\ldots,U_n]$ modulo the ideal generated by monomials of generating intervals between $\x$ and $\y$, its graded dimension is some power $q^d$ (where $d$ is the degree of the minimal generator) times $\frac{\prod_{i=1}^{n-k}(1-q^{2l_i})}{(1-q^2)^n}$, which equals $\frac{\prod_{i=1}^{n-k} (l_i)_{q^2}}{(1-q^2)^k}$. The degree $d$ of the minimal generator is $\sum_{i=1}^k |x_i - y_i|$.
\end{proof}

\begin{example}
The matrices for $[,]_{OSz}$ on $K_0(\B_l^{\Bbbk}(n,k))$ are given in the following table.
\begin{center}
\setlength{\kbcolsep}{-4pt}
\begin{small}
\begin{tabular}{|c|c|c|c|}
  \hline
  % after \\: \hline or \cline{col1-col2} \cline{col3-col4} ...
  $ \xy (0,0)*+{ K_0(\B_l^{\Bbbk}(n,3) } \endxy$
  & $K_0(\B_l^{\Bbbk}(n,2))$
  & $\xy (0,0)*+{K_0(\B_l^{\Bbbk}(n,1)} \endxy$
  & $K_0(\B_l^{\Bbbk}(n,0))$
  \\
  \hline  \hline
  $
  \frac{1}{(1-q^2)^3}
  [1]   $
&
  $\frac{1}{(1-q^2)^2}
	\left( \kbordermatrix{
& {\tiny [P_{1,2}]}&\mkern-6mu  {\tiny[P_{0,2}]} \mkern-6mu  & \mkern-6mu  {\tiny[P_{0,1}]} \mkern-8mu   \\
{[P_{1,2}]}  &\mkern-6mu  1 + q^2 + q^4 \mkern-6mu &\mkern-6mu  q + q^3 \mkern-6mu &\mkern-6mu  q^2 \mkern-6mu  \\
{[P_{0,2}]}  &\mkern-6mu  q + q^3       \mkern-6mu &\mkern-6mu  1 + q^2 \mkern-6mu &\mkern-6mu  q \mkern-6mu  \\
{[P_{0,1}]}  &\mkern-6mu  q^2           \mkern-6mu &\mkern-6mu  q       \mkern-6mu &\mkern-6mu  1 \mkern-6mu
} \right)$  &
    $\frac{1}{1-q^2}
		\left( \kbordermatrix{
& {[P_{2}]}  & \mkern-6mu {[P_{1}]}  \mkern-6mu &\mkern-6mu  {[P_{0}]} \mkern-6mu  \\
{[P_{2}]}  &\mkern-6mu  1+q^2 \mkern-6mu &\mkern-6mu  q     \mkern-6mu &\mkern-6mu  0\mkern-6mu  \\
{[P_{1}]}  &\mkern-6mu  q     \mkern-6mu &\mkern-6mu  1+q^2 \mkern-6mu &\mkern-6mu  q\mkern-6mu  \\
{[P_{0}]}  &\mkern-6mu  0     \mkern-6mu &\mkern-6mu  q     \mkern-6mu &\mkern-6mu  1\mkern-6mu
} \right)$ &
 $[1]$\\
  \hline
\end{tabular}
 \end{small}
 \end{center}

where the first matrix is in the basis $[P_{0,1,2}]$ and the last matrix is in the basis $P_{[\emptyset]}$.
\end{example}

\begin{corollary}\label{cor:OSzCatFormVsSarCatForm}
We have
\[
[[\pr(P)],[\pr(P')]]_S = (k)_{q^2}^!(1-q^2)^k [[P],[P']]_{OSz}
\]
for all objects $P,P'$ of $\B^{\Bbbk}_l(n,k){\rm -proj}$.
\end{corollary}
The above corollary motivates the following definition.

\begin{definition}
The Ozsv{\'a}th--Szab{\'o} bilinear form $(,)_{OSz}$ on $V^{\otimes n}$ has matrix $\frac{1}{(1-q^2)^k}$ times the identity in the standard basis of the weight space $\left(V^{\otimes n}\right)_k$.
\end{definition}

\begin{corollary}\label{cor:OSzBilinearForm}
The identification of $K_0^{\C(q)}(\B_l^{\Bbbk}(n,k))$ and $\left(V^{\otimes n}\right)_k$ from Definition~\ref{def:OSzVnIdent} identifies $[,]_{OSz}$ with $(,)_{OSz}$.
\end{corollary}

\begin{proof}
We have
\begin{align*}
[[P(\x)],[P(\x')]]_{OSz} &= \frac{1}{(k)_{q^2}^!(1-q^2)^k } [[\pr(P(\x))],[\pr(P(\x'))]]_S \\
&= \frac{1}{(k)_{q^2}^!(1-q^2)^k } (v_{\x}^{\diamondsuit}, v_{\x'}^{\diamondsuit})_S
= (v_{\x}^{\diamondsuit},v_{\x'}^{\diamondsuit})_{OSz};
\end{align*}
the first equality follows from Corollary~\ref{cor:OSzCatFormVsSarCatForm}, the second follows from Proposition~\ref{prop:SartoriFormCat}, and the third follows from the definitions of $(,)_S$ and $(,)_{OSz}$.
\end{proof}

By Corollary~\ref{cor:OSzBilinearForm}, $[,]_{OSz}$ is perfect over $\C(q)$. Thus, as above, we get an identification of $K_0^{\C(q)}(\B_l^{\Bbbk}(n,k))$ with $(K_0^{\C(q)}(\B_l^{\Bbbk}(n,k)))^*$ and thereby with $G_0^{\C(q)}(\B_l^{\Bbbk}(n,k))$. Via this identification, the basis of simples for $G_0^{\C(q)}(\B_l^{\Bbbk}(n,k))$ gives us as basis of $K_0^{\C(q)}(\B_l^{\Bbbk}(n,k))$ and thus of $V^{\otimes n}$; the change of basis matrix from the basis of projectives (or canonical basis) to this basis is the matrix for $[,]_{OSz}$ in the basis of projectives. Below we will identify the basis of simples for $K_0^{\C(q)}(\B_l^{\Bbbk}(n,k))$ with the Ozsv{\'a}th--Szab{\'o} dual canonical basis of $V^{\otimes n}$ (to be defined).

The pairing on $G_0^{\C(q)}(\B_l^{\Bbbk}(n,k))$ induced by $[,]_{OSz}$ can be described by
\[
[M,N]_{OSz} = \overline{\chi_q \left(\Ext^*_{\B_l^{\Bbbk}(n,k)}(M,N^*) \right)};
\]
again, $N^*$ is defined as in Remark~\ref{rem:SartoriStylePairingDef} and $\chi_q$ is the $q$-graded Euler characteristic.

% ----------------------------------------------------
\subsection{Dual standard and dual canonical bases}\label{sec:DualCanonical}
% ----------------------------------------------------

\begin{definition}
From the standard and canonical bases for $V^{\otimes n}$, we obtain four bases by dualizing with respect to the above two bilinear forms $(,)_S$ and $(,)_{OSz}$. We will call these the Sartori dual standard, Sartori dual canonical, Ozsv{\'a}th--Szab{\'o} dual standard, and Ozsv{\'a}th--Szab{\'o} dual canonical bases. We will use the following notation; let $\x \in V_l(n,k)$.
\begin{itemize}
\item The Sartori dual standard basis element associated to $\x$ will be denoted $v_{\x}^{\clubsuit}$.
\item The Sartori dual canonical basis element associated to $\x$ will be denoted $v_{\x}^{\heartsuit}$.
\item The Ozsv{\'a}th--Szab{\'o} dual standard basis element associated to $\x$ will be denoted $v_{\x}^{\clubsuit \clubsuit}$.
\item The Ozsv{\'a}th--Szab{\'o} dual canonical basis element associated to $\x$ will be denoted $v_{\x}^{\heartsuit \heartsuit}$.
\end{itemize}
\end{definition}
The matrices for the bilinear forms in these dual bases are the inverses of the matrices in the original bases.

\begin{example}
Labeling $\wedge \vee$ sequences as in Example~\ref{ex:CanonicalBasisV3}, Sartori's dual canonical basis for $V^{\otimes 3}$ is
\begin{alignat*}{4}
 &v_{000}^{\heartsuit} = \frac{1}{(3)^!_{q^2}} v_{000}\qquad
 &&v_{100}^{\heartsuit} = \frac{1}{(2)^!_{q^2}} v_{100} \qquad
&&v_{010}^{\heartsuit} = \frac{1}{(2)^!_{q^2}} \left( v_{010} - qv_{100} \right) \qquad
&&v_{001}^{\heartsuit} = \frac{1}{(2)^!_{q^2}} \left( v_{001} - qv_{010} \right)
 \\
&v_{110}^{\heartsuit} = v_{110} \qquad
&&v_{101}^{\heartsuit} = v_{101} - qv_{110} \qquad
&&v_{011}^{\heartsuit} = v_{011} - qv_{101} + q^2 v_{110} \qquad
&&v_{111}^{\heartsuit} = v_{111}
\end{alignat*}
The Ozsv{\'a}th--Szab{\'o} dual canonical basis elements $v_{\x}^{\heartsuit \heartsuit}$ are obtained by replacing the coefficients $\frac{1}{(k)^!_{q^2}}$ with $(1-q^2)^k$. We have $v_{\x}^{\clubsuit} = \frac{1}{(k)^!_{q^2}} v_{\x}$ and $v_{\x}^{\clubsuit \clubsuit} = (1-q^2)^k v_{\x}$.
\end{example}

Our identification of $K_0^{\C(q)}(\cal{A}^{\Bbbk}_{n,k})$ and $G_0^{\C(q)}(\cal{A}^{\Bbbk}_{n,k})$ goes via $(K_0^{\C(q)}(\cal{A}^{\Bbbk}_{n,k}))^*$; the basis of simples for $G_0$ naturally corresponds to the dual basis to the basis of projectives for $K_0$. Under the further identification of $K_0^{\C(q)}$ with its dual, this dual basis gets sent to the dual to the basis of projectives for $K_0^{\C(q)}$ under the bilinear form $[,]_S$. Identifying $(K_0^{\C(q)}, \, [,]_S)$ with $(V^{\otimes n}, \, (,)_S)$ by Definition~\ref{def:SarVnIdent}, we see that the basis of simples for $G_0$ gets sent to the basis of $V^{\otimes n}$ that is dual to the canonical basis under $(,)_S$, i.e. the Sartori dual canonical basis. Similar reasoning applies in the Ozsv{\'a}th--Szab{\'o} case, proving the following corollary.

\begin{corollary}\label{cor:PreliminaryBasisCor}
Under the identification $K_0^{\C(q)}(\cal{A}^{\Bbbk}_{n,k}) \cong \left( V^{\otimes n} \right)_k$ of Definition~\ref{def:SarVnIdent}, we have
\begin{align}
		\{\text{indecomposable projective modules } P(\l) \} &\leftrightarrow \text{canonical basis elements } v_{\l}^{\diamondsuit}
    \nn \\ \nn
    \{ \text{simple modules } L(\l) \}  &\leftrightarrow \text{Sartori dual canonical basis elements } v_{\l}^{\heartsuit}.
\end{align}
Under the identification $K_0^{\C(q)}(\B_l^{\Bbbk}(n,k)) \cong \left( V^{\otimes n} \right)_k$ of Definition~\ref{def:OSzVnIdent}, we have
\begin{align}
		\{\text{indecomposable projective modules } P(\x) \} &\leftrightarrow \text{canonical basis elements } v_{\x}^{\diamondsuit}
    \nn \\ \nn
    \{ \text{simple modules } L(\x) \}  &\leftrightarrow \text{Ozsv{\'a}th--Szab{\'o} dual canonical basis elements } v_{\x}^{\heartsuit \heartsuit}.
\end{align}
\end{corollary}

% ----------------------------------------------------
\subsection{Sartori's categorification of standard bases for \texorpdfstring{$V^{\otimes n}$}{Vn}} \label{sec:modules}
% ----------------------------------------------------

Adding to Corollary~\ref{cor:PreliminaryBasisCor}, Sartori also defines classes of modules over $\cal{A}_{n,k}$ categorifying standard and (Sartori) dual standard basis elements. We review these modules below.

Sartori shows \cite[Proposition 5.18 and Theorem 5.24]{Sar-diagrams} that the algebras $\cal{A}_{n,k}$ are {\em graded cellular} \cite{GL,HM} and {\em properly stratified} \cite{FKM}.  These properties can be formally deduced from the algebras' connection with category $\cal{O}$, but Sartori gives an independent proof with an explicit description of projective modules $P(\mu)$, standard modules $\Delta(\mu)$, cellular modules (including proper standard modules $\bar{\Delta}(\mu)$), and simple modules $L(\mu)$ for $\mu \in D_{n,k}$.

In addition, projective modules have explicit filtrations whose subquotients are standard modules \cite[Proposition 5.19]{Sar-diagrams}.  Standard modules admit filtrations by proper standard modules~\cite[Proposition 5.21]{Sar-diagrams}, and proper standard modules admit filtrations by simples~\cite[Proposition 5.22]{Sar-diagrams}.  These modules and filtrations give rise to various bases and change of basis formulas in the Grothendieck group of $\cal{A}_{n,k}$.

Proposition~\ref{prop:AZfree} gives a basis for $\cal{A}_{n,k}^{\Z}$ as a free $\Z$-module.  Consequently, it is immediate from \cite[Proposition 5.18]{Sar-diagrams} that $\cal{A}_{n,k}^{\Z}$ is graded cellular over $\Z$.   Furthermore, the four classes of modules $L(\l)$,  $\Delta(\l)$,   $\bar{\Delta}(\l)$, and  $P(\l)$ for $\l \in D_{n,k}$ all can be defined integrally, giving modules for $\cal{A}_{n,k}^{\Z}$.  It follows that, working over an arbitrary field $\Bbbk$, the algebras $\cal{A}_{n,k}^{\Z} \otimes_{\Z} \Bbbk$ are properly stratified algebras.

The filtrations described above along with the properly stratified structure on $\cal{A}_{n,k}^{\Bbbk}$ give rise to identities in $G_0(\cal{A}_{n,k}^{\Bbbk})$:
\begin{align}
  [P(\l)] = \sum_{\mu \in D_{n,k}} d_{\l,\mu} [\Delta(\mu)]
, \qquad
 [\bar{\Delta}(\mu)] = \sum_{\mu \in D_{n,k}} d_{\l,\mu} [L(\l)]
, \qquad
 [ \Delta(\mu)] = [k]_0^! [\bar{\Delta}(\mu)]
\end{align}
where
\[
d_{\l,\mu} :=
\left\{
  \begin{array}{ll}
    q^{\deg(\und{\l}\mu)}, & \hbox{if $\und{\l}\mu$ is an oriented lower fork diagram} \\
    0, & \hbox{otherwise,}
  \end{array}
\right.
\]
and
\[
[k]_0 = \frac{q^{2k}-1}{q^2-1} \quad \text{and} \quad [k]_0^! := [k]_0 [k-1]_0 \dots [1]_0.
\]

\begin{remark}\label{rem:SartoriStyleK0G0Ident}
\emph{A priori}, one can get a class in $G_0(\cal{A}_{n,k}^{\Bbbk})$ from a finitely generated projective module $P$ in two ways. Since $\cal{A}_{n,k}^{\Bbbk}$ and thus $P$ is finite-dimensional, $P$ is an object of $\cal{A}_{n,k}^{\Bbbk}{\rm-fmod}$ and thus gives a class in its Grothendieck group. On the other hand, one can use the above isomorphism $K_0^{\C(q)}(\cal{A}_{n,k}^{\Bbbk}) \cong G_0^{\C(q)}(\cal{A}_{n,k}^{\Bbbk})$ to get a class in $G_0^{\C(q)}(\cal{A}_{n,k}^{\Bbbk})$. In fact, this class makes sense in $G_0(\cal{A}_{n,k}^{\Bbbk})$, and it agrees with the class of $P$ defined in the first way since they both have the same expansion in the basis of simples.
\end{remark}

The cellular structure can be used to show that the matrices $d_{\l,\mu}$ are upper triangular with determinant 1, so that they are invertible over $\Z[q,q^{-1}]$   and the classes of proper standard modules $\{\bar{\Delta}(\l) \mid \l \in D_{n,k} \}$ also form a basis for $G_0(\cal{A}_{n,k}^{\Bbbk})$. Since $[k]^!_0$ is not invertible over $\Z[q,q^{-1}]$ in general, the classes of projective modules $\{ P(\l) \mid \l \in D_{n,k} \}$ and standard modules $\{ \Delta(\l) \mid \l \in D_{n,k} \}$ do not generate $G_0(\cal{A}_{n,k}^{\Bbbk})$ over $\Z[q,q^{-1}]$.

The situation improves if we pass from $\Z[q,q^{-1}]$ to $\C(q)$. Each of the four classes of modules above gives a basis for $G_0^{\C(q)}(\cal{A}_{n,k}^{\Bbbk})$ over $\C(q)$. In particular, the classes $[P(\l)]$ give a basis; thus, we can identify $G_0^{\C(q)}(\cal{A}_{n,k}^{\Bbbk})$ with $K_0^{\C(q)}(\cal{A}_{n,k}^{\Bbbk})$ by identifying $[P(\l)]$ with $[P(\l)]$ on either side, agreeing with our previous identification as in Remark~\ref{rem:SartoriStyleK0G0Ident}. We have four bases for $K_0^{\C(q)}(\cal{A}_{n,k}^{\Bbbk})$ corresponding to the four bases for $G_0^{\C(q)}(\cal{A}_{n,k}^{\Bbbk})$.

\begin{theorem}[Theorem 7.13 of \cite{Sar-tensor}]\label{thm:SartoriG0Ident}
Under the identification $K_0^{\C(q)}(\cal{A}^{\Bbbk}_{n,k}) \cong \left( V^{\otimes n} \right)_k$ of Definition~\ref{def:SarVnIdent}, we have
  \begin{align}
		\{\text{indecomposable projective modules } P(\l) \} &\leftrightarrow \text{canonical basis elements } v_{\l}^{\diamondsuit}
		\nn  \\
    \{ \text{standard modules } \Delta(\l) \}  &\leftrightarrow \text{standard basis elements } v_{\l}
    \nn  \\
    \{\text{proper standard modules } \overline{\Delta}(\l) \}  &\leftrightarrow \text{Sartori dual standard basis elements } v_{\l}^{\clubsuit}
    \nn \\ \nn
    \{ \text{simple modules } L(\l) \}  &\leftrightarrow \text{Sartori dual canonical basis elements } v_{\l}^{\heartsuit}.
  \end{align}
\end{theorem}

% ----------------------------------------------------
\subsection{Classes in \texorpdfstring{$K_0^{\C(q)}(\B^{\Bbbk}_l(n,k))$}{the Grothendieck group} from inflated Sartori modules}\label{sec:ClassesOfInflatedModules}
% ----------------------------------------------------

As in Corollary~\ref{cor:InflIsIso}, inflation gives an isomorphism from $G_0(\cal{A}^{\Bbbk}_{n,k})$ to $G_0(\B^{\Bbbk}_l(n,k))$. Passing to $\C(q)$, we can compare inflation with the isomorphism
\[
\Phi: G_0^{\C(q)}(\cal{A}^{\Bbbk}_{n,k}) \xrightarrow{\cong} K_0^{\C(q)}(\cal{A}^{\Bbbk}_{n,k}) \xrightarrow{\pr^{-1}} K_0^{\C(q)}(\B^{\Bbbk}_l(n,k)) \xrightarrow{\cong} G_0^{\C(q)}(\B^{\Bbbk}_l(n,k)).
\]

\begin{proposition}\label{prop:PhiInfl}
With $\Phi$ defined as above, we have $\infl = (k)^!_{q^2}(1-q^2)^k \Phi$.
\end{proposition}

\begin{proof}
The formula follows from the fact that the Sartori bilinear form $(,)_S$ is $(k)^!_{q^2}(1-q^2)^k$ times the Ozsv{\'a}th--Szab{\'o} form $(,)_{OSz}$.
\end{proof}

The finite-dimensional modules over $\B^{\Bbbk}_l(n,k)$ defined above give us classes in $K_0^{\C(q)}(\B^{\Bbbk}_l(n,k))$ via the identification of $K_0^{\C(q)}(\B^{\Bbbk}_l(n,k))$ with $G_0^{\C(q)}(\B^{\Bbbk}_l(n,k))$. By Definition~\ref{def:OSzVnIdent}, we get elements of $\left(V^{\otimes n} \right)_k$.

\begin{theorem}\label{thm:BorderedCategorificationOfBases}
Let $\l \in D_{n,k}$. Under the above identification, the modules over $\B^{\Bbbk}_l(n,k)$ obtained by inflating Sartori's indecomposable projective, standard, proper standard, and simple modules $P(\l)$, $\Delta(\l)$, $\overline{\Delta}(\l)$, and $L(\l)$ categorify
\begin{itemize}
\item $(k)^!_{q^2}(1-q^2)^k$ times the canonical basis element $v_{\l}^{\diamondsuit}$,
\item $(k)^!_{q^2}(1-q^2)^k$ times the standard basis element $v_{\l}$,
\item the Ozsv{\'a}th--Szab{\'o} dual standard basis element $v_{\l}^{\clubsuit \clubsuit}$, and
\item the Ozsv{\'a}th--Szab{\'o} dual canonical basis element $v_{\l}^{\heartsuit \heartsuit}$
\end{itemize}
of $\left( V^{\otimes n} \right)_k$ respectively.
\end{theorem}

\begin{proof}
Proposition~\ref{prop:PhiInfl} implies that $\mathrm{infl}\colon G_0(\cal{A}^{\Bbbk}_{n,k}) \to G_0(\B^{\Bbbk}_l(n,k))$, viewed a map from $K_0^{\C(q)}(\cal{A}^{\Bbbk}_{n,k})$ to $K_0^{\C(q)}(\B^{\Bbbk}_l(n,k))$, is equal to $(k)^!_{q^2}(1-q^2)^k$ times the isomorphism $\pr^{-1}\colon K_0^{\C(q)}(\cal{A}^{\Bbbk}_{n,k}) \to K_0^{\C(q)}(\B^{\Bbbk}_l(n,k))$ that we have chosen. Thus, inflating a Sartori module and using Definition~\ref{def:OSzVnIdent} to get a class in $\left( V^{\otimes n} \right)_k$ amounts to using Theorem~\ref{thm:SartoriG0Ident} to get a class in $\left( V^{\otimes n} \right)_k$ directly, then multiplying the result by $(k)^!_{q^2}(1-q^2)^k$. The claim follows from Theorem~\ref{thm:SartoriG0Ident} plus the fact that multiplying the Sartori dual standard and dual canonical bases by $(k)^!_{q^2}(1-q^2)^k$ gives the Ozsv{\'a}th--Szab{\'o} dual standard and dual canonical bases.
\end{proof}

% ------------------------------------------------------
\subsection{Compact derived categories}
% ------------------------------------------------------

As discussed in \cite[Section 5.1]{KellerTilting}, the homotopy category of bounded complexes of finitely generated projective (graded) $\B_l^{\Bbbk}(n,k)$-modules $H^b(\B_l^{\Bbbk}(n,k){\rm-proj})$ is equivalent to the compact derived category $\cal{D}^c(\B_l^{\Bbbk}(n,k))$. We have
\[
K(H^b(\B_l^{\Bbbk}(n,k){\rm-proj})) \cong K_0(\B_l^{\Bbbk}(n,k))
\]
and thus
\[
K_0(\B^{\Bbbk}_l(n,k)) \cong K(\cal{D}^c(\B^{\Bbbk}_l(n,k))).
\]
Passing to $\C(q)$, we can use Definition~\ref{def:OSzVnIdent} to identify $K^{\C(q)}(\cal{D}^c(\B^{\Bbbk}_l(n,k)))$ with $\left(V^{\otimes n}\right)_k$ (we could do the same with the Sartori algebra).

\begin{corollary}
Under the above identification, we have classes in $K^{\C(q)}(\cal{D}^c(\B^{\Bbbk}_l(n,k)))$ categorifying
\begin{itemize}
\item the canonical basis,
\item $(k)^!_{q^2}(1-q^2)^k$ times the canonical basis,
\item $(k)^!_{q^2}!(1-q^2)^k$ times the standard basis,
\item the Ozsv{\'a}th--Szab{\'o} dual standard basis, and
\item the Ozsv{\'a}th--Szab{\'o} dual canonical basis
\end{itemize}
of $\left( V^{\otimes n} \right)_k$.
\end{corollary}

% ----------------------------------------------------
\subsection{\texorpdfstring{Comparison with the conventions of \cite{ManionDecat}}{Comparison of conventions}}\label{sec:ConventionComparison}
% ----------------------------------------------------

In \cite{OSzNew,OSzNewer}, Ozsv{\'a}th--Szab{\'o} give $\B_l(n,k)$ different gradings based on a choice of orientations for $n$ points. Our quotient map from Ozsv{\'a}th--Szab{\'o}'s algebra to Sartori's algebra is a degree-zero map when Ozsv{\'a}th--Szab{\'o}'s algebra is given the gradings for all $n$ points oriented negatively.

In \cite{ManionDecat}, Ozsv{\'a}th--Szab{\'o}'s algebra with these gradings was used to categorify tensor powers of $V^*$, rather than of $V$. Since Sartori uses his algebra to categorify tensor powers of $V$, our conventions cannot match those of \cite{ManionDecat} exactly.

One way to relate the conventions is as follows. The ``right modified basis'' for $(V^*)^{\otimes n}$ defined in \cite{ManionDecat} can be described (up to a power of $q$ that we will change for convenience) by letting $w_i$ be the standard basis element with $v_0^*$ in position $i$ and $v_1^*$ in all other positions; then $\ell'_i := w_i + qw_{i+1}$ for $1 \leq i \leq n-1$, while $\ell'_n := w_n$. Wedge products of the elements $\ell'_i$ (taken with $i$ in increasing order in \cite{ManionDecat}) form the right modified basis for $(V^*)^{\otimes n}$.

As vector spaces, identify $V^{\otimes n}$ with $(V^*)^{\otimes n}$ by sending the standard basis element $v_{j_1} \otimes \cdots \otimes v_{j_n}$ to the dual basis element $v_{j_n}^* \otimes \cdots \otimes v_{j_1}^*$ where $j_i \in \{0,1\}$. Then $\ell'_n$ gets sent to $\ell_1 = e_1$ while $\ell'_i$ gets sent to $\ell_{n+1-i} = e_{n+1-i} + qe_{n-i}$ for $1 \leq i \leq n-1$. More generally, wedge products of the $\ell'_i$ are sent to the canonical basis for $V^{\otimes n}$. One can check that this identification intertwines the braiding on $V$ with the braiding on $V^*$.

In \cite{ManionDecat}, for an I-state $\x$ with $0 \notin \x$ (a right I-state), the element $[P(\x)]$ of $K_0(\B^{\Bbbk}_r(n,k))$ was identified with the right modified basis element $\ell'_{x_1 + 1} \wedge \cdots \wedge \ell'_{x_k + 1}$ of $(V^*)^{\otimes n}$, where $\B^{\Bbbk}_r(n,k)$ is defined as in Remark~\ref{rem:OSzAlgSymms}. Translating to an element of $V^{\otimes n}$ as in the above paragraph, we get $\ell_{n-x_1} \wedge \cdots \wedge \ell_{n-x_k}$. This is the canonical basis element $v_{\x'}^{\diamondsuit}$ associated to the left I-state $\cal{R}(\x)$ (and thus $[P(\cal{R}(\x))]$) in Definition~\ref{def:OSzVnIdent}, where $\cal{R}$ is the Ozsv{\'a}th--Szab{\'o} symmetry mentioned in Remark~\ref{rem:OSzAlgSymms} and $\cal{R}(\x) = \{n-x_i \mid 1 \leq i \leq k\}$. It now follows from \cite[Lemma 10.1]{OSzNew}, \cite[Theorem 1.4.2]{ManionDecat}, and the previous paragraph that under the conventions of Definition~\ref{def:OSzVnIdent}, Ozsv{\'a}th--Szab{\'o}'s positive-crossing bimodule $\cal{P}_i$ over $\B_l(n,k)$ categorifies the braiding acting on factors $(i,i+1)$ of $V^{\otimes n}$ for $1 \leq i \leq n-1$. Similarly, Ozsv{\'a}th--Szab{\'o}'s negative-crossing bimodule $\cal{N}_i$ categorifies the inverse of the braiding acting on factors $(i,i+1)$.

% ====================================================
\section{Bimodules for quantum group generators} \label{subsec:EF}
% ====================================================

% ----------------------------------------------------
\subsection{The Sartori \texorpdfstring{$\cal{F}$}{F} functor} \label{subsec:SartF}
% ----------------------------------------------------

Let $\Gamma_k^{\vee}$ denote the set of $\vee\wedge$-sequences in $D_{n,k}$ whose leftmost symbol is a $\vee$, and let $\Gamma_k^{\wedge}$ be the set of $\lambda \in D_{n,k}$ whose leftmost symbol is a $\wedge$. Define idempotents
\begin{equation}
  e_k^{\vee} = \sum_{\l \in \Gamma_k^{\vee}} \1_{\l}, \qquad e_k^{\wedge} = \sum_{\l \in \Gamma_k^{\wedge}} \1_{\l}.
\end{equation}

For $\l \in \Gamma_k^{\vee}$, set $\l^{(\wedge)} \in \Gamma_{k+1}^{\wedge}$ to be the sequence obtained from $\l$ by swapping the lead term from $\vee$ to $\wedge$.  Similarly, for $\mu \in \Gamma_{k+1}^{\wedge}$ define $\mu^{\vee} \in \Gamma_k^{\vee}$ by swapping the first symbol from $\wedge$ to $\vee$.  This operation defines a bijection $\Gamma_k^{\vee} \to \Gamma_{k+1}^{\wedge}$.

For any $\l,\mu \in \Gamma_{k+1}^{\wedge}$, there is a natural surjective map $\1_{\mu} \cal{A}^{\Z}_{n,k+1} \1_{\l} \longrightarrow \1_{\mu^{(\vee)}} \cal{A}^{\Z}_{n,k} 1_{\l^{(\vee)}} $.  We thus get a surjective algebra homomorphism
\begin{equation}
  \Psi \maps e_{k+1}^{\wedge}\cal{A}^{\Z}_{n,k+1} e_{k+1}^{\wedge} \longrightarrow e_{k}^{\vee}\cal{A}^{\Z}_{n,k}e_{k}^{\vee}
\end{equation}
and thereby a well-defined surjective homomorphism (see \cite[Proposition 5.36]{Sar-diagrams})
\begin{align} \label{eq:surjective}
  \cal{A}^{\Z}_{n,k+1}/ \cal{A}^{\Z}_{n,k+1} e_{k+1}^{\vee} \cal{A}^{\Z}_{n,k+1} &\to e_k^{\vee} \cal{A}^{\Z}_{n,k} e_k^{\vee} \nn \\
 [a]\qquad  & \mapsto \Psi(e_{k+1}^{\wedge} a e_{k+1}^{\wedge}).
\end{align}

Consider the projective module $P_k^{\vee}:= \cal{A}^{\Z}_{n,k}  e_k^{\vee}$.  Sartori shows in \cite[Section 5.5]{Sar-diagrams} that $P_k^{\vee}$ is the sum of all the indecomposable projective-injective left $\cal{A}^{\Z}_{n,k}$-modules.  The left $\cal{A}^{\Z}_{n,k}$-module $P_k^{\vee}$ has a right $\cal{A}^{\Z}_{n,k+1}$-module structure induced by the map
\begin{align}
  \cal{A}^{\Z}_{n,k+1} \longrightarrow \cal{A}^{\Z}_{n,k+1} / \cal{A}^{\Z}_{n,k+1}e_{k+1}^{\vee} \cal{A}^{\Z}_{n,k+1} \longrightarrow e_k^{\vee}\cal{A}^{\Z}_{n,k} e_k^{\vee}
\end{align}
where the first arrow is the quotient map and the second is the surjective map \eqref{eq:surjective}.  This gives $P_k^{\vee}$ the structure of an $(\cal{A}^{\Z}_{n,k},\cal{A}^{\Z}_{n,k+1})$-bimodule; call this bimodule $\mathbf{F}_k = \mathbf{F}_k^S$. One can define a right-exact functor
\begin{equation}
  \cal{F}_k = \cal{F}_k^S\maps \cal{A}^{\Z}_{n,k+1}{\rm -gmod}
 \xy (-10,0)*+{}="l"; (10,0)*+{}="r";
 {\ar^{\mathbf{F}_k \otimes_{\cal{A}^{\Z}_{n,k+1} } \cdot} "l";"r"}; \endxy \cal{A}_{n,k}{\rm -gmod.}
\end{equation}
We have $\cal{F}_{k-1} \circ \cal{F}_k = 0$. Applying $\cal{F}_k$ to an indecomposable projective $P(\mu) = \cal{A}^{\Z}_{n,k+1}\1_{\mu}$ gives either an indecomposable projective or zero:
\begin{equation}\label{eq:SartoriFOnProjectives}
  \cal{F}_k (P(\mu)) := \mathbf{F}_k \otimes_{\cal{A}^{\Z}_{n,k+1}}\cal{A}^{\Z}_{n,k+1}\1_{\mu}
=
\left\{
  \begin{array}{ll}
    \cal{A}^{\Z}_{n,k}\1_{\l}, & \hbox{if $\l^{(\wedge)} = \mu$ for some $\l \in \Gamma_k^{\vee}$;} \\
    0, & \hbox{otherwise.}
  \end{array}
\right.
\end{equation}

Sartori views $\cal{F}$ as inducing a map on a topological Grothendieck group of a derived analogue of $G_0(\cal{A}^{\Bbbk}_{n,k})$. On $K_0(\cal{A}^{\Bbbk}_{n,k})$, derived categories and topological completions are not required for $\cal{F}$ to induce a map; any additive functor between categories of finitely generated projective modules induces a map on $K_0(\cal{A}^{\Bbbk}_{n,k})$.

\begin{corollary}[cf. Proposition 5.8 of \cite{Sar-tensor} and equation \eqref{eq:SartoriFOnProjectives} above]
The map from $K_0(\cal{A}^{\Bbbk}_{n,k+1})$ to $K_0(\cal{A}^{\Bbbk}_{n,k})$ induced by $\cal{F}_k^S$ agrees with the map $F\colon \left(V^{\otimes n} \right)_{k+1} \to \left(V^{\otimes n} \right)_k$ under the identification of Definition~\ref{def:SarVnIdent}.
\end{corollary}

% ----------------------------------------------------
\subsection{The Ozsv{\'a}th--Szab{\'o} \texorpdfstring{$\cal{F}$}{F} functor}
% ----------------------------------------------------

As done above for the Sartori algebras, define idempotents in $\B_l(n,k)$ by
\begin{equation}
  e_k^{\vee} = \sum_{\x \in V_l(n,k) \, : \, 0 \notin \x} \Ib_{\x}, \qquad e_k^{\wedge} = \sum_{\x \in V_l(n,k) \, : \, 0 \in \x} \Ib_{\x}.
\end{equation}
For $\x \in V_l(n,k)$ with $0 \notin \x$, let $\x^{(\wedge)} = \x \cup \{0\}$; for $\x \in V_l(n,k)$ with $0 \in \x$, let $\x^{(\vee)} = \x \setminus \{0\}$. If $0 \in \x \cap \y$, the structure of generating intervals gives us a natural surjective map
\[
\Ib_{\x} \B_l(n,k+1) \Ib_{\y} \to \Ib_{\x^{(\vee)}} \B_l(n,k) \Ib_{\y^{(\vee)}}
\]
giving us a surjective ring homomorphism
\[
\Psi'\colon e_{k+1}^{\wedge} \B_l(n,k+1) e_{k+1}^{\wedge} \to e_k^{\vee} \B_l(n,k) e_k^{\vee}
\]
Thus, analogous to \cite[Prop 5.36]{Sar-diagrams}, we have a well defined surjective map
\begin{align*}
\B_l(n,k+1) / \B_l(n,k+1) e_{k+1}^{\vee} \B_l(n,k+1) &\to e_k^{\vee} \B_l(n,k) e_k^{\vee} \\
[b] \qquad &\mapsto \Psi'(e_{k+1}^{\wedge} b e_{k+1}^{\wedge}).
\end{align*}
Let $P_k^{\vee} = \B_l(n,k) e_k^{\vee}$. As in the Sartori case, the above homomorphism gives $P_k^{\vee}$ the structure of a right module over $\B_l(n,k+1)$; thus, $P_k^{\vee}$ is a bimodule over $(\B_l(n,k),\B_l(n,k+1))$. Call this bimodule $\mathbf{F}_k = \mathbf{F}_k^{OSz}$. We define
\[
\cal{F}_k = \cal{F}_k^{OSz} \colon \B_l(n,k+1){\rm-proj} \to \B_l(n,k){\rm-proj}
\]
to be the tensor product with $\mathbf{F}_k^{OSz}$. We have $\cal{F}_{k-1} \circ \cal{F}_k = 0$ and
\[
\cal{F}_k(P(\x)) = \begin{cases} P(\x \setminus \{0\}) & 0 \in \x \\ 0 & {\rm otherwise. } \end{cases}
\]

\begin{theorem}\label{thm:OSzFCat}
The map from $K_0(\B_l^{\Bbbk}(n,k+1))$ to $K_0(\B_l^{\Bbbk}(n,k))$ induced by $\cal{F}_k^{OSz}$ agrees with the map $F\colon \left(V^{\otimes n} \right)_{k+1} \to \left(V^{\otimes n} \right)_k$ under the identification of Definition~\ref{def:OSzVnIdent}.
\end{theorem}

\begin{proof}
The result follows from \cite[Proposition 5.8]{Sar-tensor} and Definition~\ref{def:OSzVnIdent}.
\end{proof}

% ----------------------------------------------------
\subsection{Comparing the Ozsv{\'a}th--Szab{\'o} and Sartori functors}
% ----------------------------------------------------

As above, let $\mathbf{F}^S_k$ and $\mathbf{F}^{OSz}_k$ denote the bimodules giving rise to the Sartori and Ozsv{\'a}th--Szab{\'o} functors $\cal{F}_k^{S}$ and $\cal{F}_k^{OSz}$. We write ${_{\pr}}(\mathbf{F}^{OSz}_k)$ and $(\mathbf{F}_k^S)_{\infl}$ for the bimodules over $(\cal{A}^{\Z}_{n,k},\B_l(n,k+1))$ obtained by projecting the left action of $\mathbf{F}^{OSz}_k$ and inflating the right action of $\mathbf{F}^S_k$ respectively.

\begin{theorem}\label{thm:SarOSzFRel}
The bimodules ${_{\pr}}(\mathbf{F}^{OSz}_k)$ and $(\mathbf{F}_k^S)_{\infl}$ are isomorphic.
\end{theorem}

\begin{proof}
We just need to show that the right actions agree, which follows from commutativity of the square
\[
\xymatrix{
\B_l(n,k+1) / \B_l(n,k+1) e_{k+1}^{\vee} \B_l(n,k+1) \ar[r]^-{\Psi'} \ar[d]_{\Xi} & e_k^{\vee} \B_l(n,k) e_k^{\vee} \ar[d]^{\Xi} \\
\cal{A}^{\Z}_{n,k+1} / \cal{A}^{\Z}_{n,k+1} e_{k+1}^{\vee} \cal{A}^{\Z}_{n,k+1} \ar[r]_-{\Psi} & e_k^{\vee} \cal{A}^{\Z}_{n,k} e_k^{\vee}
}.
\]
To see that the square commutes, note that the generator $f_{\x,\y}$ at the top left (with $0 \in \x \cap \y$) gets sent by the left edge to the basis element of the minimal oriented fork diagram between $\x$ and $\y$, which starts with at least one $\wedge$-labeled ray. The bottom edge ``forks together'' all these $\wedge$-labeled rays and sends this basis element to the basis element of the minimal oriented fork diagram on the bottom right, which has one extra $\vee$ at the left of the new fork. On the other hand, the top edge sends $f_{\x,\y}$ to $f_{\x^{(\vee)}, \y^{(\vee)}}$, whose associated unoriented fork diagram also ``forks together'' the initial sequence of $\wedge$-labeled rays in the fork diagram for $f_{\x,\y}$, and the right edge sends $f_{\x^{(\vee)}, \y^{(\vee)}}$ to the basis element of the minimal oriented fork diagram for this unoriented fork diagram. The result now follows from $\Z[U_1,\ldots,U_n]$-linearity of all the edges.
\end{proof}

It follows that the functors $\cal{F}^S_k$ and $\cal{F}^{OSz}_k$ intertwine the projection functors from $\B_l(n,k+1)\rm{-proj}$ to $\cal{A}^{\Z}_{n,k+1}\rm{-proj}$ and from $\B_l(n,k)\rm{-proj}$ to $\cal{A}^{\Z}_{n,k}\rm{-proj}$.

% ----------------------------------------------------
\subsection{Duals of the \texorpdfstring{$\cal{F}$}{F} functors}
% ----------------------------------------------------

Let
\[
\mathbf{E}'_k = (\mathbf{E}')_k^{S} := {^{\vee}}\mathbf{F}^S_k = \Hom_{\cal{A}^{\Z}_{n,k}}(\mathbf{F}^S_k,\cal{A}^{\Z}_{n,k})
\]
(see \cite[Section 5.6]{Sar-diagrams}) and
\[
\mathbf{E}''_k = (\mathbf{E}'')_k^{OSz} := {^{\vee}}\mathbf{F}^{OSz}_k = \Hom_{\B_l(n,k)}(\mathbf{F}^{OSz}_k,\B_l(n,k)).
\]

As before, these bimodules square to zero. The functors
\[
\cal{E}'_k := \mathbf{E}'_k \otimes - \colon \cal{A}^{\Z}_{n,k}{\rm-fmod} \to \cal{A}^{\Z}_{n,k+1}{\rm-fmod}
\]
and
\[
\cal{E}''_k := \mathbf{E}''_k \otimes - \colon \B_l(n,k){\rm-fmod} \to \B_l(n,k+1){\rm-fmod}
\]
are exact since  $\mathbf{E}'_k$ and $\mathbf{E}''_k$  are projective as right modules, so they induce maps on $G_0$ after tensoring with $\Bbbk$.

The matrix for $G_0(\cal{E}'_k)$ in the basis of simples is the transpose of the matrix for $K_0(\cal{F}^S_k)$ in the basis of projectives. Identifying $G_0^{\C(q)}(\cal{A}^{\Bbbk}_{n,k})$ with $\left(V^{\otimes n}\right)_k$ via $K_0^{\C(q)}(\cal{A}^{\Bbbk}_{n,k})$ as in Definition~\ref{def:SarVnIdent}, $G_0(\cal{E}'_k)$ sends Sartori dual canonical basis elements to Sartori dual canonical basis elements or zero. Analogous claims hold for $\cal{E}''_k$ in the Ozsv{\'a}th--Szab{\'o} setting..

\begin{proposition}[cf. Theorem 7.19 of \cite{Sar-tensor}]\label{prop:SarFDualCat}
The map from $\left( V^{\otimes n} \right)_k$ to $\left( V^{\otimes n} \right)_{k+1}$ induced by the Sartori functor $\cal{E}'_k$ agrees with the action of the quantum group element\footnote{Sartori defines $E'$ in an arbitrary weight space by
\[
E  = q \frac{(1-q^{2K_1})}{(1-q^2)}E'K^{-1},
\]
which can be interpreted as defining $E'$ in the modified (or idempotent form) $\dot{U}_q(\mf{gl}(1|1))$ of $U_q(\mf{gl}(1|1))$ defined in \cite[Definition 3.2]{TVW} by
\[
E'1_{(a,b)} = q^{a+b-1}/(a+1)_{q^2}E1_{(a,b)}.
\]
}
\begin{equation}\label{eq:Eprime}
E' := \frac{q^{n-1}}{(k+1)_{q^2}}  E = q^{-1}\frac{1}{(k+1)_{q^2}} EK.
\end{equation}
\end{proposition}

Below we describe the decategorification of the Ozsv{\'a}th--Szab{\'o} functor $\cal{E}''_k$, which is similar.
\begin{theorem}\label{thm:OSzECat}
The map from $\left( V^{\otimes n} \right)_k$ to $\left( V^{\otimes n} \right)_{k+1}$ induced by $\cal{E}''_k$ agrees with the action of the quantum group element
\[
E'' = q^{-1} (1-q^2) EK = (q^{-1} - q)EK.
\]
\end{theorem}

\begin{proof}
We claim that the map $[\cal{E}''_k]$ induced by $\cal{E}''_k$ equals $1-q^{2(k+1)}$ times the map $[\cal{E}'_k]$ induced by $\cal{E}'_k$; the result then follows from Proposition~\ref{prop:SarFDualCat}. Indeed, $[\cal{E}''_k]$ acts on Ozsv{\'a}th--Szab{\'o} dual canonical basis elements the way $[\cal{E}'_k]$ acts on Sartori dual canonical basis elements. For $\x \in V_l(n,k)$ we have $v_{\x}^{\heartsuit \heartsuit} = (k)^!_{q^2}(1-q^2)^k v_{\x}^{\heartsuit}$, and for $\y \in V_l(n,k+1)$ we have $v_{\y}^{\heartsuit \heartsuit} = (k+1)^!_{q^2}(1-q^2)^{k+1} v_{\y}^{\heartsuit}$. Thus, $[\cal{E}''_k]$ sends a Sartori dual canonical basis element for $\left( V^{\otimes n} \right)_k$ to
\[
\frac{(k+1)^!_{q^2}(1-q^2)^{k+1}}{(k)^!_{q^2}(1-q^2)^k} = (k+1)_{q^2}(1-q^2) = 1-q^{2(k+1)}
\]
times where $[\cal{E}'_k]$ sends this basis element.
\end{proof}

\begin{remark}
The generators $E''$ and $F$ satisfy the anticommutation relation $E''F + FE'' = 1-K^2$; compare with \cite[Section 1.2]{TianUT}. We set $T = K^2$, a slight change of conventions from what Tian writes.
\end{remark}

Define ${_{\infl}}(\mathbf{E}'_k)$ and $(\mathbf{E}''_k)_{\pr}$ by inflating the left action on $\mathbf{E}'_k$ and projecting the right action on $\mathbf{E}''_k$.

\begin{theorem}\label{thm:SarOSzERel}
The bimodules ${_{\infl}}(\mathbf{E}'_k)$ and $(\mathbf{E}''_k)_{\pr}$ over $(\B_l(n,k+1),\cal{A}^{\Z}_{n,k})$ are isomorphic.
\end{theorem}

\begin{proof}
This result is a consequence of Theorem~\ref{thm:SarOSzFRel} because ${_{\infl}}(\mathbf{E}'_k) = {_{\infl}}({^{\vee}}\mathbf{F}_k^S) \cong {^{\vee}}((\mathbf{F}_k^S)_{\infl})$ and $(\mathbf{E}''_k)_{\pr} = ({^{\vee}}\mathbf{F}_k^{OSz})_{\pr} = {^{\vee}}({_{\pr}}(\mathbf{F}_k^{OSz}))$.
\end{proof}

It follows that the functors $\cal{E}'_k$ and $\cal{E}''_k$ intertwine the inflation functors from $\cal{A}^{\Z}_{n,k+1}\rm{-fmod}$ to $\B_l(n,k+1)\rm{-fmod}$ and from $\cal{A}^{\Z}_{n,k}\rm{-fmod}$ to $\B_l(n,k)\rm{-fmod}$.

\bibliographystyle{alpha}
\bibliography{bib_clean}

\end{document}